\documentclass[11pt,leqno]{article}

\usepackage{amsthm,amsfonts,amssymb,amsmath,oldgerm}
\usepackage{fullpage}
\usepackage{graphicx}
\usepackage{mathrsfs}
\usepackage{xcolor}
\usepackage[colorinlistoftodos,disable]{todonotes}
\usepackage[hidelinks]{hyperref}
\hypersetup{
    colorlinks=false,
    pdfborder={0 0 0},
   pdfauthor={J. Alexopoulos, B. de Rijk},
   pdftitle={Nonlinear dynamics of reaction-diffusion wave trains under large and fully nonlocalized modulations}
}
\usepackage{enumitem}

\makeatletter
\def\namedlabel#1#2{\begingroup
    #2%
    \def\@currentlabel{#2}%
    \phantomsection\label{#1}\endgroup
}
\makeatother
\usepackage{cite}

\numberwithin{equation}{section}

\newcommand{\Z}{\mathbb Z}
\newcommand{\R}{\mathbb R}
\newcommand{\C}{\mathbb C}

\newcommand{\NM}{{\mathbb{N}}}

\newcommand{\vt}{\mathring{v}}
\newcommand{\ri}{\mathrm{i}}
\newcommand{\re}{\mathrm{e}}
\newcommand{\de}{\mathrm{d}}
\newcommand{\zt}{\mathring{z}}
\newcommand{\wt}{\mathring{w}}
\newcommand{\gt}{\boldsymbol{\gamma}}
\renewcommand{\Re}{\mathrm{Re}}

\definecolor{purpp}{RGB}{188,0,237}	
	\definecolor{greeen}{RGB}{0,128,0}	
\newtheorem{theorem}{Theorem}[section]
\newtheorem{proposition}[theorem]{Proposition}
\newtheorem{corollary}[theorem]{Corollary}
\newtheorem{lemma}[theorem]{Lemma}
\newtheorem{remark}[theorem]{Remark}
\theoremstyle{definition}

\allowdisplaybreaks[3]

\newcommand{\xx}{{\zeta}}
\newcommand{\xt}{{\bar{\zeta}}}

\newcommand{\Non}{\mathcal{N}}
\newcommand{\El}{\mathcal{L}}

\usepackage{caption}
\title{Nonlinear dynamics of reaction-diffusion wave trains under large and fully nonlocalized modulations}
\author{Joannis Alexopoulos$^*$, and Bj\"orn de Rijk\thanks{Department of Mathematics, Karlsruhe Institute of Technology, Englerstra\ss e 2, 76131 Karlsruhe, Germany; \texttt{joannis.alexopoulos@kit.edu}, \texttt{bjoern.de-rijk@kit.edu}}}

\begin{document}
\maketitle
\begin{abstract}
We study the dynamics of periodic wave trains in reaction-diffusion systems on the real line under large, fully nonlocalized modulations. We prove that solutions with nearby initial data converge, at an enhanced diffusive rate, to a modulated wave train whose leading-order phase and wavenumber dynamics are governed by an explicit solution to the viscous Hamilton-Jacobi equation. This constitutes a global stability result: such initial data are generally not close to the large-time modulated wave train. In contrast to previous modulational stability results, our analysis does not require that the initial data approach phase shifts of the wave train at spatial infinity. The central methodological advance is a nontrivial extension of the recently developed $L^\infty$-stability theory to accommodate large phase modulations. This framework, based entirely on $L^\infty$-estimates, removes all localization requirements as imposed in the previous literature, allowing us to treat the full range of bounded modulational initial data under minimal regularity assumptions. The main technical contributions include: the strategic use of interpolation inequalities to balance smallness and temporal decay, and a detailed analysis of the linear dynamics under fully nonlocalized modulational data.
\end{abstract}

\section{Introduction}

In this paper, we study the dynamics of modulated wave trains in the reaction-diffusion system
\begin{equation}\label{RD0}
\partial_t u = Du_{xx} + f(u), \qquad x \in\R,\, t \geq 0,\, u\in\R^n,
\end{equation}
where $n \in \mathbb N$, $D \in \R^{n \times n}$ is a symmetric, positive-definite matrix, and $f \colon \R^n \to \R^n$ is a smooth nonlinearity. Wave trains are solutions to~\eqref{RD0} of the form $u_{\mathrm{wt}}(x,t) = \phi_0(k_0 x -\omega_0 t)$ with wavenumber $k_0 \in \R \setminus \{0\}$, temporal frequency $\omega_0 \in \R$, propagation speed $c = \omega_0/k_0$, and $1$-periodic profile function $\phi_0(\zeta)$. These periodic traveling waves represent the most fundamental patterns that arise at the onset of a Turing instability and as such play a fundamental role in pattern-forming processes in biology, chemistry, ecology and more; we refer to~\cite{DOE19} for further details and references. Under generic conditions, the wave train $u_{\mathrm{wt}}(x,t)$ can be continued in the wavenumber, giving rise to a family of wave trains $u_k(x,t) =  \phi(k x - \omega(k) t; k)$, with $1$-periodic profile function $\phi(\cdot;k)$ and temporal frequency $\omega(k)$, defined for $k$ near $k_0$ and satisfying $\omega(k_0) = \omega_0$ and $\phi(\cdot;k_0) = \phi_0$. Here, the function $\omega(k)$, which expresses the frequency in terms of the wavenumber, is referred to as the \emph{nonlinear dispersion relation}; see~\cite{DSSS}.

We are interested in the dynamics of solutions whose initial data are close to a \emph{modulated wave train}
\begin{align*}
u_{\mathrm{mod}}(x) = \phi_0\left(k_0 x + \gt_0(x)\right),
\end{align*}
where $\gt_0 \colon \R \to \R$ denotes a phase modulation. Specifically, we consider initial data of the form
\begin{align} \label{e:ICintro}
u(x,0) = \phi_0\left(k_0 x + \gt_0(x)\right) + \wt_0(x),
\end{align}
where $\wt_0 \colon \R \to \R^n$ is a small perturbation. A natural question is whether the solution $u(x,t)$ to~\eqref{RD0} with initial condition~\eqref{e:ICintro} remains close to a modulated wave train for all $t \geq 0$. That is, under what conditions does there exist a modulation function $\gt \colon \R \times [0,\infty) \to \R$ such that the perturbation
\begin{align} \label{e:asympdescr0}
\wt(x,t) = u(x,t) - \phi_0\left(k_0 x - \omega_0 t + \gt(x,t)\right)
\end{align}
remains small for all time? In this paper, we establish such a \emph{modulational stability} result under standard spectral stability assumptions on the wave train, together with the requirement that $\gt_0$ is bounded and both $\gt_0'$ and $\wt_0$ are uniformly continuous and small in $L^\infty(\R)$. In particular, for any bounded function $g_0 \colon \R \to \R$ with uniformly continuous derivative, there exists $\delta > 0$ such that the phase modulation $\gt_0(x) = g_0(\delta x)$ is admissible.

Our result advances the modulational stability theory for wave-train solutions to~\eqref{RD0} in two crucial ways. First, in contrast to all prior results~\cite{IYSA,JUNNL,JONZNL,JONZW,SAN3}, we do not require $\gt_0'$ to be $L^1$-localized and, in particular, do not assume the existence of asymptotic phase limits $\gt_\pm = \lim_{x \to \pm \infty} \gt_0(x)$; see Remark~\ref{rem:L1} and Figure~\ref{fig_nonlocalized_derivative}. Second, we remove all localization assumptions on both $\gt_0'$ and the initial perturbation $\wt_0$ that are present in the previous literature; see Remark~\ref{rem:literature}. This allows us to handle a significantly larger class of modulational initial data of the form~\eqref{e:ICintro} than considered before.

Notably, we do not require $\gt_0$ to be small in $L^\infty(\mathbb{R})$. To the best of our knowledge, the only other work addressing \emph{large} phase modulations of wave-train solutions to~\eqref{RD0} (while still assuming $\gt_0 \in L^1(\mathbb{R})$) is~\cite{IYSA}. In this context, we also highlight the recently developed stability theory for periodic patterns in \emph{planar} reaction-diffusion systems~\cite{melinand2024}. These patterns are intrinsically planar in the sense that they are nontrivial in \emph{all} spatial directions. Building on techniques from~\cite{JONZNL,JONZW}, the analysis in~\cite{melinand2024} still relies on decay induced by localization, yet it allows for large, and even unbounded, phase modulations. The reason is that localization leads to stronger temporal decay in higher spatial dimensions, permitting a relaxation of the localization requirement on the initial phase modulation: instead of assuming $\gt_0' \in L^1(\mathbb{R})$ as in~\cite{JONZNL,JONZW}, it suffices to require that $\Delta \gt_0 \in L^1(\mathbb{R}^2)$. This condition imposes only that $\gt_0$ has bounded mean oscillation (up to additive constants), allowing for unbounded $\gt_0$; see also~\S\ref{sec:unbounded}.

Our result is an \emph{asymptotic} modulational stability result in the sense that the perturbation $\wt(t)$ in~\eqref{e:asympdescr0} decays over time at a diffusive rate $\smash{t^{-\frac12}}$. Moreover, the modulation $\gt(t)$, which governs the phase dynamics in~\eqref{e:asympdescr0}, remains bounded with small derivative, which also decays at rate $\smash{t^{-\frac12}}$. As in~\cite{IYSA,JONZW,SAN3}, we show that $\gt(t)$ is well-approximated by the solution $\breve{\gt}(t)$ to the viscous Hamilton-Jacobi equation
\begin{align} \label{e:HamJacIntro}
\partial_t \breve{\gt} = \frac{d}{k_0^2} \, \breve{\gt}_{xx} + \omega'(k_0)\, \breve{\gt}_x - \frac{1}{2} \omega''(k_0)\, \breve{\gt}_x^2,
\end{align}
with initial condition $\breve{\gt}(0) = \gt_0$. Here, the coefficient $d > 0$ is an explicit Melnikov-type integral; see~\eqref{e:defad}. 

The first rigorous justification of equation~\eqref{e:HamJacIntro} as an effective description of the dynamics of slowly modulated wave trains on long time intervals was provided in~\cite{DSSS}. Our result, along with those in~\cite{IYSA,JONZW,SAN3}, confirms that this description remains valid globally.

It is not difficult to show that the solution $\breve{\gt}(t)$ to~\eqref{e:HamJacIntro} does not, in general, remain close to its initial condition $\gt_0$, and the same holds for $\gt(t)$ by approximation; see Remark~\ref{rem:global}. Thus, our modulational stability result is \emph{global} in the sense that it provides a complete description of the dynamics of solutions $u(t)$ with initial data of the form~\eqref{e:ICintro}, even though $u(t)$ typically does not remain close to $u(0)$ over time. 

It was already observed in~\cite{DSSS} that phase and wavenumber modulations are intimately connected. For the modulated wave train described in~\eqref{e:asympdescr0}, the local wavenumber, i.e. the number of waves per unit interval near position $x$, is given by $k_0 + \gt_x(x,t)$. This suggests that incorporating a modulation of the wavenumber may yield a more accurate approximation of the solution $u(t)$. We show that the refined residual
\begin{align} \label{e:asympdescr}
\mathring{y}(x,t) := u(x,t) - \phi\left(k_0 x - \omega_0 t + \gt(x,t)\left(1 + \tfrac{1}{k_0} \gt_x(x,t)\right);\, k_0 + \gt_x(x,t)\right)
\end{align}
indeed decays in $L^\infty(\R)$ at an enhanced rate $t^{-1}\log(1+t)$. This refined asymptotic description of the dynamics of $u(t)$, accounting for both phase and wavenumber modulation, was previously established in~\cite{JONZW}, but only for \emph{small} initial phase modulations $\gt_0$ and with weaker decay rates in $L^\infty(\R)$; see Table~\ref{Tabelle}. Here, we obtain it for the first time in the setting of large initial phase modulations.

While previous modulational stability results rely on localization-induced decay, employing renormalization group techniques~\cite{IYSA,SAN3}, iterative $L^1$-$H^k$ estimates~\cite{JONZNL,JONZW}, or pointwise Green's function bounds~\cite{JUNNL}, our analysis follows a different approach. We extend the recently developed $L^\infty$-based stability theory of~\cite{BjoernMod} to accommodate large modulational initial data. This method eliminates the need for any localization assumptions by leveraging decay generated by smoothing properties of the critical part of the semigroup to close the nonlinear argument. Further details on the strategy underlying our nonlinear stability analysis are provided in~\S\ref{sec:strategy}.

\begin{remark} \label{rem:L1}
{
\upshape
The assumption that $\gt_0 \in C^1(\mathbb{R})$ is bounded is equivalent to the condition
\begin{align}
\label{equiv_cond}
\sup_{R \in \mathbb{R}} \left| \int_0^R \gt_0'(x) \, \mathrm{d}x \right| < \infty.
\end{align}
In this work, we thus extend previous modulational stability results~\cite{JONZW,JONZNL,JUNNL,SAN3,IYSA}, which relied on the stronger assumption $\|\gt_0\|_{L^1} = \int_{\mathbb{R}} |\gt_0'(x)| \, \mathrm{d}x < \infty$, to the weaker condition~\eqref{equiv_cond}. However, this assumption still excludes initial data with nonzero asymptotic wavenumber shifts. Whether modulational stability persists under such wavenumber offsets remains an open question; see~\S\ref{sec:unbounded}.
}
\end{remark}

\begin{remark} \label{rem:literature} { \upshape
The modulational stability results~\cite{JONZNL,JONZW,JUNNL,SAN3} consider initial data~\eqref{e:ICintro}, where $\gt_0'$ is small in $L^1(\R)$. This readily implies that $\|\gt_0\|_\infty$ is small and that the asymptotic limits $\lim_{x \to \pm \infty} \gt_0(x) = \gt_\pm$ exist. While~\cite{IYSA} still requires $\gt_0'$ to be $L^1$-localized, the smallness assumption on $\|\gt_0'\|_{L^1}$ is dropped, thereby allowing for large \emph{phase offsets} $\gt_d = \gt_+ - \gt_-$. See Table~\ref{Tabelle} for further details.}
\end{remark}

\begin{table}[h]
\centering
\renewcommand{\arraystretch}{1.4}
  \begin{tabular}{ | c | c | c | c | } 
  \hline

 & $\gt_0$ and $\wt_0$ & decay of $\gt_x(t)$ and $\wt(t)$ & decay of $\mathring{y}(t)$ \\[5pt]

\hline

\cite{SAN3} & $\left\|\rho^2 \wt_0\right\|_{H^2} + \big\|\rho^2 \gt_0'\big\|_{H^2} \ll 1$ &  $t^{-\frac12+\alpha}$ & 

\\[5pt]

\hline

\cite{JONZNL,JONZW} & $\|\wt_0\|_{L^1 \cap H^3} + \|\gt_0'\|_{L^1 \cap H^3} \ll 1$ &  $t^{-\frac12}$   & $t^{-\frac34} \log(1+t)$\\[5pt]

\hline

\cite{JUNNL} & \begin{minipage}{0.4\textwidth} \vspace{0.1cm} \centering $\|\wt_0\|_{H^2} < \infty$ and \\ 
$\big\|\rho^{\frac32} \wt_0\big\|_{\infty} + \big\|\rho^{\frac{3}{2}}\gt_0'\big\|_{W^{1,\infty}} \ll 1$
 \vspace{0.1cm}\end{minipage} &  \begin{minipage}{0.23\textwidth} \centering \vspace{0.1cm} $\left(1+|x-a t| + \sqrt{t}\right)^{-\frac32}$\\
 $+ \, (1+t)^{-\frac12}\re^{-\frac{|x-at|^2}{C t}}$  \vspace{0.1cm} \end{minipage}  & \\[5pt]

\hline

\cite{IYSA} & \begin{minipage}{0.4\textwidth} \vspace{0.1cm} \centering $\big\|\rho^2 \wt_0\big\|_{H^2} + \big\|\rho^2\gt_0'\big\|_{H^2} < \infty$ and \\ $\|\widehat{\gt}_0'(1 + |\cdot|)\|_{L^1} + \|\wt_0\|_{H^2} \ll \gt_d$ \vspace{0.2cm}\end{minipage}  & $t^{-\frac12+\alpha}$  &  \\[5pt]  


\hline
Thm.~\ref{main_theorem} &  $\|\gt_0\|_\infty < M$ and $\|\wt_0\|_{\infty} + \|\gt_0'\|_{\infty} \ll 1$ & $t^{-\frac12}$  & $t^{-1} \log(1+t)$ \\[5pt]
\hline
\end{tabular}
\renewcommand{\arraystretch}{1} 
\caption{This table compares previously established modulational stability results with our main result, Theorem~\ref{main_theorem}, for solutions $u(t)$ to~\eqref{RD0} with initial condition~\eqref{e:ICintro}. The second column presents the assumptions on the initial modulation $\gt_0$ and perturbation $\wt_0$, where $\rho$ is the algebraic weight $\rho(x) = \smash{\sqrt{1 + x^2}}$, $\widehat{\cdot}$ denotes the Fourier transform, and $M > 0$ is an arbitrary constant fixed a priori. The third column contains the pointwise decay rates obtained for $\gt_x(t)$ and $\wt(t)$ in~\eqref{e:asympdescr0}. Here, $C > 0$ and $a \in \R$ are constants, and $\alpha \in (0, \tfrac{1}{2})$ is an arbitrary parameter fixed a priori. The final column presents the pointwise decay rates obtained for the refined residual defined in~\eqref{e:asympdescr}. We note that the inequality $\smash{\|\gt_0'\|_{L^1} \leq \|\rho^{-\frac32}\|_{L^1} \|\rho^{\frac32} \gt_0'\|_\infty}$ and the continuous embedding $H^1(\R) \hookrightarrow L^\infty(\R)$ imply that $\gt_0' \in L^1(\R)$ in~\cite{JONZW,JONZNL,JUNNL,SAN3,IYSA}. }
\label{Tabelle}
\end{table}

\begin{figure}[h]
    \centering
    \vspace{0.3cm}
    \includegraphics[width=0.55\linewidth,trim={0.1cm 0.1cm 0.1cm 0.1cm},clip]{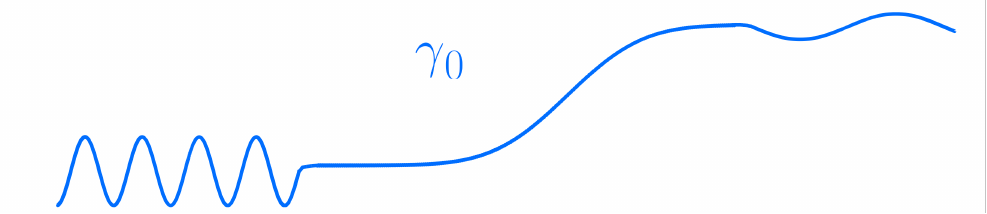}
    \caption{Example of an initial phase modulation $\gt_0 \in C_{\mathrm{ub}}^1(\R)$ with nonlocalized derivative.}
    \label{fig_nonlocalized_derivative}
\end{figure}

\subsection{Statement of main result} 

Before stating the main result, we formulate the necessary hypotheses, which concern the existence and spectral stability of the underlying wave train. We emphasize that these assumptions are standard in the nonlinear stability literature of wave trains; see, for example,~\cite{BjoernMod,IYSA,JONZ,JONZNL,JONZW,JUN,JUNNL,SCH,SAN3,ScheelWu}. We begin with the existence hypothesis:
\begin{itemize}
\item[\namedlabel{assH1}{\upshape (H1)}] There exist a wavenumber $k_0 \in \R \setminus \{0\}$ and a temporal frequency $\omega_0 \in \R$ such that~\eqref{RD0} admits a wave-train solution $u_{\mathrm{wt}}(x,t)=\phi_0(k_0x - \omega_0 t)$, where the profile function $\phi_0 \colon \R \to \R^n$ is nonconstant, smooth and $1$-periodic.
\end{itemize}

Thus, switching to the co-moving frame $\zeta = k_0x - \omega_0 t$,  we find that $\phi_0$ is a stationary solution to 
\begin{align}\label{RD}
\partial_t u = k_0^2Du_{\zeta\zeta} + \omega_0 u_\zeta + f(u).
\end{align}

Next, we turn to the required spectral assumptions. To this end, we consider the linearization of~\eqref{RD} about the wave train $\phi_0$:
\begin{align*}
\El_0 u = k_0^2 D u_{\zeta\zeta} + \omega_0 u_\zeta + f'\big(\phi_0(\zeta)\big) u.
\end{align*}
We view $\El_0$ as a $1$-periodic differential operator acting on the space $C_{\mathrm{ub}}(\R)$, with dense domain $D(\El_0) = C_{\mathrm{ub}}^2(\R)$. Here, $C_{\mathrm{ub}}^m(\R)$, $m \in \mathbb{N}_0$, denotes the Banach space of bounded, uniformly continuous functions that are $m$ times differentiable, with all derivatives up to order $m$ also bounded and uniformly continuous. We equip $C_{\mathrm{ub}}^m(\R)$ with the standard $\smash{W^{m,\infty}}$-norm.

Applying the Floquet-Bloch transform to $\El_0$ yields the family of Bloch operators
\begin{align*}
\El(\xi) u = k_0^2 D \left(\partial_\zeta + \ri \xi\right)^2 u + \omega_0 \left(\partial_\zeta + \ri \xi\right) u + f'(\phi_0(\zeta)) u,
\end{align*}
defined on $L_{\mathrm{per}}^2(0,1)$ with domain $D(\El(\xi)) = H_{\mathrm{per}}^2(0,1)$, and parameterized by the Bloch frequency $\xi \in [-\pi,\pi)$. Due to translational invariance, the derivative $\phi_0'$ of the wave train lies in the kernel of $\El(0)$. As a result, the spectrum of $\El_0$, given by
$$\sigma(\El_0) = \bigcup_{\xi \in [-\pi,\pi)} \sigma(\El(\xi)),$$
necessarily touches the origin. The most stable nondegenerate spectral configuration, known as \emph{diffusive spectral stability}, is then characterized by the following conditions:
\begin{itemize}
\setlength\itemsep{0em}
\item[\namedlabel{assD1}{\upshape (D1)}] It holds $\sigma(\El_0)\subset\{\lambda\in\C:\Re(\lambda)<0\}\cup\{0\}$;
\item[\namedlabel{assD2}{\upshape (D2)}] There exists $\theta>0$ such that for any $\xi\in[-\pi,\pi)$ we have $\Re\,\sigma(\El(\xi))\leq-\theta \xi^2$;
\item[\namedlabel{assD3}{\upshape (D3)}] $0$ is a simple eigenvalue of $\El(0)$.
\end{itemize}
Examples of reaction-diffusion systems that support diffusively spectrally stable wave trains include the complex Ginzburg-Landau equation~\cite{vanH}, the Gierer-Meinhardt system~\cite{PLO}, and the Brusselator model~\cite{SUKH}. We refer to~\cite{BjoernMod} for further discussion and additional references.

It follows from Hypothesis~\ref{assD3} that $0$ is also a simple eigenvalue of the adjoint operator $\El(0)^*$. We denote by $\smash{\widetilde{\Phi}_0} \in H^2_{\mathrm{per}}(0,1)$ the corresponding eigenfunction, normalized such that
\begin{align} \label{e:adjoint}
\big\langle \widetilde{\Phi}_0, \phi_0' \big\rangle_{L^2(0,1)} = 1.
\end{align}
Hypothesis~\ref{assD3} also implies that the wave train $\phi_0$ can be continued in the wavenumber, yielding a family of wave-train solutions to~\eqref{RD0} of the form $u_k(x,t) = \phi(k x - \omega(k) t; k)$, defined for $k$ near $k_0$, where the frequency $\omega(k)$ and the $1$-periodic profile function $\phi(\cdot;k)$ obey
\begin{align} \label{e:gauge}
\omega(k_0) = \omega_0, \qquad \phi(\cdot;k_0) = \phi_0, \qquad \big\langle \widetilde{\Phi}_0, \partial_k \phi(\cdot,k_0) \big\rangle_{L^2(0,1)} = 0,
\end{align}
cf.~Proposition~\ref{prop:family}. 

We now present our main result, which shows that solutions to~\eqref{RD} with initial data close to a largely modulated wave train converge at an enhanced diffusive rate to a modulated wave train, whose phase and wavenumber dynamics are governed by a solution $\breve{\gamma}(t)$ to the viscous Hamilton-Jacobi equation
\begin{align}
\partial_t \breve\gamma = d\breve\gamma_{\xx\xx} + a \breve\gamma_{\xx} + \nu \breve\gamma_\zeta^2, \label{e:HamJac}
\end{align}
with coefficients
\begin{align} \label{e:defad}
a = \omega_0  - k_0\omega'(k_0), \qquad d = \big\langle \widetilde{\Phi}_0,D \phi_0' + 2k_0D \partial_{\zeta k} \phi(\cdot;k_0)\big\rangle_{L^2(0,1)}, \qquad \nu = -\frac{1}{2}k_0^2 \omega''(k_0).
\end{align}

\begin{theorem}
\label{main_theorem}
Assume~\ref{assH1} and~\ref{assD1}-\ref{assD3}. Fix $\alpha \in [0,\frac16)$ and $M > 0$. Then, there exist constants $K,\varepsilon > 0$ such that, whenever $u_0 \in C_{\mathrm{ub}}(\mathbb{R})$ and $\gamma_0 \in C_{\mathrm{ub}}^1(\mathbb{R})$ satisfy 
\begin{align*}
\|\gamma_0\|_\infty \leq M, \qquad E_0 := \|u_0 - \phi_0(\cdot + \gamma_0(\cdot))\|_{\infty} + \|\gamma_0'\|_\infty <\varepsilon, 
\end{align*}
there exist a scalar function 
\begin{align*}
\gamma\in C\big([0,\infty), C_{\mathrm{ub}}^1(\mathbb{R})\big) \cap C^j\big((0,\infty), C_{\mathrm{ub}}^l(\mathbb{R})\big), \qquad j, l \in \NM_0
\end{align*}
with $\gamma(0) = \gamma_0$ and a unique classical global solution
\begin{align*} 
 u \in \mathcal{X} := C\big([0,\infty),C_{\mathrm{ub}}(\R)\big) \cap C\big((0,\infty),C_{\mathrm{ub}}^2(\R)\big) \cap C^1\big((0,\infty),C_{\mathrm{ub}}(\R)\big)
\end{align*}
to~\eqref{RD} with $u(\cdot,0) = u_0$, which obey the estimates
\begin{align} 
\label{e:mtest10}
\begin{split}
\left\|\gamma(t)\right\|_\infty \leq K, \qquad \left\|u(t)-\phi_0(\cdot + \gamma(\cdot,t))\right\|_\infty, \left\|\gamma_\xx(t)\right\|_\infty &\leq \frac{KE_0^\alpha}{(1+t)^{\frac12 - \alpha}}, \\
\left\|u(t) - \phi\big(\cdot + \gamma(\cdot,t)\left(1+\gamma_\xx(\cdot,t)\right);\,k_0(1 + \gamma_\xx(\cdot,t))\big)\right\|_\infty &\leq KE_0^\alpha\frac{\log(2+t)}{(1+t)^{1 - 2\alpha}}
\end{split}
\end{align}
for $t \geq 0$, and
\begin{align}  \label{e:mtest2}
\begin{split}
\left\|u(t)-\phi_0(\cdot + \gamma(\cdot,t))\right\|_{C_\mathrm{ub}^1} &\leq KE_0^\alpha \frac{ (1+t)^\alpha}{\sqrt{t}},\\ 
\left\|u(t) - \phi\big(\cdot + \gamma(\cdot,t)\left(1+\gamma_\xx(\cdot,t)\right);\,k_0(1 + \gamma_\xx(\cdot,t))\big)\right\|_{C_\mathrm{ub}^1}, \left\|\gamma_{\xx\xx}(t)\right\|_{C_{\mathrm{ub}}^1} &\leq KE_0^\alpha\frac{\log(2+t)}{\sqrt{t} \, (1+t)^{\frac12 - 2\alpha}}
\end{split}
\end{align}
for all $t > 0$. Moreover, there exists a unique classical global solution 
\begin{align*}
\breve{\gamma} \in \mathcal{Y} := C\big([0,\infty),C_{\mathrm{ub}}^1(\R)\big) \cap C\big((0,\infty),C_{\mathrm{ub}}^2(\R)\big) \cap C^1\big((0,\infty),C_{\mathrm{ub}}(\R)\big)
\end{align*}
with initial condition $\breve{\gamma}(0) = \gamma_0$ to the viscous Hamilton-Jacobi equation~\eqref{e:HamJac} such that we have the approximations
\begin{align}
\label{e:mtest3}
\left\|\gamma(t) - \breve{\gamma}(t)\right\|_\infty &\leq KE_0^{\frac23\alpha}, \qquad
\left\|\gamma_\xx(t) - \breve{\gamma}_\xx(t)\right\|_\infty \leq \frac{KE_0^\alpha}{\sqrt{1+t}}
\end{align}
for all $t \geq 0$. In particular, it holds
\begin{align} \label{e:mtest33}
\left\|u(t) - \phi_0\left(\cdot+\breve{\gamma}(\cdot,t)\right)\right\|_{\infty}  &\leq KE_0^{\frac23\alpha}, \qquad t \geq 0.
\end{align}
\end{theorem}

\begin{remark}{\upshape
We observe that the temporal decay rates presented in Theorem~\ref{main_theorem} match those established in the sharp nonlinear stability result for wave trains under $C_{\mathrm{ub}}$-perturbations in~\cite{BjoernMod}, confirming their optimality up to a possible logarithmic correction; see~\cite[Section~6.1]{BjoernMod}. Since our primary objective was to establish a modulational stability result with sharp temporal decay rates, we did not strive for optimal powers of the $E_0$-terms in the estimates~\eqref{e:mtest10},~\eqref{e:mtest2},~\eqref{e:mtest3}, and~\eqref{e:mtest33} in Theorem~\ref{main_theorem}. As the toy example in~\S\ref{sec:heuristic} suggests, we expect that these powers may be improved with further technical effort.
}
\end{remark}

\begin{remark} \label{rem:global}
{\upshape
Let $g_0 \in C_{\mathrm{ub}}^1(\mathbb{R})$ be any function such that $g_0' \in L^1(\mathbb{R}) \setminus \{0\}$ has mean zero. Set $M = \|g_0\|_\infty$ and fix $\alpha \in (0, \tfrac{1}{6})$. Let $K, \varepsilon > 0$ be as in Theorem~\ref{main_theorem}. Clearly, there exists $\delta > 0$ such that the function $\gamma_0 \in C_{\mathrm{ub}}^1(\mathbb{R})$ given by $\gamma_0(\xx) = g_0(\delta \xx)$ satisfies $\|\gamma_0\|_\infty \leq M$ and $\|\gamma_0'\|_\infty < \varepsilon$. Let $\breve{\gamma}$ denote the solution to the viscous Hamilton-Jacobi equation~\eqref{e:HamJac} with initial condition $\breve{\gamma}(0) = \gamma_0$. Then, $\breve{k}(t) := \breve{\gamma}_\xx(t)$ solves the viscous Burgers equation
\begin{align*} \partial_t \breve k = d\partial_\xx^2 \breve k + a\partial_\xx \breve k + \nu \partial_\xx \big(\breve{k}^2\big). \end{align*}
Since $\gamma_0' \in L^1(\mathbb{R})$ has mean zero, it follows from~\cite[Theorem~2 and Corollary~1]{FREI} that $\breve{k}(t)$ converges to zero in $L^1(\mathbb{R}) \cap L^\infty(\mathbb{R})$ as $t \to \infty$. This implies that $\breve{\gamma}(t)$ converges to a spatially constant function $\breve{\gamma}_\infty$ in $C_{\mathrm{ub}}^1(\mathbb{R})$ as $t \to \infty$. In particular, for large times, $\breve{\gamma}(t)$ is not close to its initial condition $\gamma_0$ in $C_{\mathrm{ub}}^1(\mathbb{R})$. By the estimates~\eqref{e:mtest10} and~\eqref{e:mtest3}, the same holds for $\gamma(t)$ and $u(t)$, demonstrating that Theorem~\ref{main_theorem} provides a \emph{global} stability result.
}
\end{remark}

\begin{remark} \label{rem:regularity}
{\upshape Although the reaction-diffusion system~\eqref{RD} admits mild solutions for initial data in $L^\infty(\mathbb{R})$ via standard analytic semigroup theory, the minimal condition ensuring that a solution $u(t)$ converges to its initial condition $u_0$ in $L^\infty(\mathbb{R})$ as $t \downarrow 0$ is that $u_0$ is uniformly continuous; see~\cite[Theorems~3.1.7 and 7.1.2]{LUN}. This right-continuity of $u(t)$ at $t = 0$ is clearly necessary for our main result to be meaningful, as is the requirement that the initial phase modulation $\gamma_0$ is bounded and differentiable, since we rely on the smallness of its derivative. In this sense, the regularity assumptions imposed on the initial data in Theorem~\ref{main_theorem} are minimal.
}\end{remark}

\begin{remark}
{
\upshape
Estimate~\eqref{e:mtest33} provides a leading-order description of the solution $u(t)$ that is valid for \emph{all} times. This estimate is fully explicit in terms of the underlying wave train $\phi_0$ and the solution $\breve{\gamma}(t)$ to the viscous Hamilton-Jacobi equation~\eqref{e:HamJac}, with initial condition $\breve{\gamma}(0) = \gamma_0 \in C_{\mathrm{ub}}^1(\mathbb{R})$. In contrast, in the other work~\cite{IYSA} addressing large phase modulations of wave-train solutions to~\eqref{RD0}, the following estimate was established for fixed $\alpha > 0$:
\begin{align}
\|u(t) - \phi_0\left(\cdot+\Phi(\cdot,t)\right)\|_{\infty} \leq Kt^{-\frac12+\alpha}, \qquad t \geq 0, \label{e:Iyer}
\end{align}
where $\Phi$ is a self-similar front solution to~\eqref{e:HamJac} and $K > 0$ is a time-independent constant. While the estimate~\eqref{e:Iyer} does not capture the leading-order behavior of $u(t)$ for short times, it yields near-diffusive decay for large times. We argue that this large-time decay stems from the stronger assumptions imposed on the initial phase modulation $\gamma_0$ in~\cite{IYSA}.

Indeed, even with conditions that are less stringent than those in~\cite{IYSA}, we expect to recover an estimate matching~\eqref{e:Iyer}. To formulate these conditions, we fix asymptotic limits $\gamma_\pm \in \R$ and consider the associated self-similar front solution
\begin{align*}
\Phi(\xx,t) = 
\begin{cases} 
\gamma_- + \gamma_d \, \mathrm{erf}\left(\dfrac{\xx + a(t+1)}{\sqrt{d(t+1)}}\right), & \nu = 0, \\[1.5ex]
\displaystyle \gamma_- + \frac{d}{\nu} \log\left(1 + \left(\re^{\frac{\nu}{d} \gamma_d} - 1\right) \mathrm{erf}\left(\dfrac{\xx + a(t+1)}{\sqrt{4d(t+1)}}\right)\right), & \nu \neq 0,
\end{cases}
\end{align*}
to~\eqref{e:HamJac}, where $\gamma_d := \gamma_+ - \gamma_-$ denotes its phase offset, and
\begin{align*}
\mathrm{erf}(\zeta) = \frac{1}{\sqrt{4\pi}} \int_{-\infty}^{\zeta} \re^{-\frac{\xt^2}{4}} \, \mathrm{d} \xt
\end{align*}
is the Gaussian error function. Thus, under the additional assumptions that $\gamma_0$ is twice continuously differentiable with $\gamma_0'' \in L^2(\mathbb{R})$ and satisfies
\begin{align*}
\left\| \gamma_0 - \Phi \right\|_{L^1 \cap L^\infty} \leq M,
\end{align*}
we expect estimate~\eqref{e:mtest33} in Theorem~\ref{main_theorem} to improve to
\begin{align} \label{e:finalloc}
\left\| u(t) - \phi_0(\cdot + \Phi(t)) \right\|_{\infty} \leq K \frac{\left(\log(2+t)\right)^2}{\sqrt{1+t}}, \qquad t \geq 0.
\end{align}
We note that estimate~\eqref{e:finalloc} is slightly stronger than~\eqref{e:Iyer} Since the proof of~\eqref{e:finalloc} requires substantial additional effort, especially due to the need for tracking $L^2$-norms in the nonlinear iteration, we defer it to future work.
}
\end{remark}

\begin{remark}
{\upshape Converting to the original $(x,t)$-variables in Theorem~\ref{main_theorem}, we obtain a constant $\mathring{K} > 0$ such that, for each $\wt_0 \in C_{\mathrm{ub}}(\R)$ and $\gt_0 \in C_{\mathrm{ub}}^1(\R)$ with $\|\gt_0\|_\infty \leq M$ and $E_0 = \|\wt_0\|_\infty + \tfrac{1}{k_0}\|\gt_0'\|_\infty < \varepsilon$, there exist a unique classical global solution $u \in \mathcal{X}$ to the reaction-diffusion system~\eqref{RD0} with initial condition~\eqref{e:ICintro} and a phase modulation function $\gt \in\mathcal{Y}$ with $\gt(0) = \gt_0$ such that
 \begin{align*} 
\begin{split}
\left\|\gt(t)\right\|_\infty \leq \mathring{K}, \qquad \left\|\wt(t)\right\|,\left\|\gt_x(t)\right\|_\infty \leq \frac{\mathring{K}E_0^\alpha}{(1+t)^{\frac12 - \alpha}}, \qquad \left\|\mathring{y}(t)\right\|_\infty \leq \mathring{K}E_0^\alpha\frac{\log(2+t)}{(1+t)^{1 - 2\alpha}}
\end{split} 
\end{align*}
for all $t \geq 0$, and
\begin{align*} 
\begin{split}
\left\|\wt(t)\right\|_{C_\mathrm{ub}^1} \leq \mathring{K}E_0^\alpha \frac{ (1+t)^\alpha}{\sqrt{t}}, \qquad \left\|\mathring{y}(t)\right\|_{C_\mathrm{ub}^1} &\leq \mathring{K}E_0^\alpha\frac{\log(2+t)}{\sqrt{t} \, (1+t)^{\frac12 - 2\alpha}}
\end{split} 
\end{align*}
for all $t > 0$, where $\wt(t)$ and $\mathring{y}(t)$ are the residuals given by~\eqref{e:asympdescr0} and~\eqref{e:asympdescr}, respectively. Moreover, we have the approximations
\begin{align*}
\left\|\gt(t) - \breve{\gt}(t)\right\|_\infty &\leq \mathring{K} E_0^{\frac23 \alpha}, \qquad
\left\|\gt_x(t) - \breve{\gt}_x(t)\right\|_\infty \leq \frac{\mathring{K}E_0^\alpha}{\sqrt{1+t}}
\end{align*}
for $t \geq 0$, where $\breve{\gt} \in \mathcal{Y}$ is the classical solution to the viscous Hamilton-Jacobi equation~\eqref{e:HamJacIntro} with initial condition $\breve{\gt}(0) = \gt_0$.
}\end{remark}

\subsection{Strategy of proof} \label{sec:strategy}

Our proof builds on an extension of the recently developed stability theory in~\cite{BjoernMod}, which is based on pure $L^\infty$-estimates and accommodates fully nonlocalized perturbations. A key difference from the setting in~\cite{BjoernMod} is that we cannot exploit smallness of $\|\gamma_0\|_\infty$, which presents a significant challenge. In particular, a linear term of the form $\re^{\El_0 t}(\phi_0' \gamma_0)$ appears in the Duhamel formulation of the \emph{inverse-modulated perturbation}
\begin{align} \label{e:defv}
v(\xx,t) = u(\xx - \gamma(\xx,t),t) - \phi_0(\xx),
\end{align}
where $\re^{\El_0t}$ denotes the semigroup on $C_{\mathrm{ub}}(\R)$ generated by the linearization $\El_0$. The inverse-modulated perturbation $v(t)$ controls the dynamics of the difference
\begin{align} \label{e:defringv}
\mathring{v}(\xx,t) = u(\xx,t) - \phi_0(\xx + \gamma(\xx,t))
\end{align}
between the solution and the modulated wave train. It is typically used in nonlinear stability analyses, since its evolution equation exhibits more favorable decay properties compared to the one of the \emph{forward-modulated perturbation} $\mathring{v}(t)$; see~\cite{ZUM23}. A crucial observation is that the lack of smallness in the Duhamel formulation of $v(t)$ can be precisely quantified using the identity
\begin{align}
\label{e:strat1}
\begin{split}
\re^{\El_0 t}\left(\phi_0' \gamma_0\right) - \phi_0' \gamma_0 &= \int_0^t \re^{\El_0 s} \El_0 \left(\phi_0' \gamma_0\right) \de s\\ 
&= \int_0^t \re^{\El_0 s} \left(k_0^2 D \left(\partial_\xx\left(\phi_0' \gamma_0'\right) + \phi_0'' \gamma_0'\right) + \omega_0 \phi_0' \gamma_0'\right) \de s,
\end{split}
\end{align}
which follows from standard semigroup theory~\cite[Proposition~2.1.4]{LUN} and the fact that $\phi_0'$ lies in the kernel of $\El(0)$. Since the right-hand side of~\eqref{e:strat1} depends only on the \emph{derivative} $\gamma_0'$, whose $L^\infty$-norm is assumed to be small, the lack of smallness in $\re^{\El_0 t}(\phi_0' \gamma_0)$ manifests solely in the term $\phi_0' \gamma_0$.

Since the linear operator $\El_0$ has spectrum touching the imaginary axis, the semigroup $\re^{\El_0 t}$ does not exhibit decay. To obtain sufficient decay to close a nonlinear argument, it is therefore necessary to decompose $\re^{\El_0 t}$ into a principal part of the form $\phi_0' S_p^0(t)$, where $S_p^0(t)$ enjoys the same bounds and smoothing properties as the heat semigroup $\smash{\re^{\partial_\xx^2 t}}$, and a residual part exhibiting faster temporal decay. Specifically, the phase modulation $\gamma(t)$ is chosen in such a way that, for large times, it captures the slowest-decaying contributions associated with $S_p^0(t)$ in the Duhamel formulation of the inverse-modulated perturbation $v(t)$; see~\S\ref{sec:phase_mod}. Since the phase modulation enters only via its \emph{derivatives} in this formulation, we may exploit improved decay of $\partial_\xx^j \partial_t^l S_p^0(t)$ for $j,l \in \NM_0$ with $j+l\geq1$ which arises due to diffusive smoothing.

It is thus essential to characterize the lack of smallness in $S_p^0(t)(\phi_0' \gamma_0)$, analogously to the identity~\eqref{e:strat1}. Inspired by the Fourier-series arguments in~\cite{JONZNL}, we perform a detailed linear analysis, extending the $L^\infty$-techniques from~\cite{HDRS22} to the setting of large modulational initial data. We find, similar to~\eqref{e:strat1}, that $S_p^0(t)(\phi_0' \gamma_0) - \gamma_0$ can be bounded in terms of $\gamma_0'$. Crucially, this implies that the same holds for spatial and temporal derivatives of $S_p^0(t)(\phi_0' \gamma_0)$, allowing us to recover smallness in the linear terms arising in the Duhamel representations of \emph{derivatives} of $\gamma(t)$.

There is, however, a caveat: bounding $S_p^0(t)(\phi_0' \gamma_0) - \gamma_0$ and its derivatives in terms of $\gamma_0'$ yields only weak decay rates and, in some cases, even temporal growth. This is also reflected in identity~\eqref{e:strat1}, where the time integral on the right-hand side naturally leads to bounds exhibiting growth in $t$. We address this issue by \emph{interpolating} between two types of bounds: (i) classical bounds on $\re^{\El_0 t}(\phi_0' \gamma_0)$ and $\smash{\partial_\xx^j \partial_t^l S_p^0(t)(\phi_0' \gamma_0)}$, used in the stability argument in~\cite{BjoernMod}, which provide sufficient decay but not smallness, and (ii) bounds on $\re^{\El_0 t}(\phi_0' \gamma_0) - \phi_0' \gamma_0$ and $\smash{\partial_\xx^j \partial_t^l(S_p^0(t)(\phi_0' \gamma_0) - \gamma_0)}$ in terms of $\gamma_0'$, which provide smallness but not sufficient decay. In~\S\ref{sec:heuristic}, the core idea behind this interpolation strategy is illustrated by means of a representative toy example.

By carefully balancing smallness and decay, we are able to derive sufficient bounds to close the nonlinear iteration scheme. Heuristically, this is possible because the nonlinear estimates in~\cite{BjoernMod} leave some leeway, i.e., the decay bounds on the nonlinear terms are not critical and can tolerate some weakening.

An important distinction from the analysis in~\cite{BjoernMod} lies in how we control regularity in the nonlinear iteration. In~\cite{BjoernMod}, regularity is controlled via estimates on the perturbation $\tilde{v}(t) = u(t) - \phi_0$. However, Theorem~\ref{main_theorem} shows that in the current modulational setting, $\tilde{v}(t)$ is no longer small, making it natural to instead consider to consider the forward-modulated perturbation $\mathring{v}(t)$. Unlike the inverse-modulated perturbation $v(t)$, which obeys a quasilinear equation (see~\S\ref{sec:inv_mod_pert}), the forward-modulated perturbation satisfies a semilinear equation that does not suffer from a loss of derivatives. However, the obtained decay rates of $\mathring{v}(t)$ are too weak to gain good enough large-time control on derivatives. Therefore, following the approach in~\cite{AdR1}, we establish a nonlinear damping estimate for the \emph{modified forward-modulated perturbation}
\begin{align} \label{e:defringz}
\mathring{z}(\xx,t) = u(\xx,t) - \phi\left(\xx + \gamma(\xx,t)\left(1+\gamma_\xx(\xx,t)\right);\,k_0(1 + \gamma_\xx(\xx,t))\right)
\end{align}
in order to control regularity for large times. As indicated by the estimate~\eqref{e:mtest10}, this variable enjoys better decay properties than $\vt(t)$. Due to the absence of $L^2$-localization, we derive this energy estimate in uniformly local Sobolev spaces; see~\cite{SU17book} for background.

With the regularity control in hand, we are able to close a nonlinear iteration argument, thereby proving Theorem~\ref{main_theorem}. Further details of the nonlinear iteration scheme are provided in~\S\ref{sec:itscheme}.

\subsection{Outline}

In~\S\ref{sec:heuristic}, we present the core idea behind handling large modulational data through interpolation arguments in the context of a toy model. Section~\ref{sec:prelim} reviews standard preliminary results on wave trains. In~\S\ref{sec:decomp}, we recall the semigroup decomposition from~\cite{BjoernMod} and summarize the corresponding linear estimates established therein. We also carry out a detailed analysis of the linear dynamics arising from large, fully nonlocalized modulational initial data and derive associated estimates. In~\S\ref{sec:itscheme}, we develop our nonlinear iteration scheme and establish a nonlinear damping estimate to control regularity for large times. Section~\ref{sec:nonlinearstab} is devoted to the proof of our main result, Theorem~\ref{main_theorem}. An outlook on future research directions is provided in~\S\ref{sec:discussion}. Appendix~\ref{app:aux} contains technical low- and high-frequency estimates required for the linear analysis of large modulational data. Finally, Appendix~\ref{app:B} provides the proof of the local existence result for the phase modulation.

\paragraph*{Notation.} Let $S$ be a set, and let $A, B \colon S \to \R$. Throughout the paper, the expression ``$A(x) \lesssim B(x)$ for $x \in S$'', means that there exists a constant $C>0$, independent of $x$, such that $A(x) \leq CB(x)$ holds for all $x \in S$. 

\paragraph*{Acknowledgments.}  This project is funded by the Deutsche Forschungsgemeinschaft (DFG, German Research Foundation) -- Project-ID 491897824 and Project-ID 258734477 - SFB 1173. 

\paragraph*{Data availability statement.} Data sharing is not applicable to this article as no datasets were generated or analyzed during the current study.

\section{Heuristic behind handling large modulational data} \label{sec:heuristic}

The goal of this section is to convey the core idea behind our $L^\infty$-based nonlinear iteration argument designed to handle large phase modulations. Motivated by the fact that the leading-order modulational dynamics is given by the viscous Hamilton-Jacobi equation~\eqref{e:HamJac}, we consider
\begin{align} \label{e:toy1}
\partial_t w = d w_{\zeta\zeta} + a w_\zeta + \nu w_\zeta^2 + \mu w_\zeta^3
\end{align}
with $d > 0$ and $a, \mu, \nu \in \R$. Here, the cubic contribution $w_\zeta^3$ plays the role of a higher-order residual term. When $\nu \neq 0$, the quadratic term in~\eqref{e:toy1} can be removed via the Cole-Hopf transform
\begin{align*}
v(t) = \re^{\frac{\nu}{d} w(t)},
\end{align*}
which transforms the equation into
\begin{align*}
\partial_t v = d v_{\zeta\zeta} + a v_\zeta + \frac{\mu d^2}{\nu^2}\, \frac{v_\zeta^3}{v^2}.
\end{align*}
This motivates restricting attention to the case $\nu = 0$ in~\eqref{e:toy1}. Moreover, by transitioning to a co-moving frame and rescaling space, we may assume without loss of generality that $a = 0$ and $d = 1$. The resulting simplified toy problem takes the form
\begin{align} \label{e:toy2}
\partial_t w = w_{\zeta\zeta} + \mu w_\zeta^3
\end{align}
with $\mu \in \R$. We use this model to illustrate how our $L^\infty$-based scheme effectively manages large phase modulations. Specifically, under the sole assumption that the initial condition $w_0 \in C_{\mathrm{ub}}^1(\R)$ has sufficiently small derivative, we establish global existence of the associated solution $w(t)$ to~\eqref{e:toy2} and show that its derivative $w_\zeta(t)$ decays at a near-diffusive rate.

Take $w_0 \in C_{\mathrm{ub}}^1(\R)$. Integrating~\eqref{e:toy2}, we obtain the Duhamel formula
\begin{align}
w(t) &= \re^{\partial_\zeta^2 t} w_0 + \mu \int_0^t \re^{\partial_\zeta^2 (t-s)} w_\zeta(s)^3 \de s. \label{e:du_heat1} 
\end{align}
By standard local existence theory for semilinear parabolic equations~\cite{LUN}, there exist $T \in (0,\infty]$ and a unique maximal solution $w \in C\big([0,T), C_{\mathrm{ub}}^1(\R)\big)$ satisfying~\eqref{e:du_heat1}. If $T < \infty$, then it holds
\begin{align} 
\sup_{t \uparrow T} \|w(t)\|_{C_{\mathrm{ub}}^1} = \infty. \label{e:qblowup}
\end{align}

The smoothing effect of the heat semigroup leads to the linear $L^\infty$-estimate
\begin{align} 
\left\|\partial_\zeta^j \re^{\partial_\zeta^2 t}v\right\|_\infty \lesssim t^{-\frac{j}{2}} \|v\|_\infty \label{e:linheat}
\end{align}
for $j = 0,1$, $v \in C_{\mathrm{ub}}(\R)$, and $t > 0$. On the other hand, we may commute the derivative with the semigroup, which gives
\begin{align} 
\left\|\partial_\zeta \re^{\partial_\zeta^2 t}v\right\|_\infty = \left\|\re^{\partial_\zeta^2 t} \partial_\zeta v\right\|_\infty \leq \|\partial_\zeta v\|_\infty \label{e:linheat2}
\end{align}
for $v \in C_{\mathrm{ub}}^1(\R)$ and $t \geq 0$. Estimates of the form~\eqref{e:linheat2}, where derivatives are transferred onto the initial data through the linear propagator, are crucial to exploit smallness in the derivative of the (possibly large) initial data; see~\S\ref{sec:mod_data}. 

If the initial condition $w_0$ is small in $C_{\mathrm{ub}}^1(\R)$, then the linear behavior dominates, and the solution $w(t)$ to~\eqref{e:du_heat1} remains bounded while its derivative decays diffusively at rate $\smash{t^{-\frac{1}{2}}}$; see~\cite[Section~2]{HDRS22}. Here, we show that, by sacrificing a small amount of decay, we obtain global-in-time control on solutions to~\eqref{e:du_heat1} with large initial data. To that end, we fix $M > 0$ and $a \in (0,\tfrac{1}{6})$, take $w_0 \in C_{\mathrm{ub}}^1(\R)$ with $\|w_0\|_{C_{\mathrm{ub}}^1} \leq M$, and introduce the template function $\eta \colon [0,T) \to \R$ defined by
\begin{align*}
\eta(t) = \sup_{0 \leq s \leq t} (1+s)^{\frac{1}{2} - a} \, \left\|\partial_\zeta w(s)\right\|_\infty.
\end{align*}

Using the linear bound~\eqref{e:linheat}, we estimate the right-hand side of~\eqref{e:du_heat1} as
\begin{align} \label{e:du_heat2}
\|w(t)\|_\infty \lesssim \|w_0\|_\infty + \int_0^t \frac{\eta(t)^3}{(1+s)^{\frac{3}{2} - 3a}} \de s \lesssim \|w_0\|_\infty+ \eta(t)^3,
\end{align}
for $t \in [0,T)$. For the spatial derivative of~\eqref{e:du_heat1}, we apply~\eqref{e:linheat2} for $s \in [0,1]$ and use the interpolation inequality
\begin{align*}
\left\|\partial_\zeta \re^{\partial_\zeta^2 s} w_0\right\|_\infty 
&\leq \left\|\partial_\zeta \re^{\partial_\zeta^2 s} w_0\right\|_\infty^{1 - 2a} 
\left\|\re^{\partial_\zeta^2 s} w_0'\right\|_\infty^{2a} 
\lesssim (1+s)^{\frac{1}{2} - a} \|w_0\|_\infty^{1 - 2a} \|w_0'\|_\infty^{2a},
\end{align*}
for $s \geq 1$, yielding the bound
\begin{align*}
\left\|\partial_\zeta \re^{\partial_\zeta^2 s} w_0\right\|_\infty 
\lesssim (1+s)^{\frac{1}{2} - a} \|w_0\|_{C_{\mathrm{ub}}^1}^{1 - 2a} \|w_0'\|_\infty^{2a},
\end{align*}
for all $s \geq 0$. Altogether, we estimate the spatial derivative of~\eqref{e:du_heat1} as
\begin{align*}
\|w_\zeta(s)\|_\infty 
&\lesssim (1+s)^{\frac{1}{2} - a} \|w_0\|_{C_{\mathrm{ub}}^1}^{1 - 2a} \|w_0'\|_\infty^{2a} 
+ \int_0^s \frac{\eta(s)^3}{\sqrt{s-r} \, (1+r)^{\frac{3}{2} - 3a}} \de r \\
&\lesssim (1+s)^{-\frac{1}{2} + a} \left( M^{1 - 2a} \|w_0'\|_\infty^{2a} + \eta(s)^3 \right),
\end{align*}
for $s \in [0,T)$. Taking the supremum in the latter inequality over $s$, we find a constant $C \geq 1$, independent of $t$ and $w_0$, such that
\begin{align} 
\eta(t) \leq C \left( \|w_0'\|_\infty^{2a} + \eta(t)^3 \right), \label{e:heat_key}
\end{align}
for all $t \in [0,T)$. Now assume $\|w_0'\|_\infty^{4a} < \tfrac{1}{8C^2}$. If there exists $t \in (0,T)$ with $\eta(t) > 2C\|w_0'\|_\infty^{2a}$, then by continuity of $\eta$ and the fact that $\eta(0) = \|w_0'\|_\infty < 2C\|w_0'\|_\infty^{2a}$, there must exist $\tau \in [0,t]$ such that $\eta(\tau) = 2C\|w_0'\|_\infty^{2a}$. Applying~\eqref{e:heat_key} and $\|w_0'\|_\infty^{4a} < \tfrac{1}{8C^3}$, we obtain
\begin{align*}
\eta(\tau) \leq C\left(\|w_0'\|_\infty^{2a} + \eta(\tau)^3\right) 
< 2C \|w_0'\|_\infty^{2a},
\end{align*}
yielding a contradiction. Hence, we must have $\eta(t) \leq 2C\|w_0'\|_\infty^{2a}$ for all $t \in [0,T)$. Combining this with~\eqref{e:qblowup} and~\eqref{e:du_heat2}, we conclude that $T = \infty$. So, the solution is global and enjoys the bounds
\begin{align*}
\|w(t)\|_{\infty} \lesssim \|w_0\|_\infty + \|w_0'\|_\infty^{6a}, \qquad
\|w_\zeta(t)\|_{\infty} \lesssim \frac{\|w_0'\|_\infty^{2a}}{(1+t)^{\frac{1}{2} - a}}
\end{align*}
for $t \geq 0$. 

We emphasize that the nonlinear iteration argument above relies exclusively on $L^\infty$-estimates and requires smallness solely in the \emph{derivative} of the initial data. This is because the nonlinearity in~\eqref{e:toy2} involves only spatial derivatives of $w$. Analogously, the equations for the inverse- and forward-modulated perturbations, used in the proof of Theorem~\ref{main_theorem}, depend only on derivatives of the phase modulation $\gamma(t)$.

\section{Preliminaries} \label{sec:prelim}

We recall some basic results on wave trains and their dispersion relations from~\cite{BjoernMod}. For a more comprehensive treatment, we refer the reader to~\cite[Section~4]{DSSS}.

The first result asserts that, under the nondegeneracy condition~\ref{assD3}, the wave train can be continued in the wavenumber.

\begin{proposition} \label{prop:family}
Assume~\ref{assH1} and~\ref{assD3}. Then, there exist $r_0 \in (0,\frac12)$ and smooth functions $\phi \colon \R \times (k_0-r_0,k_0 + r_0) \to \R$ and $\omega \colon (k_0-r_0,k_0+r_0) \to \R$, satisfying~\eqref{e:gauge}, such that
$$u_k(x,t) = \phi(k x - \omega(k) t;k)$$
is a wave-train solution to~\eqref{RD0} for each wavenumber $k \in (k_0-r_0,k_0+r_0)$. The profile function $\phi(\cdot;k)$ has period $1$ for each $k \in (k_0-r_0,k_0+r_0)$.
\end{proposition}

The next result provides an expansion of the critical spectral curve of the linearization $\El_0$ in terms of the Bloch frequency parameter $\xi$. The curve, known as the \emph{linear dispersion relation}, touches the origin in a quadratic tangency. In addition, the result includes an expansion of the associated Bloch eigenfunctions.

\begin{proposition} \label{prop:speccons}
Assume~\ref{assH1} and~\ref{assD2}-\ref{assD3}. Then, there exist a constant $\xi_0 \in (0,\pi)$ and an analytic curve $\lambda_c \colon (-\xi_0,\xi_0) \to \C$ satisfying 
\begin{itemize}
\setlength\itemsep{0em}
\item[(i)] The complex number $\lambda_c(\xi)$ is a simple eigenvalue of $\El(\xi)$ for any $\xi \in (-\xi_0,\xi_0)$. An associated eigenfunction $\Phi_\xi$ of $\El(\xi)$ lies in $H_{\mathrm{per}}^m(0,1)$ for each $m \in \mathbb N_0$, satisfies $\Phi_0 = \phi_0'$, is analytic in $\xi$ and fulfills
\begin{align*}
 \big\langle \widetilde{\Phi}_0,\Phi_\xi\big\rangle_{L^2(0,1)} = 1.
\end{align*}
\item[(ii)] The complex conjugate $\overline{\lambda_c(\xi)}$ is a simple eigenvalue of the adjoint $\El(\xi)^*$ for any $\xi \in (-\xi_0,\xi_0)$. An associated eigenfunction $\widetilde{\Phi}_\xi$ lies in $H_{\mathrm{per}}^m(0,1)$ for each $m \in \mathbb N_0$, is smooth in $\xi$ and satisfies
\begin{align*}
 \big\langle \widetilde{\Phi}_\xi,\Phi_\xi\big\rangle_{L^2(0,1)} = 1.
\end{align*}
\item[(iii)] The expansions
\begin{align*}
\left|\lambda_c(\xi) - \ri a\xi + d \xi^2\right| \lesssim |\xi|^3, \qquad \left\|\Phi_\xi - \phi_0' - \ri k_0 \xi \partial_k \phi(\cdot;k_0)\right\|_{H^m(0,1)} \lesssim |\xi|^2
\end{align*}
hold for $\xi \in (-\xi_0,\xi_0)$ with coefficients $a \in \R$ and $d > 0$ given by~\eqref{e:defad}.
\end{itemize}
\end{proposition}

\section{Semigroup decomposition and linear estimates} \label{sec:decomp}

The linearization $\El_0$ of~\eqref{RD} about the wave train $\phi_0$ is a densely defined sectorial operator on $C_{\mathrm{ub}}(\R)$, with domain $D(\El_0) = C_{\mathrm{ub}}^2(\R)$; see~\cite[Corollary~3.1.9]{LUN}. In this section, we recall the decomposition of the associated analytic semigroup $\re^{\El_0 t}$ carried out in~\cite{BjoernMod}, and summarize the corresponding linear estimates established therein. In addition, we derive novel estimates on modulational data, which are crucial for treating large phase modulations.

The first result from~\cite{BjoernMod} expresses $\re^{\El_0 t}$ as the sum of an explicit principal part, exhibiting the same $L^\infty$-bounds as the heat semigroup $\smash{\re^{\partial_\zeta^2 t}}$, and a remainder that decays algebraically at rate $t^{-1}$.

\begin{proposition}[\!{\!\cite[Propositions~3.2 and~3.4]{BjoernMod}}] \label{prop:lin1}
Assume~\ref{assH1} and~\ref{assD1}-\ref{assD3}. Let $\phi(\cdot;k)$, $\xi_0$, $\lambda_c$, $a$, and $\smash{\widetilde{\Phi}_\xi}$ be as in Propositions~\ref{prop:family} and~\ref{prop:speccons}. Fix $j,l,m \in \NM_0$ and $\ell_0,\ell_1 \in \{0,1\}$ with $\ell_0 + \ell_1 \leq 1$. The semigroup $\re^{\El_0 t}$ decomposes as
\begin{align} \label{e:decomp_full_semigroup}
\re^{\El_0 t} = \left(\phi_0' + k_0 \partial_k \phi(\cdot;k_0) \partial_\xx \right)S_p^0(t) + \widetilde{S}(t)
\end{align}
for $t \geq 0$, where the principal part satisfies the commutator identities
\begin{align} \label{e:commu} 
\begin{split}
\partial_\xx S_p^0(t) - S_p^0(t)\partial_\xx &= S_p^1(t), \qquad
\partial_\xx^2 S_p^0(t) - S_p^0(t)\partial_\xx^2
= 2\partial_\xx S_p^1(t) - S_p^2(t)
\end{split}
\end{align}
and is explicitly given by
\begin{align} \label{e:prin_rep}
S_p^i(t) v(\zeta) = \chi(t) \int_\R G_p^i(\xx,\xt,t) v(\xt) \de \xt, \qquad G_p^i(\xx,\xt,t) = \frac{1}{2\pi} \int_{\R} \rho(\xi) \re^{\ri\xi(\xx - \xt)} \re^{\lambda_c(\xi) t} \partial_\xt^i \widetilde{\Phi}_\xi(\xt)^* \de \xi
\end{align}
for $i = 0,1,2$, where $\rho \colon \R \to [0,1]$ and $\chi \colon [0,\infty) \to [0,1]$ are smooth cut-off functions satisfying $\rho(\xi)=1$ for $|\xi|<\frac{\xi_0}{2}$, $\rho(\xi)=0$ for $|\xi| > \xi_0$, $\chi(t) = 0$ for $t \in [0,1]$, and $\chi(t) = 1$ for $t \in [2,\infty)$. Moreover, the estimates
\begin{align*}
\left\|\partial_\zeta^{\ell_0} \re^{\El_0 t} \partial_\zeta^{\ell_1} v\right\|_\infty &\lesssim \left(1 + t^{-\frac{\ell_0+\ell_1}{2}}\right) \|v\|_\infty,\\
\left\|\widetilde{S}(t) \partial_\zeta^{\ell_1} v\right\|_\infty &\lesssim (1+t)^{-1}\left(1 + t^{-\frac{\ell_1}{2}}\right) \|v\|_\infty,\\
\left\|\left(\partial_t - a \partial_\zeta\right)^j \partial_\zeta^l S_p^i(t) \partial_\zeta^m w\right\|_\infty &\lesssim \left(1 + t\right)^{-\frac{2j+l}{2}} \|w\|_\infty
\end{align*}
hold for all $v \in C_{\mathrm{ub}}^{\ell_1}(\R)$, $w \in C_{\mathrm{ub}}^m(\R)$, $t > 0$, and $i = 0,1,2$. 
\end{proposition}

The second result from~\cite{BjoernMod} connects the principal component of the semigroup to the convective heat equation $\partial_t u = d u_{\zeta\zeta} + a u_{\zeta}$, which arises from the quadratic truncation of the low-frequency expansion of the linear dispersion relation in Proposition~\ref{prop:speccons}.

\begin{proposition}[\!{\!\cite[Propositions~3.6 and~3.7]{BjoernMod}}] \label{prop:lin2}
Assume~\ref{assH1} and~\ref{assD1}-\ref{assD3}. Let $l \in \NM_0$, $m \in \{0,1\}$, and $i \in \{0,1,2\}$. There exists a bounded linear operator $A_h \colon L^2_{\mathrm{per}}\big((0,1),\R^n\big) \to C(\R,\R)$ such that the principal component $S_p^i(t)$ decomposes as
\begin{align}
\label{e:prin_rep1} S_p^i(t) = S_h^i(t) + \widetilde{S}_r^i(t),
\end{align}
where we denote
\begin{align*}
S_h^i(t)v = \re^{\left(d\partial_\xx^2 + a\partial_\xx\right) t} \left[\left(\partial_\xx^i \widetilde{\Phi}_0^*\right) v\right]
\end{align*}
for $v \in C_{\mathrm{ub}}(\R)$ and $t \geq 0$. Moreover, we have
\begin{align} \label{e:prin_rep2}
S_h^0(t)\left(gv\right) = \re^{\left(d\partial_\xx^2 + a\partial_\xx\right) t}\left(\langle \widetilde{\Phi}_0, g\rangle_{L^2(0,1)} v - A_h(g) \partial_\xx v\right) + \partial_\xx \re^{\left(d\partial_\xx^2 + a\partial_\xx\right) t}\left(A_h(g) v\right)
\end{align}
for $g \in L^2_{\mathrm{per}}((0,1),\R^n)$, $v \in C_{\mathrm{ub}}^1(\R,\R)$, and $t > 0$. Finally, the estimates
\begin{align*}
\left\|\partial_\zeta^l S_h^i(t)v\right\|_{\infty} \lesssim t^{-\frac{l}{2}}\|v\|_{\infty}, \qquad
\left\|\partial_\zeta^m \widetilde{S}_r^i(t)v\right\|_\infty &\lesssim (1+t)^{-\frac{1}{2}}t^{-\frac{m}{2}}\|v\|_{\infty}
\end{align*}
and
\begin{align} \label{e:Gamma_rates}
\begin{split}
 \left\|\partial_\xx^{1 + l} \re^{\left(d\partial_\xx^2 + a \partial_\xx\right) t} w\right\|_\infty &\lesssim t^{-\frac{l}{2}} \|w'\|_\infty, \qquad \left\|(\partial_t - a \partial_\xx) \partial_\xx^l \re^{\left(d\partial_\xx^2 + a \partial_\xx\right) t} w\right\|_\infty \lesssim t^{-\frac{1+l}{2}} \|w'\|_\infty, \\ 
& \left\|\partial_\zeta^l \re^{\left(d\partial_\xx^2 + a\partial_\xx\right) t}v\right\|_{\infty} \lesssim t^{-\frac{l}{2}}\|v\|_{\infty}
\end{split}
\end{align}
hold for $v \in C_{\mathrm{ub}}(\R)$, $w \in C_{\mathrm{ub}}^1(\R)$, and $t > 0$.
\end{proposition}

\subsection{\texorpdfstring{$L^\infty$}{Linfty}-estimates on modulational data} \label{sec:mod_data}

Initial phase modulations $\gamma_0$ of the wave train $\phi_0$ give rise to data of the form $\phi_0' \gamma_0$ in the equation for the inverse-modulated perturbation; see~\S\ref{sec:inv_mod_pert}. Such \emph{modulational data} correspond to the linearized approximation of the modulational perturbation $\phi_0(\xx + \gamma_0(\xx)) - \phi_0(\xx) \approx \phi_0'(\xx) \gamma_0(\xx)$. Since we allow for large phase modulations $\gamma_0$, the nonlinear argument can only exploit smallness in the \emph{derivative} $\gamma_0'$. Here, we establish bounds that are designed to estimate the action of linear propagators on modulational data in terms of $\gamma_0'$.

Linear estimates on modulational data have also been obtained in~\cite{JONZNL,JUNNL} under the additional assumption that $\gamma_0' \in L^1(\R)$. The approach in~\cite{JONZNL,JUNNL} involves a detailed analysis in Bloch frequency domain, relying on the Hausdorff-Young inequality to control the action of the propagator in terms of the $L^1$-norm of $\gamma_0'$. Although one could potentially extend this analysis to $C_{\mathrm{ub}}$-data by working with tempered distributions, we avoid such technicalities. Instead, we derive bounds using abstract semigroup theory, a careful decomposition of the principal component $S_p^0(t)$ of the semigroup, and an adaptation of the high- and low-frequency $L^\infty$-bounds from~\cite[Appendix~A]{BjoernMod} and~\cite[Appendix~A]{HDRS22} to the setting of modulational data.

We begin by estimating the action of the full semigroup on modulational data.

\begin{proposition} \label{prop:lin_mod_1}
Assume~\ref{assH1} and~\ref{assD1}-\ref{assD3}. Then, the estimates
\begin{align*}
\left\|\re^{\El_0 t}\left(\phi_0' v\right) - \phi_0' v\right\|_\infty &\lesssim \sqrt{t (1+t)} \, \left\|v'\right\|_\infty,\\
\left\|\re^{\left(d\partial_\xx^2 + a\partial_\xx\right) t} v - v\right\|_\infty &\lesssim \sqrt{t(1+t)} \, \left\|v'\right\|_\infty,\\
\left\|\re^{-\partial_\xx^4 t} v - v\right\|_\infty &\lesssim t^{\frac14} \left\|v'\right\|_\infty
\end{align*}
hold for all $t > 0$ and $v \in C_{\mathrm{ub}}^1(\R,\R)$. 
\end{proposition}
\begin{proof}
We employ~\cite[Proposition~2.1.4]{LUN} and use that $\phi_0'$ lies in the kernel of $\El(0)$ to infer
\begin{align*}
\re^{\El_0 t}\left(\phi_0' v\right) - \phi_0' v &= \int_0^t \re^{\El_0 s} \El_0 \left(\phi_0' v\right) \de s = \int_0^t \re^{\El_0 s} \left(k_0^2 D \left(\partial_\xx\left(\phi_0' v'\right) + \phi_0'' v'\right) + \omega_0 \phi_0' v'\right) \de s, \\
\re^{\left(d\partial_\xx^2 + a\partial_\xx\right) t} v - v &= \int_0^t \re^{\left(d\partial_\xx^2 + a\partial_\xx\right) s} \left(d\partial_\xx + a\right) v' \de s, \qquad
\re^{-\partial_\xx^4 t} w - w = \int_0^t \re^{-\partial_\xx^4 s} \partial_\xx^3 w' \de s
\end{align*}
for $v \in C_{\mathrm{ub}}^2(\R,\R)$, $w \in C_{\mathrm{ub}}^4(\R)$, and $t > 0$. Applying the estimates from Propositions~\ref{prop:lin1} and~\ref{prop:lin2} and the bound 
\begin{align*} \|\partial_\xx^3 \re^{-\partial_\xx^4 s} z\|_\infty \lesssim s^{-\frac{3}{4}} \|z\|_\infty, \qquad z \in C_{\mathrm{ub}}(\R), \, s > 0\end{align*} 
to the right-hand sides of the latter, we arrive at
\begin{align*}
\left\|\re^{\El_0 t}\left(\phi_0' v\right) - \phi_0' v\right\|_\infty, \left\|\re^{\left(d\partial_\xx^2 + a\partial_\xx\right) t} v - v\right\|_\infty &\lesssim \int_0^t \left(1 + \frac{1}{\sqrt{s}}\right) \|v'\|_\infty \de s\lesssim \sqrt{t(1+t)} \, \left\|v'\right\|_\infty,\\
\left\|\re^{-\partial_\xx^4 t} w - w\right\|_\infty &\lesssim \int_0^t s^{-\frac34}\left\|w'\right\|_\infty \de s \lesssim t^{\frac14} \left\|w'\right\|_\infty
\end{align*}
for $v \in C_{\mathrm{ub}}^2(\R,\R)$, $w \in C_{\mathrm{ub}}^4(\R)$, and $t > 0$. Thus, the result follows by density.
\end{proof}

The remaining estimates concern the action of the principal component of the semigroup on modulational data. 

\begin{proposition} \label{prop:lin_mod_2}
Assume~\ref{assH1} and~\ref{assD1}-\ref{assD3}. Let $g \in L^2_\mathrm{per}(0,1)$. Fix $l,m \in \NM_0$ with $l \geq 1$. Then, the estimates
\begin{align}
\left\|\left(\partial_t - a \partial_\zeta\right)^m \partial_\zeta^l S_p^0(t)\left(gv\right)\right\|_\infty &\lesssim \left(1 + t\right)^{-\frac{2m+l-1}{2}} \|v'\|_\infty, \label{e:modbound1}\\
\left\|\left(\partial_t - a\partial_\xx\right)^m S_p^0(t)\left(\phi_0' v\right) - \left(\partial_t^m \chi(t)\right) \re^{\left(d\partial_\xx^2 + a \partial_\xx\right) t} v\right\|_\infty &\lesssim \left\|v'\right\|_\infty \label{e:modbound2}
\end{align}
hold for $t \geq 0$ and $v \in C_{\mathrm{ub}}^1(\R,\R)$. 
\end{proposition}
\begin{proof}
We compute
\begin{align*}
\begin{split}
&\left(\partial_t - a\partial_\zeta\right)^m \partial_\zeta^l \int_\R G_p^0(\xx,\xt,t)g(\xt)v(\xt) \de \xt\\ 
&\qquad = \frac{\ri^l}{2\pi} \int_\R \int_\R \re^{\lambda_c(\xi) t} \xi^{l + 2m} \rho(\xi) \left(\frac{\lambda_c(\xi) - \ri a\xi}{\xi^2}\right)^m \widetilde{\Phi}_\xi(\xt)^* g(\xt) \re^{\ri \xi(\xx - \xt)} v(\xt) \de \xi \de \xt\\
\end{split}
\end{align*}
for $\xx \in \R$, $t \geq 1$, and $v \in C_{\mathrm{ub}}^1(\R,\R)$. Using Proposition~\ref{prop:speccons}, we observe that we can apply Lemma~\ref{lem:semigroupEstimate1} with $\lambda(\xi) = \lambda_c(\xi)$, $m_1 = l+2m-1 \geq 0$, $m_2 = 1$, and
\begin{align*}
F(\xi,\xx,\xt,t) = \rho(\xi) \left(\frac{\lambda_c(\xi) - \ri a\xi}{\xi^2}\right)^m \widetilde{\Phi}_\xi(\xt)^* g(\xt).
\end{align*}
This results in the bound~\eqref{e:modbound1} upon recalling the representation~\eqref{e:prin_rep} of the principal component of the semigroup and noting that $\chi'(t)$ is supported on $[1,2]$.

We proceed with proving the second estimate. First, applying the product rule to the representation~\eqref{e:prin_rep}, we establish
\begin{align}
\label{e:principal_decomp_bound}
\begin{split}
&\left\|(\partial_t - a\partial_\xx)^m S_p^0(t)\left(\phi_0' v\right) - \left(\partial_t^m \chi(t)\right) \re^{\left(d\partial_\xx^2 + a \partial_\xx\right) t} v\right\|_\infty\\ 
&\qquad \lesssim \left\|\int_\R G_p^0(\cdot,\xt,t)\phi_0'(\xt)v(\xt) \de \xt -  \re^{\left(d\partial_\xx^2 + a \partial_\xx\right) t} v\right\|_\infty\\ 
&\qquad \qquad + \, \sum_{\ell = 1}^m \left\|\left(\partial_t - a\partial_\zeta\right)^\ell \int_\R G_p^0(\cdot,\xt,t)\phi_0'(\xt)v(\xt) \de \xt\right\|_\infty =: J_1 + J_2
\end{split}
\end{align}
for $t \geq 1$ and $v \in C_{\mathrm{ub}}^1(\R,\R)$. Analogous to the proof of estimate~\eqref{e:modbound1}, we obtain the bound
\begin{align*}
|J_2| \lesssim \|v'\|_\infty
\end{align*}
for $t \geq 1$ and $v \in C_{\mathrm{ub}}^1(\R,\R)$. So, all that remains is to estimate $J_1$. To that end, we set $\lambda_r(\xi) = \lambda_c(\xi) - \ri a \xi + d \xi^2$ and decompose
\begin{align} \label{e:principal_decomp}
\int_\R G_p^0(\xx,\xt,t)\phi_0'(\xt)v(\xt) \de \xt -  \re^{\left(d\partial_\xx^2 + a \partial_\xx\right) t} v &= I + II + III + IV
\end{align}
for $\xx \in \R$, $t \geq 1$, and $v \in C_{\mathrm{ub}}^1(\R,\R)$, where we denote
\begin{align*}
I &= \frac{1}{2\pi} \int_\R \int_\R \re^{\lambda_c(\xi) t} \xi \rho(\xi) \frac{\widetilde{\Phi}_\xi(\xt)^* - \widetilde{\Phi}_0(\xt)^*}{\xi} \, \phi_0'(\xt) \re^{\ri\xi(\xx - \xt)} v(\xt) \de \xi \de \xt,\\
II &= \frac{1}{2\pi} \int_\R \int_\R \re^{\left(\ri a\xi - \frac{d}{2} \xi^2\right)t} \xi \rho(\xi) \re^{-\frac{d}{2} \xi^2 t} 
\, \frac{\re^{\lambda_r(\xi) t} - 1}{\xi} \, \widetilde{\Phi}_0(\xt)^* \phi_0'(\xt) \re^{\ri \xi(\xx - \xt)} v(\xt) \de \xi \de \xt,\\
III &= \frac{1}{2\pi} \int_\R \int_\R \re^{\left(\ri a\xi - d \xi^2\right)t} \rho(\xi) \left(\widetilde{\Phi}_0(\xt)^* \phi_0'(\xt) - 1\right) \re^{\ri \xi(\xx - \xt)} v(\xt) \de \xi \de \xt,\\
IV &= \frac{1}{2\pi} \int_\R \int_\R \re^{\ri\xi(\xx - \xt) + \left(\ri a\xi - d \xi^2\right)t} \left(\rho(\xi) - 1\right) v(\xt) \de \xi \de \xt,
\end{align*}
In the following, we subsequently bound the contributions $I$, $II$, $III$, and $IV$. First, we use Proposition~\ref{prop:speccons} and apply Lemma~\ref{lem:semigroupEstimate1} to $I$ with $\lambda(\xi) = \lambda_c(\xi)$, $m_1 = 0$, $m_2 = 1$, and
\begin{align*}
F(\xi,\xx,\xt,t) = \rho(\xi) \frac{\widetilde{\Phi}_\xi(\xt)^* - \widetilde{\Phi}_0(\xt)^*}{\xi} \, \phi_0'(\xt).
\end{align*}
Thus, we arrive at the bound
\begin{align*}
|I| \lesssim \|v'\|_\infty
\end{align*}
for $\xx \in \R$, $t \geq 1$, and $v \in C_{\mathrm{ub}}^1(\R,\R)$. 

We proceed with bounding $II$. We note that, by Proposition~\ref{prop:speccons}, there exists a constant $C > 0$ such that $|\lambda_r(\xi)| \leq C|\xi|^3$, $|\lambda_r'(\xi)| \leq C|\xi|^2$, and $|\lambda_r''(\xi)| \leq C|\xi|$ for $\xi \in (-\xi_0,\xi_0)$. Combining this with the identity $|\re^{z} - 1| \leq \re^{|z|}-1 \leq |z| \re^{|z|}$ for $z \in \C$, we establish
\begin{align*}
\left|\re^{-\frac{d}{2} \xi^2t} \, \frac{\re^{\lambda_r(\xi) t} - 1}{\xi}\right| &\lesssim \xi^2 t \re^{-\frac{d}{4} \xi^2t} \lesssim 1,\\
\left|\partial_\xi\left(\re^{-\frac{d}{2} \xi^2t} \, \frac{\re^{\lambda_r(\xi) t} - 1}{\xi}\right)\right| &\lesssim |\xi| t\left(1 + \xi^2 t\right) \re^{-\frac{d}{4} \xi^2t} \lesssim \sqrt{t},
\end{align*}
and
\begin{align*}
\left|\partial_\xi^2\left(\re^{-\frac{d}{2} \xi^2t} \, \frac{\re^{\lambda_r(\xi) t} - 1}{\xi}\right)\right| &\lesssim t\left(1 + \xi^2 t + \xi^4 t^2\right) \re^{-\frac{d}{4} \xi^2t}
\lesssim t
\end{align*}
for $\xi \in (-\xi_0,\xi_0)$ and $t \geq 1$.  Hence, applying Lemma~\ref{lem:semigroupEstimate1} with $\lambda(\xi) = a \ri \xi - \frac{d}{2} \xi^2$, $m_1 = 0$, $m_2 = 1$, and
\begin{align*}
F(\xi,\xx,\xt,t) = \rho(\xi) \re^{-\frac{d}{2} \xi^2 t} \frac{\re^{\lambda_r(\xi) t} - 1}{\xi} \, \widetilde{\Phi}_0(\xt)^* \phi_0'(\xt),
\end{align*}
we bound
\begin{align*}
|II| \lesssim \|v'\|_\infty.
\end{align*}
for $\xx \in \R$, $t \geq 1$, and $v \in C_{\mathrm{ub}}^1(\R,\R)$. 

To estimate the contribution $III$, we observe that we can apply Lemma~\ref{lem:semigroupEstimate1} with $\lambda(\xi) = a \ri \xi - d \xi^2$, $m_1 = 0$, $m_2 = 0$, and
\begin{align*}
F(\xi,\xx,\xt,t) = \rho(\xi) \left(\widetilde{\Phi}_0(\xt)^* \phi_0'(\xt) - 1\right),
\end{align*}
because it holds $\int_0^1 \widetilde{\Phi}_0(\xt)^* \phi_0'(\xt) \de \xt = \langle\widetilde{\Phi}_0,\phi_0'\rangle_{L^2(0,1)} = 1$ by~\eqref{e:adjoint}. Thus, we arrive at the estimate
\begin{align*}
|III| \lesssim \|v'\|_\infty
\end{align*}
for $\xx \in \R$, $t \geq 1$, and $v \in C_{\mathrm{ub}}^1(\R,\R)$.

Finally, we estimate the contribution $IV$. Integration by parts in $\xt$ yields
\begin{align*}
IV = \frac{\ri}{2\pi} \int_\R \int_\R \re^{\ri\xi(\xx - \xt) + \left(\ri a\xi - d \xi^2\right)t} \, \frac{\rho(\xi) - 1}{\xi} \, v'(\xt) \de \xi \de \xt
\end{align*}
for $\xx \in \R$, $t \geq 1$, and $v \in C_{\mathrm{ub}}^1(\R,\R)$. Hence, applying the high-frequency estimate from Lemma~\ref{lemma_higher_order_low} with $F(\xi) = (\rho(\xi) - 1)\xi^{-1}$, we find a constant $\mu_0 > 0$ such that
\begin{align*}
|IV| \lesssim \re^{-\mu_0 t} \|v'\|_\infty
\end{align*}
for $\xx \in \R$, $t \geq 1$, and $v \in C_{\mathrm{ub}}^1(\R,\R)$. Combining~\eqref{e:principal_decomp_bound} with~\eqref{e:principal_decomp}, and the estimates on $J_2$, $I$, $II$, $III$, and $IV$, we arrive at the bound~\eqref{e:modbound2}, which concludes the proof.
\end{proof}

\section{Nonlinear iteration scheme} \label{sec:itscheme}

In this section, we introduce the nonlinear iteration scheme that will be employed in~\S\ref{sec:nonlinearstab} to prove our nonlinear stability result, Theorem~\ref{main_theorem}. To this end, let $u_{\mathrm{wt}}(x,t) = \phi_0(k_0x - \omega_0 t)$ denote a diffusively spectrally stable wave-train solution of~\eqref{RD0}, satisfying assumptions~\ref{assH1} and~\ref{assD1}-\ref{assD3}. Take a perturbation $\vt_0 \in C_{\mathrm{ub}}(\R)$ and a phase modulation $\gamma_0 \in C_{\mathrm{ub}}^1(\R)$, and consider the solution $u(t)$ to~\eqref{RD} with initial condition $u(0) = u_0 \in C_{\mathrm{ub}}(\R)$ given by
\begin{align} \label{e:defu0}
u_0(\xx) = \phi_0(\xx + \gamma_0(\xx)) + \mathring{v}_0(\xx), \qquad \xx \in \R.
\end{align}
Assuming that $\mathring{v}_0$ and $\gamma_0'$ are sufficiently small in $C_{\mathrm{ub}}(\R)$, our aim is to construct a spatiotemporal modulation function $\gamma(t)$ with $\gamma(0) = \gamma_0$ such that the solution $u(t)$ to~\eqref{RD} can be written in the form
\begin{align*}
u(\xx,t) = \phi_0(\xx + \gamma(\xx,t)) + \mathring{v}(\xx,t) 
\end{align*}
where both $\gamma_\xx(t)$ and the remainder $\mathring{v}(t)$ stay small over time in $C_{\mathrm{ub}}(\R)$ and decay at diffusive rates. In particular, this precludes finite-time blow-up and implies that the solution $u(t)$ exists globally in time. 

A nonlinear iteration argument cannot be closed by directly estimating the \emph{forward-modulated perturbation} $\mathring{v}(t)$, due to insufficient decay in the nonlinear terms of its evolution equation; see~\cite{ZUM23} for a detailed discussion. As outlined in~\S\ref{sec:strategy}, our approach instead relies on the fact that the $L^\infty$-norm of $\mathring{v}(t)$ is equivalent to that of the \emph{inverse-modulated perturbation} $v(t)$, which is given by~\eqref{e:defv}. Taking inspiration from~\cite{JONZNL}, we then choose a phase modulation $\gamma(t)$, which captures the most critical terms in the Duhamel formula for the inverse-modulated perturbation $v(t)$ and satisfies $\gamma(0) = \gamma_0$. As in~\cite{BjoernMod,JONZW,SAN3}, we find that $\gamma(t)$ obeys a perturbed viscous Hamilton-Jacobi equation. After eliminating its dominant nonlinear term via the Cole-Hopf transformation, an $L^\infty$-based nonlinear iteration argument involving $\gamma(t)$ and $v(t)$ can be closed. 

However, due to the quasilinear nature of the evolution equation for the inverse-modulated perturbation $v(t)$, an apparent loss of regularity must be addressed. To control regularity, we distinguish between short and long times. For short times, we estimate the forward-modulated perturbation $\vt(t)$ iteratively using its Duhamel representation, which is of semilinear nature and does not suffer from a loss of derivatives. This allows us to control regularity of the inverse-modulated perturbation $v(t)$ by relating $\mathring{v}(t)$ and $v(t)$, including their derivatives, through mean-value type estimates; see~\cite{ZUM23}. However, the obtained decay of $\mathring{v}(t)$ and its derivatives is too slow to provide effective control for large times. To overcome this, we follow the approach in~\cite{AdR1} and employ forward-modulated damping estimates. Specifically, we derive a nonlinear damping estimate in uniformly local Sobolev spaces for the \emph{modified forward-modulated perturbation}~\eqref{e:defringz}, which also satisfies a semilinear equation without derivative loss. This energy estimate provides control over the $L^\infty$-norms of derivatives of $\mathring{z}(t)$ in terms of the $L^\infty$-norm of $\mathring{z}(t)$ itself. To close the nonlinear argument, we then relate $\mathring{z}(t)$ to the residual
\begin{align} \label{e:defz}
z(t) = v(t) - k_0 \partial_k \phi(\cdot;1)\, \gamma_\xx(t),
\end{align}
along with their spatial derivatives, again via mean-value type arguments. As a result, derivatives of $v(t)$ are ultimately controlled by the $L^\infty$-norm of $\mathring{z}(t)$ (and hence of $z(t)$). We remark that short-time regularity control through the nonlinear damping estimate on $\mathring{z}(t)$ is not feasible due to the presence of the term $\gamma_{\zeta t}(t)$ in the nonlinearity of the evolution equation for $\mathring{z}(t)$; see~\S\ref{sec:forward_modulated}. This term exhibits a non-integrable blow-up as $t \downarrow 0$, with $L^\infty$-norm scaling like $t^{-1}$, thereby precluding control in the short-time regime.

The remainder of this section is organized as follows. We begin by establishing local existence and uniqueness of the solution $u(t)$ to~\eqref{RD}. We then derive the equation satisfied by the inverse-modulated perturbation $v(t)$ and obtain $L^\infty$-bounds on its nonlinear terms. Next, we introduce the phase modulation $\gamma(t)$ and derive the perturbed viscous Hamilton-Jacobi equation that governs its dynamics. Finally, we formulate the equations for the forward-modulated perturbation $\vt(t)$ and the modified forward-modulated perturbation $\mathring{z}(t)$, which are used for short- and long-time regularity control, respectively, and we establish a nonlinear damping estimate on $\zt(t)$.

\subsection{Local existence and uniqueness of the solution}

Since the advection-diffusion operator $L_0 = k_0^2 D \partial_{\xx\xx} + \omega_0 \partial_\xx$ is sectorial on $\smash{C_{\mathrm{ub}}^l(\R)}$ with dense domain $D(L_0) = \smash{C_{\mathrm{ub}}^{l+2}(\R)}$ by~\cite[Corollary~3.1.9]{LUN} and the nonlinear map $u \mapsto \smash{\partial_u^j f(u)}$ is locally Lipschitz continuous on $\smash{C_{\mathrm{ub}}^l(\R)}$ for any $j,l \in \NM_0$, local existence and uniqueness of the solution $u(t)$ to the reaction-diffusion system~\eqref{RD} follows directly from standard analytic semigroup theory; see~\cite[Theorem~7.1.5 and Propositions~7.1.8 and~7.1.10]{LUN}.

\begin{proposition} \label{well_posed_full_sol}
Assume~\ref{assH1}. Let $u_0 \in C_{\mathrm{ub}}(\R)$. Then, there exists a maximal time $T_{\max} \in (0,\infty]$ such that~\eqref{RD} admits a unique classical solution
\begin{align*} u \in C\big([0,T_{\max}),C_{\mathrm{ub}}(\R)\big) \cap C\big((0,T_{\max}),C_{\mathrm{ub}}^2(\R)\big) \cap C^1\big((0,T_{\max}),C_{\mathrm{ub}}(\R)\big), \end{align*}
with initial condition $u(0) = u_0$. Moreover, the map $[0,T_{\max}) \to C_{\mathrm{ub}}(\R), t \mapsto \sqrt{t} \, u_\xx(t)$ is continuous and, if $T_{\max} < \infty$, then we have
\begin{align} \label{e:blowupu}
\limsup_{t \uparrow T_{\max}} \left\|u(t)\right\|_\infty = \infty. \end{align}
Finally, for any $j,l \in \NM_0$ we have $u \in C^j\big((0,T_{\max}),C_{\mathrm{ub}}^l(\R)\big)$.
\end{proposition}

\subsection{Inverse-modulated perturbation equation} 
\label{sec:inv_mod_pert}

Using that $u(t)$ and $\phi_0$ are solutions to~\eqref{RD}, one finds that the inverse-modulated perturbation, given by~\eqref{e:defv}, satisfies the quasilinear equation
\begin{align}
\left(\partial_t - \El_0\right)\left[v + \phi_0' \gamma - \gamma_\xx v\right] = \Non(v,\gamma,\partial_t \gamma), \label{e:modpertbeq}
\end{align}
where the nonlinearity $\Non$ is given by
\begin{align} \label{e:defnonl}
\Non(v,\gamma,\gamma_t) &= \mathcal Q(v,\gamma) + \partial_\xx \mathcal R(v,\gamma,\gamma_t) + \partial_{\xx}^2 \mathcal S(v,\gamma)
\end{align}
with
\begin{align}
\begin{split}
\mathcal Q(v,\gamma) &= \left(f(\phi_0+v) - f(\phi_0) - f'(\phi_0) v\right)\left(1-\gamma_\xx\right),\\
\mathcal R(v,\gamma,\gamma_t) &= -\gamma_t v + \omega_0 \gamma_\xx v + \frac{k_0^2}{1-\gamma_{\xx}} D \left(\gamma_\xx^2 \phi_0' - \frac{\gamma_{\xx\xx} v}{1-\gamma_\xx}\right),\\
\mathcal S(v,\gamma) &= k_0^2D \left(2\gamma_\xx v + \frac{\gamma_\xx^2 v}{1-\gamma_\xx}\right).
\end{split} \label{e:defnonl0}
\end{align}
We refer to~\cite[Lemma~4.2]{JONZ} for a detailed derivation of~\eqref{e:modpertbeq}.

The nonlinearities obey the following $L^\infty$-bounds.

\begin{lemma} \label{lemma_nonlinear_bound_on_N}
Assume~\ref{assH1}. Fix a constant $C > 0$. Then, we have
\begin{align*}
\left\|\mathcal Q(v,\gamma)\right\|_\infty &\lesssim \|v\|_\infty^2,\\
\left\|\mathcal R(v,\gamma,\gamma_t)\right\|_\infty &\lesssim \|v\|_\infty \|(\gamma_\xx,\gamma_t)\|_{C_{\mathrm{ub}}^1 \times C_{\mathrm{ub}}} + \|\gamma_\xx\|_\infty^2,\\
\left\|\mathcal S(v,\gamma)\right\|_\infty &\lesssim \|v\|_{\infty} \|\gamma_\xx\|_{\infty},\\
\left\|\partial_\xx \mathcal S(w,\gamma)\right\|_\infty &\lesssim \|w\|_{C_{\mathrm{ub}}^1} \|\gamma_\xx\|_{C_{\mathrm{ub}}^1}
\end{align*}
for $v \in C_{\mathrm{ub}}(\R)$, $w \in C_{\mathrm{ub}}^1(\R)$, and $(\gamma,\gamma_t) \in C_{\mathrm{ub}}^2(\R) \times C_{\mathrm{ub}}(\R)$ satisfying $\|v\|_\infty \leq C$ and $\|\gamma_\xx\|_{\infty} \leq \frac{1}{2}$.
\end{lemma}

\subsection{Choice of phase modulation} \label{sec:phase_mod}

Assuming that $\gamma(t)$ satisfies $\gamma(0) = \gamma_0$, we arrive at the Duhamel formulation
\begin{align}
v(t) + \phi_0'\gamma(t) = \re^{\El_0 t}\left(v_0 + \phi_0' \gamma_0 - \gamma_0' v_0\right) + \int_0^t \re^{\El_0(t-s)}\mathcal{N}(v(s),\gamma(s),\partial_t \gamma(s))\de s + \gamma_\xx(t)v(t) \label{e:intv}
\end{align}
after integrating~\eqref{e:modpertbeq}, where we denote
\begin{align} \label{e:defv0}
v_0(\xx) = u_0(\xx- \gamma_0(\xx)) - \phi_0(\xx), \qquad \xx \in \R.
\end{align}

As in~\cite{JONZNL,JUNNL}, we make a judicious choice for the phase modulation $\gamma(t)$ such that the linear term $\phi_0' \gamma(t)$ on the left-hand side of~\eqref{e:intv} compensates for the slowest decaying nonlinear contributions on the right-hand side of~\eqref{e:intv}. This choice must be compatible with the initial condition $\gamma(0) = \gamma_0 \in C_{\mathrm{ub}}^1(\R)$. Moreover, to ensure sufficient regularity for the subsequent nonlinear iteration scheme, we require that $\gamma(t)$ is smooth for all $t > 0$. Thus, motivated by the semigroup decomposition~\eqref{e:decomp_full_semigroup}, we define $\gamma(t)$ through the integral equation
\begin{align}
\begin{split}
\gamma(t) &= S_p^0(t)\left(v_0 + \phi_0'\gamma_0 + \gamma_0' v_0\right) + \int_0^t S_p^0(t-s) \mathcal{N}(v(s),\gamma(s),\partial_t \gamma(s))\de s\\ &\qquad + \, \left(1 - \chi(t)\right) \re^{-\partial_\xx^4 t} \gamma_0, 
\end{split}
\label{e:intgamma}
\end{align}
where $\chi(t)$ is the temporal cut-off function from Proposition~\ref{prop:lin1}. 

Before establishing existence of a solution $\gamma(t)$ to~\eqref{e:intgamma}, we argue that our implicit choice for $\gamma(t)$ indeed possesses the required properties. To begin with, we have $\gamma(0) = \gamma_0$, since $S_p^0(t)$ vanishes on $[0,1]$. Second, since the sectorial operator $-\partial_\xx^4$ generates an analytic semigroup on $C_{\mathrm{ub}}(\R)$, cf.~\cite[Proposition~2.4.4 and Corollary~3.1.9]{LUN}, and the propagator $S_p^0(t)$ is smoothing by Proposition~\ref{prop:lin1}, it follows that $\gamma(t)$ is smooth for all $t > 0$. The choice of the analytic semigroup $\smash{\re^{-\partial_\xx^4 t}}$, rather than the standard heat semigroup $\smash{\re^{\partial_\xx^2 t}}$, is motivated by the need to control third-order spatial derivatives of $\gamma(t)$ at short times within our nonlinear argument; see Remark~\ref{rem:Gamma}.

Substituting~\eqref{e:decomp_full_semigroup},~\eqref{e:defz}, and~\eqref{e:intgamma} into~\eqref{e:intv}, we obtain the Duhamel formulation
\begin{align}
\begin{split}
z(t) &= \widetilde{S}(t)\left(v_0 + \phi_0'\gamma_0 + \gamma_0' v_0\right) +\int_0^t\widetilde{S}(t-s)\mathcal{N}(v(s),\gamma(s),\partial_t \gamma(s)) \de s + \gamma_\xx(t)v(t)\\ 
&\qquad - \, \left(1 - \chi(t)\right)\left(\phi_0' + k_0 \partial_k \phi(\cdot;k_0) \partial_\xx\right) \re^{-\partial_\xx^4 t} \gamma_0
\end{split}\label{e:intz}
\end{align}
for the residual $z(t)$. Recalling that the propagator $\widetilde{S}(t)$ decays at rate $t^{-1}$ as $t \to \infty$, cf.~Proposition~\ref{prop:lin1}, we confirm that our choice of $\gamma(t)$ indeed eliminates the slowest decaying contributions on the right-hand side of~\eqref{e:intv}.

\begin{remark}\label{rem:Gamma} { \upshape
Our choice for the analytic semigroup $\smash{\re^{-\partial_\xx^4 t}}$ in~\eqref{e:intgamma} is motivated by the structure of the nonlinear terms in the perturbed viscous Hamilton-Jacobi equation, which governs the leading-order dynamics of the phase modulation $\gamma(t)$; see~\S\ref{sec:visc_HamJac}. In particular, the nonlinearity involves a temporal derivative and up to three spatial derivatives of $\gamma(t)$. To control these terms in the associated Duhamel formulation, it is crucial that the $C_{\mathrm{ub}}^2$-norm of $\gamma_\xx(t)$, as well as the $L^\infty$-norm of $\gamma_t(t)$, exists for $t > 0$ and exhibits a singularity as $t \downarrow 0$ that remains integrable in time. Since we merely assume $\gamma_0 \in C_{\mathrm{ub}}^1(\R)$, this motivates the introduction of $\smash{\re^{-\partial_\xx^4 t} \gamma_0}$. Indeed, for fixed $j \in \NM_0$, standard analytic semigroup theory~\cite{LUN} yields the estimates 
\begin{align} \label{e:Gamma_rates3}
\begin{split}
\left\|\re^{-\partial_\xx^4 t} \gamma\right\|_\infty &\leq \|\gamma\|_\infty, \qquad \left\|\partial_t \partial_\xx^j \re^{-\partial_\xx^4 t} \gamma\right\|_\infty \lesssim t^{-\frac{3+j}{4}} \|\gamma'\|_\infty, \qquad \left\|\partial_\xx^{1 + j} \re^{-\partial_\xx^4 t} \gamma\right\|_\infty \lesssim t^{-\frac{j}{4}} \|\gamma'\|_\infty
\end{split}
\end{align}
for $t > 0$ and $\gamma \in C_{\mathrm{ub}}^1(\R)$, which show that the temporal derivative and the first four spatial derivatives of $\smash{\re^{-\partial_\xx^4 t} \gamma_0}$ permit such integrable blow-up behavior as $t \downarrow 0$. So, the term $\smash{\re^{-\partial_\xx^4 t} \gamma_0}$ is tailored to the demands of our nonlinear argument. Alternatively, one could impose higher regularity on the initial phase modulation $\gamma_0$, which allows replacing $\smash{\re^{-\partial_\xx^4 t} \gamma_0}$ with $\gamma_0$ in~\eqref{e:intgamma}, as done in~\cite{JONZNL}.
}\end{remark}

We conclude this subsection by establishing local existence of a solution $\gamma(t)$ to the integral equation~\eqref{e:intgamma}. The contraction-mapping argument is adapted from~\cite[Proposition~4.4]{BjoernMod}, accommodating large initial data $\gamma_0 \in C_{\mathrm{ub}}^1(\R)$, while still requiring that $\gamma_\xx(t)$ remains sufficiently small to ensure that the nonlinearity in~\eqref{e:defnonl} is well-defined.

\begin{proposition} \label{p:gamma}
Assume~\ref{assH1} and~\ref{assD3}. Let $r_0 \in (0,\frac12)$ be as in Proposition~\ref{prop:family}. Let $\mathring{v}_0 \in C_{\mathrm{ub}}(\R)$ and $\gamma_0 \in C_{\mathrm{ub}}^1(\R)$ with $\|\gamma_0'\|_\infty < r_0$. Define $u_0, v_0 \in C_{\mathrm{ub}}(\R)$ by~\eqref{e:defu0} and~\eqref{e:defv0}, respectively. For $u$ and $T_{\max}$ as in Proposition~\ref{well_posed_full_sol}, there exists a maximal time $\tau_{\max} \in \left[1,\max\{1,T_{\max}\}\right]$ such that~\eqref{e:intgamma}, with $v$ given by~\eqref{e:defv}, has a solution
\begin{align} \label{e:classicalgamma} \gamma \in C\big([0,\tau_{\max}),C_{\mathrm{ub}}^1(\R)\big) \cap C^j\big((0,\tau_{\max}),C_{\mathrm{ub}}^l(\R)\big), \qquad j,l \in \NM_0, \end{align}
satisfying $\gamma(t) = \re^{-\partial_\xx^4 t} \gamma_0$ for $t \in [0,1]$. In addition, it holds $\|\gamma_\xx(t)\|_\infty < r_0$ for all $t \in [0,\tau_{\max})$. Finally, $\tau_{\max} < T_{\max}$ implies 
\begin{align} \label{e:blowupgamma1}
\limsup_{t \uparrow \tau_{\max}} \left\|\gamma_\xx(t)\right\|_\infty = r_0
\end{align}
or
\begin{align} \label{e:blowupgamma2}
\limsup_{t \uparrow \tau_{\max}} \left\|\left(\gamma(t),\partial_t \gamma(t)\right)\right\|_{C_{\mathrm{ub}}^2 \times C_{\mathrm{ub}}} = \infty.\end{align}
\end{proposition}
\begin{proof}
We delegate the proof to Appendix~\ref{app:B}.
\end{proof}

Local existence and regularity of the inverse-modulated perturbation $v(t)$ and the residual $z(t)$ follow directly from Propositions~\ref{well_posed_full_sol} and~\ref{p:gamma}.

\begin{corollary} \label{C:local_v}
Assume~\ref{assH1} and~\ref{assD3}. Let $r_0 \in (0,\frac12)$ be as in Proposition~\ref{prop:family}. Let $\mathring{v}_0 \in C_{\mathrm{ub}}(\R)$ and $\gamma_0 \in C_{\mathrm{ub}}^1(\R)$ with $\|\gamma_0'\|_\infty < r_0$. Define $u_0, v_0 \in C_{\mathrm{ub}}(\R)$ by~\eqref{e:defu0} and~\eqref{e:defv0}, respectively. For $u$ as in Proposition~\ref{well_posed_full_sol} and $\gamma$ and $\tau_{\max}$ as in Proposition~\ref{p:gamma}, the inverse-modulated perturbation $v$, defined by~\eqref{e:defv}, and the residual $z$, defined by~\eqref{e:defz}, satisfy
\begin{align*}
v,z \in C\big([0,\tau_{\max}),C_{\mathrm{ub}}(\R)\big) \cap C\big((0,\tau_{\max}),C_{\mathrm{ub}}^2(\R)\big) \cap C^1\big((0,\tau_{\max}),C_{\mathrm{ub}}(\R)\big).
\end{align*}
Moreover, the map $[0,\tau_{\max}) \to C_{\mathrm{ub}}(\R), \, t \mapsto \sqrt{t} \,\big(v_\xx(t),z_\xx(t)\big)$ is continuous. Finally, the Duhamel formulations~\eqref{e:intv} and~\eqref{e:intz} are valid for $t \in [0,\tau_{\max})$.
\end{corollary}

\subsection{Derivation of viscous Hamilton-Jacobi equation} \label{sec:visc_HamJac}

Proceeding along the lines of~\cite{BjoernMod}, we derive a perturbed viscous Hamilton-Jacobi equation governing the dynamics of the phase modulation $\gamma(t)$.

As a first step, we isolate all contributions involving $\gamma_\zeta^2$ in the nonlinearity in~\eqref{e:intgamma}. These terms are critical, as they exhibit the slowest decay. As shown in~\cite[Section~4.3]{BjoernMod}, the nonlinearity~\eqref{e:defnonl} admits the decomposition
\begin{align} \label{e:nonldecomp}
\Non(v(s),\gamma(s),\partial_t \gamma(s)) = k_0^2 f_p \gamma_{\xx}(s)^2 + \mathcal N_p(z(s),v(s),\gamma(s),\widetilde{\gamma}(s)),
\end{align}
where we denote
\begin{align*} \widetilde{\gamma}(t) := \partial_t \gamma(t) - a\gamma_\xx(t),\end{align*}
$f_p$ is the $1$-periodic function
\begin{align*}
f_p &= \frac{1}{2} f''(\phi_0)\left(\partial_k \phi(\cdot;k_0),\partial_k \phi(\cdot;k_0)\right) + \omega'(k_0) \partial_{\xx k} \phi(\cdot;k_0) + D\left(\phi_0'' + 2k_0 \partial_{\xx\xx k} \phi(\cdot;k_0)\right),
\end{align*}
and the residual is given by
\begin{align} \label{e:defNp}
\mathcal N_p(z,v,\gamma,\widetilde{\gamma}) &= \mathcal Q_p(z,v,\gamma) + \partial_\xx \mathcal R_p(z,v,\gamma,\widetilde{\gamma}) + \partial_{\xx\xx} \mathcal S_p(z,v,\gamma),
\end{align}
with
\begin{align*}
\mathcal Q_p(z,v,\gamma) &= \left(f(\phi_0+v) - f(\phi_0) - f'(\phi_0) v\right)\gamma_\xx + f(\phi_0+v) - f(\phi_0) - f'(\phi_0) v - \frac{1}{2} f''(\phi_0)(v,v)\\
&\qquad + \frac{1}{2}f''(\phi_0)(z,z) + k_0 \gamma_{\xx} f''(\phi_0)(z,\partial_k \phi(\cdot;k_0)) + 2k_0^2\omega'(k_0) \gamma_\xx \gamma_{\xx\xx} \partial_k \phi(\cdot;k_0)\\
&\qquad + \, 2k_0^2 D \left(\gamma_\xx \gamma_{\xx\xx} \left(\phi_0' + 4k_0\partial_{\xx k} \phi(\cdot;k_0)\right) + 2k_0\left(\gamma_{\xx\xx}^2 + \gamma_\xx \gamma_{\xx\xx\xx}\right) \partial_k \phi(\cdot;k_0)\right),\\
\mathcal R_p(z,v,\gamma,\widetilde{\gamma}) &= -v \widetilde{\gamma} + k_0\omega_0'(k_0) \gamma_{\xx} z + \frac{k_0^2}{1-\gamma_\xx}D\left(\gamma_\xx^3 \phi_0' - \frac{\gamma_{\xx\xx} v}{1-\gamma_\xx}\right),\\
\mathcal S_p(z,v,\gamma) &= k_0^2D \left(2\gamma_\xx z + \frac{\gamma_\xx^2 v}{1-\gamma_\xx}\right).
\end{align*}

Applying Taylor's theorem, we obtain the following nonlinear estimate.

\begin{lemma} \label{lem:nlboundsmod3}
Assume~\ref{assH1} and~\ref{assD3}. Fix a constant $C > 0$. Then, we have
\begin{align*}
\begin{split}
\|\mathcal Q_p(z,v,\gamma)\|_\infty &\lesssim \left(\|v\|_\infty + \|\gamma_\xx\|_\infty\right)\|v\|_\infty^2 + \left(\|z\|_\infty + \|\gamma_\xx\|_\infty\right)\|z\|_\infty + \|\gamma_\xx\|_{C_{\mathrm{ub}}^1}\|\gamma_{\xx\xx}\|_{C_{\mathrm{ub}}^1},\\
\|\mathcal R_p(z,v,\gamma,\widetilde{\gamma})\|_\infty &\lesssim \|v\|_\infty \left(\|\widetilde{\gamma}\|_\infty + \|\gamma_{\xx\xx}\|_\infty\right) + \|z\|_\infty \|\gamma_\xx\|_\infty + \|\gamma_\xx\|^3_\infty,\\
\|\mathcal S_p(z,v,\gamma)\|_\infty &\lesssim \|v\|_\infty\|\gamma_\xx\|^2_\infty + \|z\|_\infty \|\gamma_\xx\|_\infty
\end{split}
\end{align*}
for $z \in C_{\mathrm{ub}}(\R)$, $v \in C_{\mathrm{ub}}(\R)$, and $(\gamma,\widetilde{\gamma}) \in C_{\mathrm{ub}}^3(\R) \times C_{\mathrm{ub}}(\R)$ satisfying $\|v\|_\infty \leq C$ and $\|\gamma_\xx\|_\infty \leq \frac{1}{2}$.
\end{lemma}

Next, we decompose the phase modulation $\gamma(t)$ into a solution $\widetilde{y}(t)$ solving a perturbed viscous Hamilton-Jacobi equation and a remainder $r(t)$ exhibiting higher-order decay. To this end, we use the commutator identities~\eqref{e:commu} and insert the decomposition~\eqref{e:prin_rep1} of the propagator $S_p^i(t)$ and the decomposition~\eqref{e:nonldecomp} of the nonlinearity into~\eqref{e:intgamma}. Thus, after using~\eqref{e:adjoint} and~\eqref{e:prin_rep2} to reexpress $\smash{\re^{\left(d\partial_\xx^2 - c_g\partial_\xx\right) t} (\widetilde{\Phi}_0^*\mathbf{f}_p \gamma_\xx^2)}$ and setting
\begin{align*}
\nu = \langle \widetilde{\Phi}_0, \mathbf{f}_p\rangle_{L^2(0,1)},
\end{align*}
we arrive at the decomposition
\begin{align} \label{e:decomp_gamma}
\gamma(t) = \widetilde{y}(t) + r(t)
\end{align}
with
\begin{align}
\begin{split}
 \widetilde{y}(t) &= S_h^0(t) \left(v_0 + \gamma_0' v_0\right) + 
 \re^{\left(d\partial_\xx^2 + a\partial_\xx\right) t}\left(\gamma_0 - A_h(\phi_0') \gamma_0'\right)\\
 &\qquad + \, \int_0^t \re^{\left(d\partial_\xx^2 + a\partial_\xx\right) (t-s)} \left(\nu \gamma_\xx(s)^2 - k_0^2 A_h(f_p) \partial_\xx \left(\gamma_\xx(s)^2\right)\right) \de s\\
 &\qquad + \, \int_0^t S_h^0(t-s) \mathcal Q_p(z(s),v(s),\gamma(s))\de s
- \int_0^t S_h^1(t-s) \mathcal R_p(z(s),v(s),\gamma(s),\widetilde{\gamma}(s)) \de s\\
 &\qquad + \, \int_0^t S_h^2(t-s) \mathcal S_p(z(s),v(s),\gamma(s)) \de s, 
\end{split}
\label{e:def_tilde_y}
\end{align}
and remainder
\begin{align}
\begin{split}
    r(t) &= \widetilde{S}_r^0(t) \left(v_0 + \phi_0'\gamma_0 + \gamma_0' v_0\right) + \partial_\xx \re^{\left(d\partial_\xx^2 + a\partial_\xx\right) t} \left(A_h(\phi_0') \gamma_0\right) + k_0^2 \int_0^t \widetilde{S}_r^0(t-s)\left(f_p \gamma_\xx(s)^2\right) \de s\\
&\qquad + \, k_0^2 \partial_\xx \int_0^t \re^{\left(d\partial_\xx^2 + a\partial_\xx\right) (t-s)}\left(A_h(f_p) \gamma_\xx(s)^2\right) + \int_0^t \widetilde{S}_r^0(t-s) \mathcal Q_p(z(s),v(s),\gamma(s)) \de s\\
&\qquad - \, \int_0^t \widetilde{S}_r^1(t-s) \mathcal R_p(z(s),v(s),\gamma(s),\widetilde{\gamma}(s)) \de s
+ \int_0^t \widetilde{S}_r^2(t-s) \mathcal S_p(z(s),v(s),\gamma(s)) \de s\\
&\qquad + \, \partial_\xx \int_0^t S_p^0(t-s) \mathcal R_p(z(s),v(s),\gamma(s),\widetilde{\gamma}(s)) \de s + (1-\chi(t)) \re^{-\partial_\xx^4 t} \gamma_0\\ 
&\qquad + \, \partial_\xx^2 \int_0^t S_p^0(t-s) \mathcal S_p(z(s),v(s),\gamma(s)) \de s - 2\partial_\xx \int_0^t S_p^1(t-s) \mathcal S_p(z(s),v(s),\gamma(s)) \de s.
    \end{split}
    \label{e:intr}
\end{align}
Invoking the linear estimates from Propositions~\ref{prop:lin1} and~\ref{prop:lin2}, we observe that $r(t)$ accounts for the contributions in~\eqref{e:intgamma} that decay on the linear level. On the other hand, applying the convective heat operator $\partial_t - d\partial_\xx^2 - a\partial_\xx$ to~\eqref{e:def_tilde_y}, we arrive at the viscous Hamilton-Jacobi equation
\begin{align} \label{e:hamjac}
\left(\partial_t - d\partial_\xx^2 - a\partial_\xx\right) \widetilde{y} = \nu \widetilde{y}_\xx^2 + G(r,\widetilde{y},z,v,\gamma,\widetilde{\gamma})
\end{align}
with perturbation
\begin{align*}
G(z,v,\gamma,\widetilde{\gamma}) &= 2\nu\widetilde{y}_\xx r_\xx + \nu r_\xx^2 - k_0^2 A_h(f_p) \partial_\xx (\gamma_\xx^2) + \widetilde{\Phi}_0^*\mathcal Q_p(z,v,\gamma) - \left(\partial_\xx \widetilde{\Phi}_0^*\right)\mathcal R_p(z,v,\gamma,\widetilde{\gamma})\\ &\qquad + \, \left(\partial_\xx^2 \widetilde{\Phi}_0^*\right)\mathcal S_p(z,v,\gamma).
\end{align*}
That is, $\widetilde{y}(t)$ is a solution to~\eqref{e:hamjac} with initial condition
\begin{align*}
\widetilde{y}(0) = \gamma_0 - A_h(\phi_0') \gamma_0' + \widetilde{\Phi}_0^*\left(v_0 + \gamma_0' v_0\right) \in C_{\mathrm{ub}}(\R).
\end{align*}

By expressing the right-hand side of~\eqref{e:hamjac} as
\begin{align*}
F(z,v,\gamma,\widetilde{\gamma}) = \nu \gamma_\xx^2 - k_0^2 A_h(f_p) \partial_\xx (\gamma_\xx^2) + \widetilde{\Phi}_0^*\mathcal Q_p(z,v,\gamma) - \left(\partial_\xx \widetilde{\Phi}_0^*\right)\mathcal R_p(z,v,\gamma,\widetilde{\gamma}) + \left(\partial_\xx^2 \widetilde{\Phi}_0^*\right)\mathcal S_p(z,v,\gamma),
\end{align*}
we can interpret equation~\eqref{e:hamjac} as an inhomogeneous linear parabolic problem. The inhomogeneity $t \mapsto F(z(t),v(t),\gamma(t),\widetilde{\gamma}(t))$ belongs to $C\big((0,\tau_{\max}),C_{\mathrm{ub}}^1(\R)\big) \cap L^1\big((0,\tau_{\max}),C_{\mathrm{ub}}(\R)\big)$ by combining the estimates~\eqref{e:Gamma_rates} with Proposition~\ref{p:gamma} and Corollary~\ref{C:local_v}. Consequently, regularity properties of $\widetilde{y}(t)$ and $r(t) = \gamma(t) - \widetilde{y}(t)$ follow from analytic semigroup theory, cf.~\cite[Theorem~4.3.11]{LUN}.

\begin{corollary} \label{C:local_r}
Assume~\ref{assH1} and~\ref{assD3}. Let $r_0 \in (0,\frac12)$ be as in Proposition~\ref{prop:family}. Let $\mathring{v}_0 \in C_{\mathrm{ub}}(\R)$ and $\gamma_0 \in C_{\mathrm{ub}}^1(\R)$ with $\|\gamma_0'\|_\infty < r_0$. Define $u_0, v_0 \in C_{\mathrm{ub}}(\R)$ by~\eqref{e:defu0} and~\eqref{e:defv0}, respectively. For $\gamma$ and $\tau_{\max}$ as in Proposition~\ref{p:gamma} and and for $v$ and $z$ as in Corollary~\ref{C:local_v}, the Hamilton-Jacobi variable $\widetilde{y}$, given by~\eqref{e:def_tilde_y} and the residual $r$, given by~\eqref{e:intr}, obey
\begin{align*}
\widetilde{y},r \in C\big([0,\tau_{\max}),C_{\mathrm{ub}}(\R)\big) \cap C\big((0,\tau_{\max}),C_{\mathrm{ub}}^2(\R)\big) \cap C^1\big((0,\tau_{\max}),C_{\mathrm{ub}}(\R)\big).
\end{align*}
In addition, the map $[0,\tau_{\max}) \to C_{\mathrm{ub}}(\R) \times C_{\mathrm{ub}}(\R), t \mapsto \sqrt{t} \, \big(\widetilde{y}_\xx(t),r_\xx(t)\big)$ is continuous.
\end{corollary}

\begin{remark}{ \upshape
The decomposition~\eqref{e:decomp_gamma} generalizes the corresponding one in~\cite{BjoernMod}, where the initial phase modulation was assumed to vanish identically. Indeed, setting $\gamma_0 = 0$ in~\eqref{e:def_tilde_y} and~\eqref{e:intr} recovers the decomposition used in~\cite{BjoernMod}. Alternatively, by assuming more regular initial data $\mathring{v}_0 \in C_{\mathrm{ub}}^2(\R)$, one could employ the simpler decomposition from~\cite{AdR1}, thereby avoiding the commutator identities~\eqref{e:commu} needed to shift derivatives from the nonlinearity~\eqref{e:defNp} onto the propagators.}
\end{remark}

\subsection{Application of the Cole-Hopf transform} 

In the forthcoming nonlinear analysis, we control the remainder $r(t)$ by estimating the right-hand side of~\eqref{e:intr}. For short times, a similar approach allows us to bound $\widetilde{y}(t)$ via~\eqref{e:def_tilde_y}. However, for large times, this strategy fails due to the presence of the critical nonlinear term $\nu y_\xx(s)^2$ in~\eqref{e:def_tilde_y}, which decays at the nonintegrable rate $(1+s)^{-1}$ as $s \to \infty$. To address this issue, we proceed as in~\cite{AdR1,BjoernMod} and eliminate the problematic nonlinear term $\nu \widetilde{y}_\xx^2$ in~\eqref{e:hamjac} by applying the Cole-Hopf transformation. To this end, we introduce the new variable
\begin{align} \label{e:defy} y(t) = \re^{\frac{\nu}{d} \widetilde{y}(t)},\end{align}
which satisfies
\begin{align}\label{e:regy} y \in C\big([0,\tau_{\max}),C_{\mathrm{ub}}(\R)\big) \cap C\big((0,\tau_{\max}),C_{\mathrm{ub}}^2(\R)\big) \cap C^1\big((0,\tau_{\max}),C_{\mathrm{ub}}(\R)\big)\end{align} 
by Corollary~\ref{C:local_r}. We obtain the convective heat equation
\begin{align} \left(\partial_t - d\partial_\xx^2 - a\partial_\xx\right)y = 2\nu r_\xx y_\xx + \frac{\nu}{d}\left(\nu r_\xx^2 + G(z,v,\gamma,\widetilde{\gamma})\right)y\label{e:colehopf} \end{align}
in which critical $y_\xx^2$-contributions are  no longer present. Integrating~\eqref{e:colehopf}, we arrive at the Duhamel formulation
\begin{align}
\begin{split}
y(t) &= \re^{\left(d\partial_\xx^2 + a\partial_\xx\right) (t-1)} y(1) + \int_1^t \re^{\left(d\partial_\xx^2 + a\partial_\xx\right) (t-s)} \Non_c(r(s),y(s),z(s),v(s),\gamma(s),\widetilde{\gamma}(s)) \de s,
\end{split}
\label{e:inty}
\end{align}
for $t \in [0,\tau_{\max})$ with $t \geq 1$, with nonlinearity
\begin{align*}
\Non_c(r,y,z,v,\gamma,\widetilde{\gamma}) = 2\nu r_\xx y_\xx + \frac{\nu}{d}\left(\nu r_\xx^2 + G(z,v,\gamma,\widetilde{\gamma})\right)y.
\end{align*}
Using Lemma~\ref{lem:nlboundsmod3}, we readily obtain the following nonlinear estimate.
\begin{lemma}\label{lemma_nonlinear_bound_on_G}
Fix a constant $C > 0$. Then, we have
\begin{align*}
\begin{split}
\left\|\Non_c(r,y,z,v,\gamma,\widetilde{\gamma})\right\|_\infty &\lesssim \|r_\xx\|_\infty \|y_\xx\|_\infty + \|y\|_\infty\Big(\|r_\xx\|_\infty^2 + \left(\|v\|_\infty + \|\gamma_\xx\|_\infty\right)\|v\|_\infty^2 + \|\gamma_\xx\|_{C_{\mathrm{ub}}^1}\|\gamma_{\xx\xx}\|_{C_{\mathrm{ub}}^1} \\ 
&\qquad +\, \|v\|_\infty \left(\|\widetilde{\gamma}\|_\infty + \|\gamma_{\xx\xx}\|_\infty + \|\gamma_{\xx}\|^2_\infty\right) + \left(\|z\|_\infty + \|\gamma_\xx\|_\infty\right)\|z\|_\infty+ \|\gamma_\xx\|^3_\infty \Big),
\end{split}
\end{align*}
for $z \in C_{\mathrm{ub}}(\R)$, $v \in C_{\mathrm{ub}}(\R)$, $r,y \in C_{\mathrm{ub}}^1(\R)$ and $(\gamma,\widetilde{\gamma}) \in C_{\mathrm{ub}}^3(\R) \times C_{\mathrm{ub}}(\R)$ with $\|v\|_\infty \leq C$ and $\|\gamma_\xx\|_\infty \leq \frac{1}{2}$.
\end{lemma}

\subsection{Forward-modulated perturbation} 

Since the equation~\eqref{e:modpertbeq} for the inverse-modulated perturbation $v(t)$ is quasilinear, it presents an apparent loss of derivatives that must be addressed in the nonlinear argument. In order to control regularity for \emph{short} times, we iteratively estimate the forward-modulated perturbation $\mathring{v}(t)$ via its Duhamel representation, which is semilinear and does not suffer from a loss of derivatives. As a first step, we note that the following existence and regularity properties of $\mathring{v}(t)$ follow directly from Propositions~\ref{well_posed_full_sol} and~\ref{p:gamma}.

\begin{corollary} \label{c:local_forward_v}
Assume~\ref{assH1} and~\ref{assD3}. Let $r_0 \in (0,\frac12)$ be as in Proposition~\ref{prop:family}. Let $\mathring{v}_0 \in C_{\mathrm{ub}}(\R)$ and $\gamma_0 \in C_{\mathrm{ub}}^1(\R)$ with $\|\gamma_0'\|_\infty < r_0$. Define $u_0, v_0 \in C_{\mathrm{ub}}(\R)$ by~\eqref{e:defu0} and~\eqref{e:defv0}, respectively. For $u$ as in Proposition~\ref{well_posed_full_sol}, and $\gamma$ and $\tau_{\max}$ as in Proposition~\ref{p:gamma}, the forward-modulated perturbation $\vt(t)$, given by~\eqref{e:defringv}, is well-defined for $t \in [0,\tau_{\max})$ and satisfies
\begin{align*}
\vt \in C\big([0,\tau_{\max}),C_{\mathrm{ub}}(\R)\big) \cap C^j\big((0,\tau_{\max}),C_{\mathrm{ub}}^l(\R)\big)
\end{align*}
for any $j,l \in \NM_0$.
\end{corollary}

Differentiating $\vt(t)$ with respect to time and using the facts that $\phi_0$ and $u(t)$ solve~\eqref{RD} and that $\gamma(t) = \smash{\re^{-\partial_\xx^4 t} \gamma_0}$ for $t \in [0,1]$ by Proposition~\ref{p:gamma}, we derive the semilinear equation
\begin{align} \label{e:Pert}
\partial_t \vt &= k_0^2 D \vt_{\xx\xx} + \omega_0 \vt_\xx + A[\gamma_0] \vt + \mathring{\mathcal{N}}(\vt,\gamma,\partial_t \gamma)
\end{align}
valid for $t \in [0,1]$, with spatiotemporal coefficient
\begin{align*}
A[\gamma_0](\xx,t) := f'(\phi_0(\xx+\Gamma(\xx,t))), \qquad \Gamma(\xx,t) := \left(\re^{-\partial_\xx^4 t} \gamma_0\right)[\zeta]
\end{align*}
and nonlinearity given by
\begin{align*}
\mathring{\mathcal{N}}(\vt,\gamma,\gamma_t) &= f(\vt + \phi_0(\kappa(\gamma))) - f(\phi_0(\kappa(\gamma))) - f'(\phi_0(\kappa)) \vt
+ \phi_0'(\kappa(\gamma)) \left(\omega_0 \gamma_\xx - \gamma_t\right)\\ 
&\qquad + k_0^2 D \left(\phi_0'(\kappa(\gamma)) \gamma_{\xx\xx} + \phi_0''(\kappa(\gamma)) \gamma_\xx(2+\gamma_\xx)\right),             
\end{align*}
with shorthand notation
\begin{align*}
\kappa(\gamma)(\xx,t) = \xx + \gamma(\xx,t).
\end{align*}
Differentiating~\eqref{e:Pert} with respect to space, we obtain an evolution equation for $\wt(t) := \mathring{v}_\xx(t)$, reading
\begin{align} \label{e:Pert11}
\partial_t \wt &= k_0^2 D \wt_{\xx\xx} + \omega_0 \wt_\xx + A[\gamma_0] \wt + \mathring{\mathcal{N}_1}(\vt,\gamma,\partial_t \gamma)
\end{align}
valid for $t \in [0,1]$, with nonlinearity
\begin{align*}
\mathring{\mathcal{N}}_1(\vt,\gamma,\gamma_t) &= (1 + \gamma_\xx) f''(\phi_0(\kappa(\gamma))) \left(\phi_0'(\kappa(\gamma)),\vt\right) + \partial_\xx \mathring{\mathcal{N}}(\vt,\gamma,\gamma_t).  \end{align*}

We now state short-time bounds on the temporal Green's function associated with the linearized equation
\begin{align} \label{e:linmathringv}
\partial_t \vt &= k_0^2 D \vt_{\xx\xx} + \omega_0 \vt_\xx + A[\gamma_0] \vt.
\end{align}
These bounds have been established in~\cite{ZUH}, using Levi's parametrix method.

\begin{proposition} \label{prop:pointwiseGreen}
Assume~\ref{assH1}. There exist constants $C,M > 0$ and $t_* \in (0,1]$ such that for each $\gamma_0 \in C_{\mathrm{ub}}^1(\R)$ the temporal Green's function $G \colon \R \times (0,t_*] \times \R \times (0,t_*] \to \R^{n \times n}$ associated with~\eqref{e:linmathringv} is continuously differentiable in its second and fourth coordinate and twice continuously differentiable in its first and third coordinate. The solution to the linearized equation~\eqref{e:linmathringv} with initial condition $\vt(s) = \vt_0$ is given by
\begin{align*}
v(\xx,t) = \int_\R G(\xx,t;\xt,s) \vt_0(\xt) \de \xt
\end{align*}
for all $\xx \in \R$ and any $s,t \in [0,t_*]$ with $0 \leq s < t$. Moreover, the Green's function enjoys the pointwise bounds
\begin{align*}
\left|\partial_\xx^j G(\xx,t;\xt,s)\right| \leq C (t-s)^{-\frac{1+j}{2}} \re^{-\frac{(\xx-\xt)^2}{M(t-s)}}, \qquad j = 0,1,2
\end{align*}
for all $\xx,\xt \in \R$ and $s,t \in (0,t_*]$ with $0 \leq s < t$.
\end{proposition}
\begin{proof}
We first observe that the coefficient $A[\gamma_0]$ in~\eqref{e:linmathringv} is uniformly bounded on $\mathbb{R} \times [0,1]$, with a bound that is independent of $\gamma_0$. Since $\gamma_0$ belongs to the interpolation space $C_{\mathrm{ub}}^1(\mathbb{R})$ of class $\smash{J_{\frac{1}{4}}}$ between $C_{\mathrm{ub}}(\mathbb{R})$ and the domain $C_{\mathrm{ub}}^4(\mathbb{R})$ of the sectorial operator $\smash{-\partial_\xx^4}$, it follows from~\cite[Corollary~2.2.3 and Proposition~2.2.4]{LUN} that the map $[0,1] \to C_{\mathrm{ub}}(\mathbb{R}), \, t \mapsto A[\gamma_0](\cdot,t)$ is H\"older continuous with exponent $\alpha \in (0,\frac{1}{4})$. Moreover, for each fixed $t \in [0,1]$, we have $A[\gamma_0](\cdot,t) \in C_{\mathrm{ub}}^1(\mathbb{R})$, so the map is also H\"older continuous in space with exponent $2\alpha$. Thus, the claim follows directly from~\cite[Proposition~11.3]{ZUH}.
\end{proof}

Integrating~\eqref{e:Pert} and~\eqref{e:Pert11} yields the Duhamel formulas
\begin{align}
\vt(\xx,t) = \int_\R G(\xx,t;\xt,0) \mathring{v}_0(\xt) \de \xt + \int_0^t \int_\R G(\xx,t;\xt,s) \mathring{\mathcal{N}}(\vt(\xx,s),\gamma(\xx,s),\partial_t \gamma(\xx,s))\de s \label{e:intmathringv}
\end{align}
for $\xx \in \R$ and $t \in (0,t_*]$, and
\begin{align}
\vt_\xx(\xx,t) = \int_\R G\left(\xx,t;\xt,\tfrac{t_*}{2}\right) \mathring{v}_\xx\left(\xt,\tfrac{t_*}{2}\right) \de \xt + \int_{\frac{t_*}{2}}^t \int_\R G(\xx,t;\xt,s) \mathring{\mathcal{N}}_1(\vt(\xx,s),\gamma(\xx,s),\partial_t \gamma(\xx,s))\de s \label{e:intmathringw}
\end{align}
for $\xx \in \R$ and $t \in [\tfrac{t_*}{2},t_*]$, where we chose $t = \frac{t_*}{2}$, rather than $t = 0$, as our left integration boundary in~\eqref{e:intmathringw}, since $\mathring{v}_\xx(t)$ is only guaranteed to exist for $t > 0$; see Corollary~\ref{c:local_forward_v}. The representations~\eqref{e:intmathringv} and~\eqref{e:intmathringw}, combined with the pointwise Green's function estimates from Proposition~\ref{prop:pointwiseGreen} and the following nonlinear bounds, whose derivation follows directly from Taylor's theorem, yield the short-time regularity control needed for the forthcoming nonlinear analysis; see also Remark~\ref{rem:Gamma}.

\begin{lemma} \label{lemma_nonlinear_bound_on_mathringN}
Assume~\ref{assH1}. Fix a constant $C> 0$. Then, we have
\begin{align*}
\left\|\mathring{\mathcal N}(\vt,\gamma,\gamma_t)\right\|_\infty &\lesssim \|\vt\|_\infty^2 + \|\gamma_\xx\|_{C_{\mathrm{ub}}^1} + \|\gamma_t\|_\infty,
\end{align*}
for $\vt \in C_{\mathrm{ub}}(\R)$ and $(\gamma,\gamma_t) \in C_{\mathrm{ub}}^2(\R) \times C_{\mathrm{ub}}(\R)$ satisfying $\|\vt\|_\infty, \|\gamma_\xx\|_{\infty} \leq C$. Moreover, we have
\begin{align*}
\left\|\mathring{\mathcal N}_1(\vt,\gamma,\gamma_t)\right\|_\infty \lesssim \|\vt\|_{C_{\mathrm{ub}}^1} + \|\gamma_\xx\|_{C_{\mathrm{ub}}^2} + \|\gamma_t\|_{C_{\mathrm{ub}}^1}
\end{align*}
for $\vt \in C_{\mathrm{ub}}^1(\R)$ and $(\gamma,\gamma_t) \in C_{\mathrm{ub}}^3(\R) \times C_{\mathrm{ub}}^1(\R)$ satisfying $\|\vt\|_\infty, \|\gamma_\xx\|_{\infty} \leq C$.
\end{lemma}

\subsection{Forward-modulated damping} \label{sec:forward_modulated}

In order to control regularity in the forthcoming nonlinear argument for \emph{large} times, we follow the strategy of~\cite{AdR1} and establish a nonlinear damping estimate for the modified forward-modulated perturbation $\zt(t)$ given by~\eqref{e:defringz}. This damping estimate extends the one in~\cite{JONZ} to a pure $L^\infty$-setting. It exploits the dissipative structure of the underlying reaction-diffusion system~\eqref{RD} and relies on the embedding of the uniformly local Sobolev space $H^1_{\mathrm{ul}}(\R)$ into $C_{\mathrm{ub}}(\R)$; see~\cite[Lemma~8.3.11]{SU17book}.

Before proving the nonlinear damping estimate, we first derive an evolution equation for $\zt(t)$. To this end, we recall that $u_k(x,t) = \phi(kx - \omega(k) t; k)$ is a solution to~\eqref{RD0} for all wavenumbers $k \in (k_0-r_0,k_0+r_0)$ by Proposition~\ref{prop:family}. Thus, using that $u(t)$ solves~\eqref{RD}, we obtain
\begin{align} \label{e:Pert1}
\partial_t \zt = k_0^2 D \zt_{\xx\xx} + \omega_0 \zt_\xx + \mathring{\mathcal{Q}}(\zt,\gamma) + \mathring{\mathcal{R}}(\gamma,\tilde{\gamma},\partial_t \gamma),
\end{align}
with
\begin{align*}
\mathring{\mathcal{Q}}(\zt,\gamma) &= f\left(\zt + \phi(\beta(\gamma))\right) - f(\phi(\beta(\gamma)))
\end{align*}
and
\begin{align*}
\mathring{\mathcal{R}}(\gamma,\tilde{\gamma},\gamma_t) 
&= k_0^2 D\Big[\phi_{yy}(\beta(\gamma))\left(\left(1+\gamma_\xx(1+\gamma_\xx) + \gamma\gamma_{\xx\xx}\right)^2 - (1+\gamma_\xx)^2\right) + 
k_0^2 \phi_{kk}(\beta(\gamma))\gamma_{\xx\xx}^2 \\ 
&\qquad + 2 k_0 \phi_{yk}(\beta(\gamma)) \gamma_{\xx\xx}\left(1+\gamma_\xx(1+\gamma_\xx) + \gamma\gamma_{\xx\xx}\right) + \phi_{y}(\beta(\gamma))\left(\gamma_{\xx\xx}(1+3\gamma_\xx) + \gamma\gamma_{\xx\xx\xx}\right) \\
&\qquad + k_0 \phi_{k}(\beta(\gamma))\gamma_{\xx\xx\xx} 
\Big] + k_0 \phi_{k}(\beta(\gamma))\left(\omega_0 \gamma_{\xx\xx} - \gamma_{\xx t}\right)\\
&\qquad + \phi_{y}(\beta(\gamma))\left(\omega_0 + k_0 \omega'(k_0) \gamma_\xx - \omega\big(k_0(1+\gamma_\xx)\big) - \tilde{\gamma} + \omega_0 \left(\gamma_\xx^2 + \gamma \gamma_{\xx\xx}\right)- \gamma_t \gamma_\xx - \gamma \gamma_{\xx t}\right),
\end{align*}
where we used
\begin{align*}
\beta(\gamma)(\xx,t) = \big(\xx + \gamma(\xx,t)(1+\gamma_\xx(\xx,t));\, k_0\left(1+\gamma_\xx(\xx,t)\right)\big)
\end{align*}
to abbreviate the argument of $\phi(y;k)$ and its derivatives. For more details on the derivation of~\eqref{e:Pert1}, we refer to~\cite[Appendix~B]{AdR1}. 

Since the continuation $\phi(\cdot; k)$ of the wave train $\phi_0$ with respect to the wavenumber $k$ is defined for all $k \in (k_0 - r_0, k_0 + r_0)$ by Proposition~\ref{prop:family}, it follows from Propositions~\ref{well_posed_full_sol} and~\ref{p:gamma} that $\mathring{z}(t)$ is well-defined for all $t \in [0, \tau_{\max})$. Its regularity properties are also a direct consequence of these propositions and are summarized in the following statement.

\begin{corollary} \label{c:local_forward_z}
Assume~\ref{assH1} and~\ref{assD3}. Let $r_0 \in (0,\frac12)$ be as in Proposition~\ref{prop:family}. Let $\mathring{v}_0 \in C_{\mathrm{ub}}(\R)$ and $\gamma_0 \in C_{\mathrm{ub}}^1(\R)$ with $\|\gamma_0'\|_\infty < r_0$. Define $u_0, v_0 \in C_{\mathrm{ub}}(\R)$ by~\eqref{e:defu0} and~\eqref{e:defv0}, respectively. For $u$ as in Proposition~\ref{well_posed_full_sol}, and $\gamma$ and $\tau_{\max}$ as in Proposition~\ref{p:gamma}, the modified forward-modulated perturbation $\zt(t)$, given by~\eqref{e:defringz}, is well-defined for $t \in [0,\tau_{\max})$ and satisfies
\begin{align*}
\zt \in C\big([0,\tau_{\max}),C_{\mathrm{ub}}(\R)\big) \cap C^j\big((0,\tau_{\max}),C_{\mathrm{ub}}^l(\R)\big)
\end{align*}
for any $j,l \in \NM_0$.
\end{corollary}

We are now in position to establish the relevant nonlinear damping estimate. The proof follows the strategy of~\cite[Proposition~4.9]{AdR1}, but differs in that damping is induced by the second derivative across all components, not just in the first component.

\begin{proposition} \label{prop:nonlinear_damping}
Assume~\ref{assH1} and~\ref{assD3}. Fix a constant $R > 0$. Let $r_0 \in (0,\frac12)$ be as in Proposition~\ref{prop:family}. Let $\mathring{v}_0 \in C_{\mathrm{ub}}(\R)$ and $\gamma_0 \in C_{\mathrm{ub}}^1(\R)$ with $\|\gamma_0'\|_\infty < r_0$. Define $u_0, v_0 \in C_{\mathrm{ub}}(\R)$ by~\eqref{e:defu0} and~\eqref{e:defv0}, respectively. Let $\gamma(t)$ and $\tau_{\max}$ be as in Proposition~\ref{p:gamma}, let $t_*$ be as in Proposition~\ref{prop:pointwiseGreen}, and let $\zt(t)$ be as in Corollary~\ref{c:local_forward_z}. There exists a $\mathring{v}_0$- and $\gamma_0$-independent constant $C > 0$ such that the nonlinear damping estimate
\begin{align} \label{e:dampingineq}
\begin{split}
\big\|\zt(t)\big\|_{C_{\mathrm{ub}}^1} &\leq C\Bigg(\big\|\zt(t)\big\|_{\infty} + \bigg(\re^{t_*-t}  \big\|\zt(t_*)\big\|_{C_{\mathrm{ub}}^2}^2 + \int_{t_*}^t \re^{s-t} \Big(\big\|\mathring{z}(s)\big\|_{\infty}^2 + \|\gamma_{\xx\xx}(s)\|_{C_{\mathrm{ub}}^3}^2 \\ 
&\qquad  + \, \|\partial_s \gamma_{\xx}(s)\|_{C_{\mathrm{ub}}^2}^2 + \big\|\tilde{\gamma}(s)\big\|_{C_{\mathrm{ub}}^2}^2 + \|\gamma_{\xx}(s)\|_{\infty}^2\left(\|\gamma_{\xx}(s)\|_{\infty}^2 + \|\partial_s \gamma(s)\|_{\infty}^2\right)\Big) \,\de s\bigg)^{\frac12}\Bigg)
\end{split}
\end{align}
holds for all $t \in [0,\tau_{\max})$ with $t \geq t_*$ and
\begin{align}\label{e:upbound}
\sup_{t_* \leq s \leq t} \left(\big\|\mathring{z}(s)\big\|_{C_{\mathrm{ub}}^1} + \|\gamma(s)\|_{C_{\mathrm{ub}}^2}\right) \leq R.
\end{align}
\end{proposition}
\begin{proof}
Set $\vartheta = \frac{1}{2}$. We begin by relating the $C_{\mathrm{ub}}^1$-norm of $\zt(t)$ to a uniformly local Sobolev norm. To this end, introduce the window function $\varrho \colon \R \to \R$ given by
\begin{align*}
\varrho(\xx) = \frac{2}{2+\xx^2},
\end{align*}
which is positive, smooth, and $L^1$-integrable. It satisfies the inequality
\begin{align} \label{e:rhoineq}
|\varrho'(\xx)| \leq \varrho(\xx) \leq 1
\end{align}
for all $\xx \in \R$. Applying the Gagliaro-Nirenberg interpolation inequality, we estimate
\begin{align*}
\|z_\xx\|_\infty = \sup_{y \in \R} \|\varrho(\vartheta(\cdot + y)) z_\xx\|_\infty &\lesssim  \|z\|_\infty + \sup_{y \in \R} \|\partial_\xx(\varrho(\vartheta(\cdot + y)) z)\|_\infty\\
&\lesssim \|z\|_{\infty} + \sup_{y 
\in \R} \left\|\varrho(\vartheta(\cdot + y)) z\right\|_\infty^{\frac13} \left\|\partial_\xx^2(\varrho(\vartheta(\cdot + y)) z)\right\|_2^{\frac23}\\
&\lesssim \|z\|_{\infty} + \|z\|_{\infty}^{\frac13} \left(\|z\|_{C_\mathrm{ub}^1}^{\frac23} + \sup_{y 
\in \R} \left\|\varrho(\vartheta(\cdot + y)) z_{\xx\xx}\right\|_2^{\frac23}\right)
\end{align*}
for $z \in C_{\mathrm{ub}}^1(\R)$. Hence, employing Young's inequality and rearranging terms, we obtain the bound
\begin{align} \label{e:upbounddamping}
\begin{split}
\|\zt(t)\|_{C_{\mathrm{ub}}^1} &\lesssim \|\zt(t)\|_{\infty} + \sup_{y \in \R} E_y(t)^{\frac12},
\end{split}
\end{align}
valid for all $t \in (0,\tau_{\max})$, where we define
\begin{align*}
E_y(t) = \int_\R \varrho(\vartheta(\xx+y)) |\mathring{z}_{\xx\xx}(\xx,t)|^2 \,\de \xx
\end{align*}
for $y \in \R$ and $t \in (0,\tau_{\max})$.

We now derive a differential inequality for the energy $E_y(t)$. By Proposition~\ref{p:gamma} and Corollary~\ref{c:local_forward_z}, the function $t \mapsto E_y(t)$ is differentiable on $(0,\tau_{\max})$. Since the diffusion matrix $D$ is positive definite, there exists a constant $d_0 > 0$ such that the coercivity estimate
\begin{align} \label{e:Dposdef}
\langle z, D z \rangle \geq d_0 |z|^2
\end{align}
holds for all $z \in \R^n$. Fix $y \in \R$ and $t \in [0,\tau_{\max})$ with $t > t_*$ such that~\eqref{e:upbound} holds. Using~\eqref{e:Pert1}, we compute
\begin{align} \label{e:damptemp}
\frac{1}{2} \partial_s E_y(s) = I + II
\end{align}
for $s \in [t_*,t]$, where we denote
\begin{align*}
I &= \int_\R \varrho(\vartheta(\xx+y)) \left\langle \partial_{\xx}^2 \mathring{z}(\xx,s), k_0^2 D \partial_{\xx}^4 \mathring{z}(\xx,s) + \omega_0 \partial_{\xx}^3 \mathring{z}(\xx,s) \right \rangle \,\de \xx,\\
II &= \int_\R \varrho(\vartheta(\xx+y)) \left\langle \partial_{\xx}^2 \mathring{z}(\xx,s), \partial_\xx^2 \left(\mathring{\mathcal{Q}}(\zt(\xx,s),\gamma(\xx,s)) + \mathring{\mathcal{R}}(\gamma(\xx,s),\tilde{\gamma}(\xx,s),\partial_s \gamma(\xx,s)) \right) \right\rangle \,\de \xx.
\end{align*}
Integrating by parts, we find
\begin{align*}
I &= -k_0^2 \int_\R \varrho(\vartheta(\xx+y)) \left\langle \partial_{\xx}^3 \mathring{z}(\xx,s), D \partial_{\xx}^3 \mathring{z}(\xx,s) \right\rangle \,\de \xx - k_0^2 \vartheta \int_\R \varrho'(\vartheta(\xx+y)) \left\langle \partial_{\xx}^2 \mathring{z}(\xx,s), D \partial_{\xx}^3 \mathring{z}(\xx,s) \right\rangle \,\de \xx \\
&\qquad + \, \omega_0 \int_\R \varrho(\vartheta(\xx+y)) \left\langle \partial_{\xx}^2 \mathring{z}(\xx,s), \partial_{\xx}^3 \mathring{z}(\xx,s) \right\rangle \,\de \xx.
\end{align*}
Applying Young's inequality and the estimates~\eqref{e:rhoineq} and~\eqref{e:Dposdef}, we obtain a constant $C_1 > 0$, independent of $t$, $\mathring{v}_0$, and $\gamma_0$, such that
\begin{align} \label{e:dampest1} \begin{split} 
I &\leq -\frac{d_0 k_0^2}{2} \int_\R \varrho(\vartheta(\xx+y)) \left|\partial_{\xx}^3 \mathring{z}(\xx,s)\right|^2 \,\de \xx + C_1 \int_\R \varrho(\vartheta(\xx+y)) \left|\partial_{\xx}^2 \mathring{z}(\xx,s)\right|^2 \,\de \xx
\end{split}\end{align} 
for $s \in [t_*,t]$. On the other hand, estimating $II$ using Young's inequality and assumption~\eqref{e:upbound}, we find a $t$-, $\mathring{v}_0$-, and $\gamma_0$-independent constant $C_2 > 0$ such that
\begin{align} \label{e:dampest3} \begin{split} 
II &\leq C_2 \left(\int_\R \varrho(\vartheta(\xx+y))\left( \left|\partial_{\xx}^2 \mathring{z}(\xx,s)\right|^2 + \left|\partial_{\xx} \mathring{z}(\xx,s)\right|^2\right) \de\xx + \|\mathring{z}(s)\|_{\infty}^2  + \|\gamma_{\xx\xx}(s)\|_{C_{\mathrm{ub}}^3}^2 \right.\\
&\left. \phantom{\int_\R} \qquad \qquad + \|\partial_s \gamma_{\xx}(s)\|_{C_{\mathrm{ub}}^2}^2 + \big\|\tilde{\gamma}(s)\big\|_{C_{\mathrm{ub}}^2}^2 + \|\gamma_{\xx}(s)\|_{\infty}^2\left(\|\gamma_{\xx}(s)\|_{\infty}^2 + \|\partial_s \gamma(s)\|_{\infty}^2\right)\right)
\end{split}\end{align} 
for $s \in [t_*,t]$. Inserting~\eqref{e:dampest1} and~\eqref{e:dampest3} into~\eqref{e:damptemp}, we obtain a $t$-, $\mathring{v}_0$-, and $\gamma_0$-independent constant $C_3 > 0$ such that
\begin{align} \label{e:dampest5}
\begin{split}
\frac{1}{2} \partial_s E_y(s) &\leq -\frac12 E_y(s) - \frac{d_0 k_0^2}{2} \int_\R \varrho(\vartheta(\xx+y)) \left|\partial_{\xx}^3 \mathring{z}(\xx,s)\right|^2 \,\de \xx + C_3\Bigg(\|\mathring{z}(s)\|_{\infty}^2 + \|\gamma_{\xx\xx}(s)\|_{C_{\mathrm{ub}}^3}^2 \\
&\qquad + \, \|\partial_s \gamma_{\xx}(s)\|_{C_{\mathrm{ub}}^2}^2 + \big\|\tilde{\gamma}(s)\big\|_{C_{\mathrm{ub}}^2}^2 +  \|\gamma_{\xx}(s)\|_\infty^2\left(\|\gamma_{\xx}(s)\|_\infty^2 + \|\partial_s \gamma(s)\|_\infty^2\right)\\
&\qquad 
\, + \int_\R \varrho(\vartheta(\xx+y)) \left(\left|\partial_{\xx}^2 \mathring{z}(\xx,s)\right|^2 + \left|\partial_{\xx} \mathring{z}(\xx,s)\right|^2\right) \,\de \xx \Bigg) 
\end{split}
\end{align}
for $s \in [t_*,t]$. 

To estimate the last line in~\eqref{e:dampest5}, we use the interpolation inequality from~\cite[Estimate~(4.30)]{AdR1}, choosing parameters
\begin{align*}
\eta \in (0,\tfrac{1}{4}), \qquad k = 2, \qquad a_0 = 0 = a_3, \qquad a_1 = \frac{8 (1 + 2 \eta)}{\eta (5-2 \eta)}, \qquad a_2 = \frac{4 (3 + 2 \eta)}{5 - 2 \eta},
\end{align*}
resulting in
\begin{align*}
\sum_{j = 1}^2 \int_\R \varrho(\vartheta(\xx+y)) \left|\partial_\xx^j z(\xx)\right|^2 \,\de \xx &\leq \frac{2 \eta (3 + 2 \eta)}{2 \eta-5} \int_\R \varrho(\vartheta(\xx+y)) \left|\partial_\xx^3 z(\xx)\right|^2 \,\de \xx\\ 
&\qquad +\, \frac{2 (2 + \eta) (1 + 2 \eta)}{\eta^2 (5-2 \eta)} \int_\R \varrho(\vartheta(\xx+y)) \left|z(\xx)\right|^2 \de \xx.
\end{align*}
Taking $\eta \in (0,\frac14)$ small enough that
\begin{align*}
\frac{2 \eta (3 + 2 \eta)}{2 \eta - 5} \leq \frac{d_0 k_0^2}{2C_3},
\end{align*}
we obtain a constant $C_4 > 0$ such that
\begin{align} \label{e:interpolationUL}
\sum_{j = 1}^2 \int_\R \varrho(\vartheta(\xx+y)) \left|\partial_\xx^j z(\xx)\right|^2 \,\de \xx &\leq \frac{d_0 k_0^2}{2C_3} \int_\R \varrho(\vartheta(\xx+y)) \left|\partial_\xx^3 z(\xx)\right|^2 \,\de \xx + C_4 \|z\|_{\infty}^2
\end{align}
for $z \in C_{\mathrm{ub}}^3(\R)$. 

Applying the interpolation inequality~\eqref{e:interpolationUL} to~\eqref{e:dampest5}, we conclude that
\begin{align*}
\partial_s E_y(s) &\leq -E_y(s) + C_5 \left(\big\|\mathring{z}(s)\big\|_{\infty}^2 + \|\gamma_{\xx\xx}(s)\|_{C_{\mathrm{ub}}^3}^2 + \|\partial_s \gamma_{\xx}(s)\|_{C_{\mathrm{ub}}^2}^2 + \big\|\tilde{\gamma}(s)\big\|_{C_{\mathrm{ub}}^2}^2\right.\\ 
&\left.\qquad \qquad \qquad \phantom{\big\|\tilde{\gamma}(s)\big\|_{C_{\mathrm{ub}}^2}^2} + \|\gamma_{\xx}(s)\|_{\infty}^2\left(\|\gamma_{\xx}(s)\|_{\infty}^2 + \|\partial_s \gamma(s)\|_{\infty}^2\right)\right)
\end{align*}
for $s \in [t_*,t]$, where $C_5 > 0$ is independent of $t$, $\mathring{v}_0$, and $\gamma_0$. Multiplying both sides by $\re^s$ and integrating over $s \in [t_*,t]$, we arrive at
\begin{align*}
E_y(t) &\leq \re^{t_*-t} E_y(t_*) + C_5 \int_{t_*}^t \re^{s-t} \left(\big\|\mathring{z}(s)\big\|_{\infty}^2 + \|\gamma_{\xx\xx}(s)\|_{C_{\mathrm{ub}}^3}^2 + \|\partial_s \gamma_{\xx}(s)\|_{C_{\mathrm{ub}}^2}^2 + \big\|\tilde{\gamma}(s)\big\|_{C_{\mathrm{ub}}^2}^2\right.\\ 
&\left.\qquad \qquad \qquad \phantom{\big\|\tilde{\gamma}(s)\big\|_{C_{\mathrm{ub}}^2}^2} + \|\gamma_{\xx}(s)\|_{\infty}^2\left(\|\gamma_{\xx}(s)\|_{\infty}^2 + \|\partial_s \gamma(s)\|_{\infty}^2\right)\right) \,\de s.\end{align*}
The damping estimate~\eqref{e:dampingineq} now follows by plugging the latter bound into~\eqref{e:upbounddamping} and using the fact that there exists a constant $C_6 > 0$ (independent of $\gamma_0$ and $\mathring{v}_0$) such that $E_y(t_*) \leq \smash{C_6 \|\zt(t_*)\|_{C_{\mathrm{ub}}^2}^2}$.
\end{proof}

We conclude this section by recalling the results from~\cite[Lemma~4.11]{AdR1} and~\cite[Lemma~5.1]{ZUM23}, which state that the $C_{\mathrm{ub}}^k$-norms of the forward- and inverse-modulated perturbations $\vt(t)$ and $v(t)$, as well as those of the modified forward-modulated perturbation $\zt(t)$ and the residual $z(t)$, are equivalent, up to controllable errors depending on $\gamma_\xx(t)$ and its derivatives.

\begin{lemma}  \label{lem:equivalence}
Fix a constant $R > 0$. Let $\mathring{v}_0 \in C_{\mathrm{ub}}(\R)$ and $\gamma_0 \in C_{\mathrm{ub}}^1(\R)$ with $\|\gamma_0'\|_\infty < r_0$. Define $u_0, v_0 \in C_{\mathrm{ub}}(\R)$ by~\eqref{e:defu0} and~\eqref{e:defv0}, respectively. Let $\gamma(t)$ and $\tau_{\max}$ be as in Proposition~\ref{p:gamma}, let $v(t)$ and $z(t)$ be as in Corollary~\ref{C:local_v}, let $\vt(t)$ be as in Corollary~\ref{c:local_forward_v}, and let $\zt(t)$ be as in Corollary~\ref{c:local_forward_z}. Then, we have
\begin{align*}
\begin{split}
\|v(t)\|_{\infty} &\lesssim \big\|\vt(t)\big\|_{\infty} + \|\gamma_{\xx}(t)\|_{\infty}, \qquad \big\|\vt(t)\big\|_{\infty} \lesssim \|v(t)\|_{\infty} + \|\gamma_{\xx}(t)\|_{\infty},
\end{split}
\end{align*}
for any $t \in [0,\tau_{\max})$ with
\begin{align} \label{e:gammaapriori}
\sup_{0 \leq s \leq t} \|\gamma(s)\|_{\infty} \leq R.
\end{align}
Moreover, we have
\begin{align*}
\begin{split}
\big\|\zt(t)\big\|_\infty &\lesssim \|z(t)\|_\infty + \|\gamma_{\xx\xx}(t)\|_\infty + \|\gamma_{\xx}(t)\|^2_\infty \\
\|z(t)\|_{C_{\mathrm{ub}}^1} &\lesssim \big\|\zt(t)\big\|_{C_{\mathrm{ub}}^1} + \|\gamma_{\xx\xx}(t)\|_{C_{\mathrm{ub}}^1} + \|\gamma_{\xx}(t)\|^2_\infty 
\end{split}
\end{align*}
for any $t \in (0,\tau_{\max})$ satisfying~\eqref{e:gammaapriori}.
\end{lemma}
\begin{proof}
The first two inequalities were proved in~\cite[Lemma~5.1]{ZUM23}. The last two inequalities follow directly from the estimates established in the proof of~\cite[Lemma~4.11]{AdR1}.
\end{proof}

\section{Nonlinear stability argument} \label{sec:nonlinearstab}

In this section, we prove our main result, Theorem~\ref{main_theorem}, by completing a nonlinear stability argument based on a quasilinear iteration scheme built around the integral equations~\eqref{e:intgamma},~\eqref{e:intz},~\eqref{e:intr}, and~\eqref{e:inty}. Short-time regularity control is obtained via iterative estimates applied to the Duhamel representations~\eqref{e:intmathringv} and~\eqref{e:intmathringw} for the forward-modulated perturbation, while long-time regularity is ensured by the nonlinear damping estimate provided in Proposition~\ref{prop:nonlinear_damping}.

\begin{proof}[Proof of Theorem~\ref{main_theorem}] Let $r_0 \in (0,\frac12)$ be as in Proposition~\ref{prop:family}. Take $\vt_0 \in C_{\mathrm{ub}}(\R)$ and $\gamma_0 \in C_{\mathrm{ub}}^1(\R)$ with 
\begin{align*}
\|\gamma_0\|_\infty \leq M, \qquad E_0 := \|\vt_0\|_\infty + \left\|\gamma_0'\right\|_\infty < r_0.
\end{align*}
Define $u_0,v_0 \in C_{\mathrm{ub}}(\R)$ by~\eqref{e:defu0} and~\eqref{e:defv0}, respectively. 

By Proposition~\ref{well_posed_full_sol}, there exist a maximal time $T_{\max} \in (0,\infty]$ and a unique classical solution $u(t)$ to~\eqref{RD} with initial condition $u(0) = u_0$ satisfying~\eqref{e:classicalgamma}. If $T_{\max} < \infty$, then~\eqref{e:blowupu} holds. Moreover, Proposition~\ref{p:gamma} yields a maximal time $\tau_{\max} \in [1,\max\{1,T_{\max}\}]$ and a solution $\gamma(t)$ to~\eqref{e:intgamma} satisfying~\eqref{e:classicalgamma}, $\gamma(t) = \smash{\re^{-\partial_\xx^4 t} \gamma_0}$ for $t \in [0,1]$, and $\|\gamma_\xx(t)\|_\infty < r_0$ for all $t \in [0,\tau_{\max})$. Finally, if $\tau_{\max} < T_{\max}$, then we have~\eqref{e:blowupgamma1} or~\eqref{e:blowupgamma2}.

Our goal is to prove that $\tau_{\max} = T_{\max} = \infty$ and that $u(t)$ and $\gamma(t)$ satisfy the decay estimates~\eqref{e:mtest10},~\eqref{e:mtest2},~\eqref{e:mtest3}, and~\eqref{e:mtest33}, where $\breve{\gamma}(t)$ is the classical solution to the viscous Hamilton-Jacobi equation~\eqref{e:HamJac} with initial condition $\breve{\gamma}(0) = \gamma_0$. To this end, we define a template function, controlling the norms of the phase modulation $\gamma(t)$, the residuals $z(t)$ and $r(t)$, and the Cole-Hopf variable $y(t)$, which are defined by~\eqref{e:defz},~\eqref{e:intr}, and~\eqref{e:defy}, respectively. We first establish the result for $\alpha \in (0,\frac16)$. The case $\alpha = 0$ then follows a posteriori.

\paragraph*{Template function.} Let $\varrho \colon [0,\infty) \to [0,1]$ be a smooth temporal cut-off function which vanishes on $[0,\tfrac{t_*}{2}]$ and satisfies $\varrho(t) = 1$ for all $t \in [t_*,\infty)$, where $t_*$ is as in Proposition~\ref{prop:pointwiseGreen}. In addition, set $\smash{\widetilde{\tau}_{\max}} = \min\{\tau_{\max},T_{\max}\}$. By Proposition~\ref{p:gamma}, Corollaries~\ref{C:local_v} and~\ref{C:local_r}, and identities~\eqref{e:Gamma_rates3} and~\eqref{e:regy}, the template function $\eta \colon [0,\smash{\widetilde{\tau}_{\max}}) \to \R$ given by
\begin{align*}
\eta(t) &= \sup_{0\leq s\leq t} \Bigg[\frac{(1+s)^{1-2\alpha}}{\log(2+s)} \left(\|z(s)\|_{L^\infty} + \varrho(s) \left(\left\|\gamma_{\xx\xx\xx\xx}(s)\right\|_{C_{\mathrm{ub}}^1} + \big\|\widetilde{\gamma}_\xx(s)\big\|_{C_{\mathrm{ub}}^2}\right)\right) \\
&\qquad\qquad + \frac{\sqrt{s} \, (1+s)^{\frac12 - 2\alpha}}{\log(2+s)}  \left(\|r_\xx(s)\|_{\infty} + \|z_\xx(s)\|_\infty + \|\gamma_{\xx\xx\xx}(s)\|_{\infty}\right) + \frac{s^{\frac34} (1+s)^{\frac14 - 2\alpha}}{\log(2+s)} \big\|\widetilde{\gamma}(s)\big\|_\infty \\
&\qquad\qquad + \frac{s^{\frac14} (1+s)^{\frac34 - 2\alpha}}{\log(2+s)} \|\gamma_{\xx\xx}(s)\|_{\infty} + \sqrt{s}\,(1+s)^{-\alpha} \|y_\xx(s)\|_{\infty} + (1+s)^{\frac12 - \alpha} \|\gamma_\xx(s)\|_\infty \Bigg],
\end{align*}
is well-defined, continuous, non-negative, and monotonically increasing, where we recall $\widetilde{\gamma}(t) = \partial_t \gamma(t) - a\gamma_\xx(t)$. 

\paragraph*{Approach.} Fix $\alpha \in (0,\frac16)$. Our goal is to show that there exist constants $C,\eta_0 > 0$ such that the estimates
\begin{align}
\eta(0) \leq C E_0^\alpha, \qquad \eta(t) \leq C\left(E_0^\alpha + \eta(t)^2\right), \qquad \|\gamma(t)\|_\infty \leq C \label{e:etaest}
\end{align}
hold for all $t \in [0, \smash{\widetilde{\tau}_{\max}})$ with $\eta(t) \leq \eta_0$. Set
\begin{align*}
\varepsilon := \min\left\{\frac{1}{4C^2}, \frac{\eta_0}{2C} \right\}^{\frac{1}{\alpha}}.
\end{align*}
Then, provided $E_0 \in (0,\varepsilon)$, we find that~\eqref{e:etaest} implies that, for any $t \in [0, \smash{\widetilde{\tau}_{\max}})$ such that $\eta(s) \leq 2C E_0^\alpha$ for all $s \in [0, t]$, we have $\eta(t) \leq \eta_0$ and
\begin{align*} \eta(t) \leq C\left(E_0^\alpha + 4C^2 E_0^{2\alpha}\right) < 2C E_0^\alpha.
\end{align*}
Hence, by continuity of $\eta$ and the initial bound $\eta(0) \leq C E_0^\alpha$, it follows that if $E_0 \in (0, \varepsilon)$, then $\eta(t) \leq 2C E_0^\alpha < \eta_0$ for \emph{all} $t \in [0, \smash{\widetilde{\tau}_{\max}})$. The latter, in combination with the last estimate in~\eqref{e:etaest}, precludes the blow-up alternatives~\eqref{e:blowupu},~\eqref{e:blowupgamma1}, and~\eqref{e:blowupgamma2}, thereby implying $\tau_{\max} = T_{\max} = \infty$. We use the obtained global control of $z(t)$ and $\gamma(t)$ to subsequently prove the estimates~\eqref{e:mtest10} and~\eqref{e:mtest2}, where we make use of the norm equivalences established in Lemma~\ref{lem:equivalence}. Finally, we show the bounds~\eqref{e:mtest3} and~\eqref{e:mtest33} by proceeding along the lines of~\cite{BjoernMod}.

In the following, we will establish the key inequalities~\eqref{e:etaest} by bounding the terms in the template function $\eta(t)$ one by one; see Figure~\ref{fig_time_control}. Since we consider initial data of minimal regularity, cf.~Remark~\ref{rem:regularity}, higher-order derivatives may suffer from nonintegrable temporal bounds. To address this, we often distinguish between short-time bounds, which focus on reducing the number of derivatives, and long-time bounds, which ensure sufficient temporal decay.

\begin{figure}[t]
    \centering
    \vspace{0.3cm}
    \includegraphics[width=0.55\linewidth]{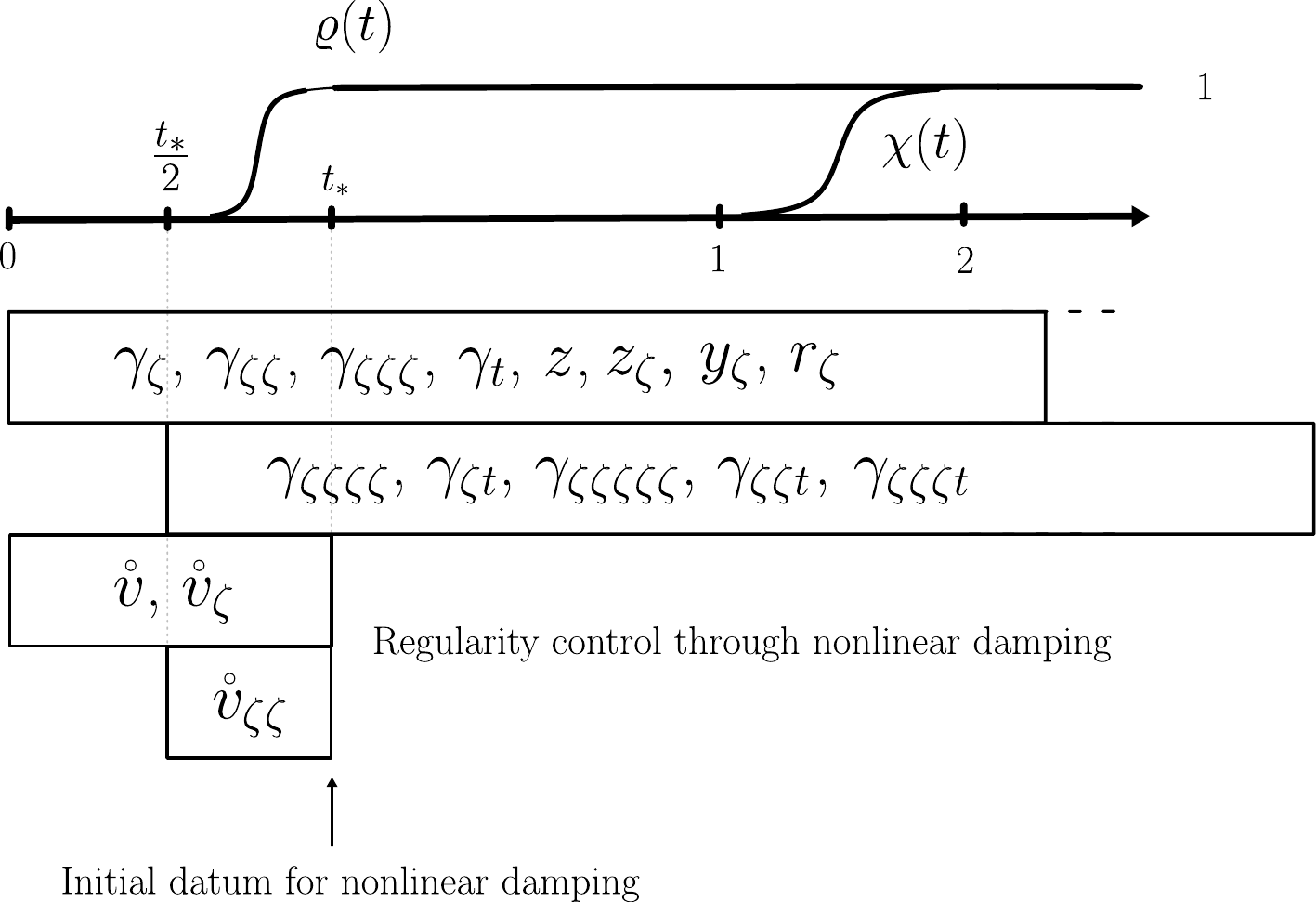}
    \caption{This figure depicts the temporal cut-off functions $\varrho(t)$ and $\chi(t)$, along with the time intervals over which we control the variables in the template function $\eta(t)$ (upper two rows). Short-time regularity on $[0, t_*]$ is obtained via iterative estimates on the Duhamel formulas for $\vt$ and $\vt_{\xx}$, while long-time regularity control is provided by the nonlinear damping estimate from Proposition~\ref{prop:nonlinear_damping} (third row). The final row indicates the time interval on which we control $\vt_{\xx\xx}$, required to supply the initial datum for the nonlinear damping estimate.}
    \label{fig_time_control}
\end{figure}

\paragraph*{Bound on \texorpdfstring{$v_0$}{v0}.}
Inserting~\eqref{e:defu0} into~\eqref{e:defv0}, we express
\begin{align*}
v_0(\xx) = \vt_0(\xx - \gamma_0(\xx)) + \phi_0\left(\xx - \gamma_0(\xx) + \gamma_0(\xx - \gamma_0(\xx)\right) - \phi_0(\xx).
\end{align*}
Using the mean-value theorem twice, we obtain the bound
\begin{align} \label{e:v0bound}
\|v_0\|_\infty \leq \|\vt_0\|_\infty + \|\phi_0'\|_\infty \|\gamma_0\|_\infty \|\gamma_0'\|_\infty \lesssim E_0.
\end{align}

\paragraph*{Interpolation bounds.} We establish bounds on the linear terms involving the initial phase modulation $\gamma_0$ in the Duhamel formulations of $\gamma(t)$, $z(t)$, $r(t)$, $\widetilde{y}(t)$, and their derivatives. Since $\gamma_0$ is bounded, the standard linear estimates from Propositions~\ref{prop:lin1} and~\ref{prop:lin2} provide sufficient temporal decay. However, since $\gamma_0$ is not necessarily small, they do not guarantee global-in-time smallness. To address this, we turn to the modulational estimates in Propositions~\ref{prop:lin_mod_1} and~\ref{prop:lin_mod_2}, which crucially exploit the smallness of the derivative $\gamma_0'$. By interpolating between the standard and modulational linear estimates, we obtain both the required temporal decay and global-in-time smallness. For a heuristic overview of this argument in a simplified setting, we refer to~\S\ref{sec:heuristic}.

We start by bounding the linear term
\begin{align*}
J_1(t) &:= \widetilde{S}(t)\left(\phi_0' \gamma_0\right) - (1-\chi(t)) \phi_0' \re^{-\partial_\xx^4 t} \gamma_0
\end{align*}
in the Duhamel representation~\eqref{e:intz} of $z(t)$. Proposition~\ref{prop:lin1} and estimate~\eqref{e:Gamma_rates3} yield
\begin{align} \label{e:interp4}
\|J_1(t)\|_\infty &\lesssim (1+t)^{-1}
\end{align}
for $t \geq 0$. Using~\eqref{e:decomp_full_semigroup}, we express
\begin{align*}
J_1(t) &= \re^{\El_0 t} \left(\phi_0' \gamma_0\right) - \phi_0' \gamma_0 - \, k_0 \partial_k \phi(\cdot;k_0) \partial_\xx S_p^0(t) \left(\phi_0'\gamma_0\right)\\
&\qquad - \phi_0' \left(\re^{-\partial_\xx^4 t} \gamma_0 - \gamma_0 + S_p^0(t)\left(\phi_0'\gamma_0\right) - \chi(t) \re^{\left(d\partial_\xx^2 + a\partial_\xx\right) t} \gamma_0 + \chi(t)\left(\re^{\left(d\partial_\xx^2 + a\partial_\xx\right) t} \gamma_0 - \re^{-\partial_\xx^4 t} \gamma_0\right) \right)
\end{align*}
Hence, Propositions~\ref{prop:lin_mod_1} and~\ref{prop:lin_mod_2} afford the bound
\begin{align} \label{e:interp5}
\|J_1(t)\|_\infty &\lesssim (1+t)  E_0
\end{align}
for $t \geq 0$. Interpolating between~\eqref{e:interp4} and~\eqref{e:interp5} and using $0\leq E_0 \leq 1$ and $0 < \alpha < \frac16$, we arrive at
\begin{align} \label{e:interp6}
\|J_1(t)\|_\infty &\lesssim (1+t)^{-1+2\alpha} E_0^{\alpha}
\end{align}
for $t \geq 0$.

We proceed with estimating the linear term
\begin{align*}
I_{j,l}(t) := \left(\partial_t - a \partial_\zeta\right)^j \partial_\zeta^l \left(S_p^0(t)(\phi_0'\gamma_0) + (1-\chi(t)) \re^{-\partial_\xx^4 t} \gamma_0\right)
\end{align*}
in the Duhamel formulation~\eqref{e:intgamma} of $\left(\partial_t - a \partial_\zeta\right)^j \partial_\zeta^l \gamma(t)$ for $j = 0,1$ and $l \in \NM_0$ with $l \leq 5$. Applying Proposition~\ref{prop:lin1}, we obtain
\begin{align} \label{e:interp1}
\left\|\left(\partial_t - a \partial_\zeta\right)^j \partial_\zeta^l S_p^0(t)\left(\phi_0' \gamma_0\right)\right\|_\infty &\lesssim (1+t)^{-\frac{2j+l}{2}}
\end{align}
for $t \geq 0$, $j = 0,1$, and $l \in \NM_0$ with $l \leq 5$. On the other hand, the modulational estimates in Propositions~\ref{prop:lin_mod_1} and~\ref{prop:lin_mod_2} yield
\begin{align}\label{e:interp2}
\begin{split}
\left\|\left(\partial_t - a \partial_\zeta\right)^j \partial_\zeta^{l} S_p^0(t)\left(\phi_0' \gamma_0\right)\right\|_\infty &\lesssim \left(1 + t\right)^{-\frac{2j+l-1}{2}} E_0,\\
\left\|\left(\partial_t - a \partial_\zeta\right) S_p^0(t)\left(\phi_0' \gamma_0\right) - \chi'(t) \re^{\left(d\partial_\xx^2 + a\partial_\xx\right) t}\gamma_0\right\|_\infty &\lesssim E_0,\\
\chi'(t) \left\|\re^{-\partial_\xx^4 t} \gamma_0 - \re^{\left(d\partial_\xx^2 + a\partial_\xx\right) t} \gamma_0\right\|_\infty &\lesssim E_0
\end{split}
\end{align}
for $t \geq 0$, $j \in \{0,1\}$, and $l \in \NM_0$ with $1 \leq l \leq 5$. Furthermore, we note that
\begin{align*}
\begin{split}
I_{1,0}(t) &= \left(\partial_t - a \partial_\xx\right) S_p^0(t)(\phi_0'\gamma_0) - \chi'(t) \re^{\left(d\partial_\xx^2 + a\partial_\xx\right) t}\gamma_0 + \chi'(t) \left(\re^{\left(d\partial_\xx^2 + a\partial_\xx\right) t}\gamma_0 - \re^{-\partial_\xx^4 t} \gamma_0\right)\\ 
&\qquad + \, (1-\chi(t)) \left(\partial_t - a \partial_\xx\right) \re^{-\partial_\xx^4 t} \gamma_0
\end{split}
\end{align*}
for $t \geq 0$. Hence, interpolating between~\eqref{e:interp1} and~\eqref{e:interp2}, using~\eqref{e:Gamma_rates3}, and recalling $0\leq E_0 \leq 1$ and $0 < \alpha < \frac16$, we establish
\begin{align}\label{e:interp33}
\begin{split}
\left\|I_{j,l}(t)\right\|_\infty &\lesssim \left(1 + t\right)^{-\frac{2j+l}{2} + 2\alpha} E_0^{\alpha}
\end{split}
\end{align}
for $t \geq 1$, $j \in \{0,1\}$, and $l \in \NM_0$ with $1 \leq \min\{j,l\}$ and $l \leq 5$.

Finally, we bound the linear terms
\begin{align*}
J_2(t) &:= \partial_\xx \left(\widetilde{S}_r^0(t)\left(\phi_0' \gamma_0\right) + \partial_\xx \re^{\left(d\partial_\xx^2 + a\partial_\xx\right) t} \left(A_h(\phi_0') \gamma_0\right) + (1-\chi(t)) \re^{-\partial_\xx^4 t} \gamma_0\right), \\
J_3(t) &:= \partial_\xx \re^{\left(d\partial_\xx^2 + a\partial_\xx\right) t} \gamma_0,
\end{align*}
appearing in the formulas~\eqref{e:intr} and~\eqref{e:def_tilde_y} for $r_\xx(t)$ and $\widetilde{y}_\xx(t)$, respectively. By Propositions~\ref{prop:lin1} and~\ref{prop:lin2} we have
\begin{align} \label{e:interp7}
\sqrt{t} \, \|J_2(t)\|_\infty &\lesssim (1+t)^{-\frac12}, \qquad \|J_3(t)\|_\infty \lesssim (1+t)^{-\frac12}
\end{align}
for $t \geq 0$. We use~\eqref{e:adjoint},~\eqref{e:prin_rep}, and~\eqref{e:prin_rep2} to rewrite
\begin{align*}
J_2(t) &= \partial_\xx \left(S_p^0(t)\left(\phi_0' \gamma_0\right) - \re^{\left(d\partial_\xx^2 + a\partial_\xx\right) t} \left(\gamma_0 - A_h(\phi_0')\gamma_0'\right) + (1-\chi(t)) \re^{-\partial_\xx^4 t} \gamma_0\right).
\end{align*}
Applying estimates~\eqref{e:interp2} and~\eqref{e:Gamma_rates3} and Proposition~\ref{prop:lin2} to the latter, we arrive at
\begin{align} \label{e:interp8}
\sqrt{t} \, \|J_2(t)\|_\infty &\lesssim \sqrt{1+t}\, E_0, \qquad \|J_3(t)\|_\infty \lesssim E_0,
\end{align}
for $t \geq 0$. Interpolating between~\eqref{e:interp7} and~\eqref{e:interp8} and using $0 \leq E_0 \leq 1$ and $0 < \alpha < \frac16$, we obtain
\begin{align} \label{e:interp9}
\begin{split}
\sqrt{t} \, \|J_2(t)\|_\infty &\lesssim (1+t)^{-\frac12+2\alpha} E_0^{\alpha}, \qquad \|J_3(t)\|_\infty \lesssim (1+t)^{-\frac12 + \alpha} E_0^\alpha
\end{split}
\end{align}
for $t \geq 0$.

\paragraph*{Bounds on \texorpdfstring{$v(t)$}{v(t)}, \texorpdfstring{$v_\xx(t)$}{v_x(t)} and \texorpdfstring{$\partial_t \gamma(t)$}{gamma_t(t)}.} 
Let $t \in [0,\smash{\widetilde{\tau}_{\max}})$ with $\eta(t) \leq \frac{1}{2}$. We bound $v(s) = z(s) + \partial_k \phi(\cdot;k_0) \gamma_\xx(s)$ and $\partial_t \gamma(s) = \widetilde{\gamma}(s) - a\gamma_\xx(s)$ as
\begin{align}
\begin{split}
\|v(s)\|_{\infty} &\lesssim \|z(s)\|_{\infty} + \|\gamma_\xx(s)\|_{\infty} \lesssim (1+s)^{-\frac12 + \alpha} \eta(t), \\ 
\sqrt{s} \, \left\|v_\xx(s)\right\|_{\infty} &\lesssim \sqrt{s} \, \left(\left\|z_\xx(s)\right\|_{\infty} + \|\gamma_\xx(s)\|_{C_{\mathrm{ub}}^1}\right) \lesssim (1+s)^{\alpha} \eta(t),\\ 
s^{\frac34}\,\|\partial_t \gamma(s)\|_{\infty} &\lesssim s^{\frac34} \left(\|\widetilde{\gamma}(s)\|_{\infty} + \|\gamma_\xx(s)\|_{\infty}\right) \lesssim (1+s)^{\frac14 + \alpha} \eta(t)
\end{split}
\label{e:vbound} \end{align}
for $s \in [0,t]$.

\paragraph*{Bounds on \texorpdfstring{$r(t)$}{r(t)} and \texorpdfstring{$r_\xx(t)$}{r_x(t)}.} We start with bounding the nonlinear terms in the representation~\eqref{e:intr} for $r(t)$ one by one. To this end, let $t \in (0,\smash{\widetilde{\tau}_{\max}})$ with $\eta(t) \leq \frac{1}{2}$. We employ Lemma~\ref{lem:nlboundsmod3} and estimate~\eqref{e:vbound} to establish the nonlinear bounds
\begin{align}
\label{e:nlest71}
\begin{split}
&\|\mathcal Q_p(z(s),v(s),\gamma(s))\|_\infty, \|\mathcal R_p(z(s),v(s),\gamma(s),\partial_t \gamma(s))\|_\infty, \|\mathcal S_p(z(s),v(s),\gamma(s))\|_\infty\\
&\qquad \lesssim \frac{\eta(s)^2 \log(2+s)}{s^{\frac34} (1+s)^{\frac34 - 3\alpha}}
\end{split}
\end{align}
for $s \in (0,t]$. So, invoking Proposition~\ref{prop:lin1} and using $0 < \alpha < \frac16$, we obtain
\begin{align} \label{e:nlest7}
\begin{split}
\left\|\partial_\xx^{m+1} \int_0^t S_p^0(t-s) \mathcal R_p(z(s),v(s),\gamma(s),\widetilde{\gamma}(s))\de s \right\|_\infty &\lesssim \int_0^t \frac{\eta(s)^2 \log(2+s)}{(1+t-s)^{\frac{1+m}{2}} s^{\frac34} (1+s)^{\frac{3}{4}-3\alpha}} \de s\\ &\lesssim \frac{\eta(t)^2}{(1+t)^{\frac{1+m}{2}}}
\end{split}
\end{align}
and, analogously,
\begin{align} \label{e:nlest78}
\begin{split}
\left\|\partial_\xx^{m+2-i} \int_0^t S_p^i(t-s) \mathcal S_p(z(s),v(s),\gamma(s))\de s \right\|_\infty &\lesssim \frac{\eta(t)^2}{(1+t)^{\frac{1+m}{2}}}
\end{split}
\end{align}
for all $t \in [0,\smash{\widetilde{\tau}_{\max}})$ with $\eta(t) \leq \frac{1}{2}$ and $i,m \in \{0,1\}$. Furthermore, Proposition~\ref{prop:lin2} and estimate~\eqref{e:nlest71} yield
\begin{align} \label{e:nlest77}
\begin{split}
t^{\frac{m}{2}} \left\|\partial_\xx^m \int_0^t \widetilde{S}_r^0(t-s) \mathcal Q_p(z(s),v(s),\gamma(s)) \de s\right\|_\infty &\lesssim \int_0^t \frac{t^{\frac{m}{2}} \eta(s)^2 \log(2+s)}{(t-s)^{\frac{m}{2}}\sqrt{1+t-s} \, s^{\frac34} (1+s)^{\frac{3}{4}-3\alpha}} \de s\\ 
&\lesssim \frac{\eta(t)^2}{\sqrt{1+t}}
\end{split}
\end{align}
and, analogously,
\begin{align} \label{e:nlest79}
\begin{split}
t^{\frac{m}{2}}\left\|\partial_\xx^m \int_0^t \widetilde{S}_r^1(t-s) \mathcal R_p(z(s),v(s),\gamma(s),\widetilde{\gamma}(s)) \de s\right\|_\infty &\lesssim \frac{\eta(t)^2}{\sqrt{1+t}},\\
t^{\frac{m}{2}}\left\|\partial_\xx^m \int_0^t \widetilde{S}_r^2(t-s) \mathcal S_p(z(s),v(s),\gamma(s)) \de s\right\|_\infty &\lesssim \frac{\eta(t)^2}{\sqrt{1+t}}
\end{split}
\end{align}
for $m = 0,1$ and all $t \in [0,\smash{\widetilde{\tau}_{\max}})$ with $\eta(t) \leq \frac{1}{2}$. Moreover, Proposition~\ref{prop:lin2} affords the bound
\begin{align} \label{e:nlest4}
\begin{split}
\left\|\partial_\xx^m \int_0^t \widetilde{S}_r^0(t-s)\left(f_p \gamma_\xx(s)^2\right) \de s\right\|_\infty &\lesssim \int_0^t \frac{\eta(s)^2 (1+s)^{2\alpha-1}}{\sqrt{1+t-s} (t-s)^{\frac{m}{2}} } \de s \lesssim \frac{\eta(t)^2 \left(\log(2+t)\right)^m}{(1+t)^{\frac{1+m}{2}-2\alpha}}
\end{split}
\end{align}
for $m = 0,1$ and all $t \in [0,\smash{\widetilde{\tau}_{\max}})$. Using Proposition~\ref{prop:lin2}, we subsequently infer
\begin{align} \label{e:nlest5}
\begin{split}
\left\|\partial_\xx \int_0^t \re^{\left(d\partial_\xx^2 + a\partial_\xx\right) (t-s)}\left(A_h(f_p) \gamma_\xx(s)^2\right) \de s\right\|_\infty &\lesssim \int_0^t \frac{\eta(s)^2}{\sqrt{t-s} (1+s)^{1-2\alpha}} \de s \lesssim \frac{\eta(t)^2}{(1+t)^{\frac12-2\alpha}}
\end{split}
\end{align}
for all $t \in [0,\smash{\widetilde{\tau}_{\max}})$. Similarly, exploiting that $\partial_\xx$ commutes with $\smash{\re^{\left(d\partial_\xx^2 + a\partial_\xx\right) (t-s)}}$, we obtain
\begin{align} \label{e:nlest6}
\begin{split}
&\left\|\partial_\xx^2 \int_0^t \re^{\left(d\partial_\xx^2 + a\partial_\xx\right) (t-s)}\left(A_h(f_p) \gamma_\xx(s)^2\right) \de s\right\|_\infty\\
&\qquad \lesssim \int_{0}^{\max\{0,t-1\}} \frac{\eta(s)^2}{(t-s)(1+s)^{1-2\alpha}} \de s + \int_{\max\{0,t-1\}}^t \frac{\eta(s)^2}{\sqrt{t-s}(1+s)^{1-2\alpha}} \de s\\
&\lesssim \frac{\eta(t)^2 \log(2+t)}{(1+t)^{1-2\alpha}}
\end{split}
\end{align}
for all $t \in [0,\smash{\widetilde{\tau}_{\max}})$. 

Next, we establish bounds on the linear term
\begin{align*}
J_*(t) := \widetilde{S}_r^0(t) \left(v_0 + \phi_0'\gamma_0 + \gamma_0' v_0\right) + \re^{\left(d\partial_\xx^2 + a\partial_\xx\right) t} \partial_\xx \left(A_h(\phi_0') \gamma_0\right) + (1-\chi(t))\re^{-\partial_\xx^4 t} \gamma_0
\end{align*}
in~\eqref{e:intr}. First, we note that 
\begin{align*}
\partial_\xx J_*(t) = J_2(t) + \partial_\xx \widetilde{S}_r^0(t)\left(v_0 + \gamma_0' v_0\right)
\end{align*}
for $t \geq 0$. Hence, using Proposition~\ref{prop:lin2}, the estimates~\eqref{e:Gamma_rates3},~\eqref{e:v0bound},~\eqref{e:interp8} and~\eqref{e:interp9}, and the facts that $0 \leq E_0 \leq 1$ and $0 < \alpha < \frac16$, we obtain
\begin{align} \label{e:linr1}
\left\|J_*(t)\right\|_\infty \lesssim \frac{1}{\sqrt{1+t}}, \qquad \sqrt{t} \, \left\|\partial_\xx J_*(t)\right\|_\infty \lesssim \sqrt{1+t} \, E_0, \qquad \sqrt{t} \, \left\|\partial_\xx J_*(t)\right\|_\infty \lesssim \frac{E_0^\alpha}{(1+t)^{\frac12 - 2\alpha}}
\end{align}
for $t \geq 0$.

Finally, combining the nonlinear bounds~\eqref{e:nlest7},~\eqref{e:nlest78},~\eqref{e:nlest77},~\eqref{e:nlest79},~\eqref{e:nlest4},~\eqref{e:nlest5}, and~\eqref{e:nlest6} with the linear bounds~\eqref{e:linr1}, we estimate the right-hand side of~\eqref{e:intr} by
\begin{align} \label{e:nlest8}
\begin{split}
\left\|r(t)\right\|_\infty &\lesssim \frac{1}{(1+t)^{\frac12 - 2\alpha}} \qquad \sqrt{t} \,  \|r_\xx(t)\|_\infty \lesssim\left(E_0^\alpha + \eta(t)^2\right)\frac{\log(2+t)}{(1+t)^{\frac12 - 2\alpha}}
\end{split}
\end{align}
and obtain
\begin{align} \label{e:nlest88}
\begin{split}
\left\|r(t) - J_*(t)\right\|_\infty \lesssim \frac{\eta(t)^2}{(1+t)^{\frac12 - 2\alpha}}
\end{split}
\end{align}
for all $t \in [0,\smash{\widetilde{\tau}_{\max}})$ with $\eta(t) \leq \frac{1}{2}$.

\paragraph*{Short-time bounds on \texorpdfstring{$y(t)$}{y(t)} and \texorpdfstring{$y_\xx(t)$}{y_x(t)}.} Recalling that the cut-off function $\chi(t)$ and the propagator $S_p^0(t)$ vanish on $[0,1]$, we find by~\eqref{e:intgamma},~\eqref{e:decomp_gamma}, and~\eqref{e:defy} that 
\begin{align}y(t) = \exp\left(\frac{\nu}{d}\left(\re^{-\partial_\xx^4 t} \gamma_0 - r(t)\right)\right) \label{e:yexpr} 
\end{align}
for $t \in [0,1]$. Hence,~\eqref{e:Gamma_rates3} and~\eqref{e:nlest8} yield an $E_0$-independent constant $K_* > 0$ such that
\begin{align} \label{e:yshort}
y(\xx,t) &\leq K_*, \qquad 1 \leq K_* y(\xx,t)
\end{align}
and
\begin{align} \label{e:yshort2}
\sqrt{t} \, \|y_\xx(t)\|_\infty \leq K_* \left(E_0^\alpha + \eta(t)^2\right) 
\end{align}
hold for all $\xx \in \R$ and $t \in [0,\smash{\widetilde{\tau}_{\max}})$ with $t \leq 1$ and $\eta(t) \leq \frac{1}{2}$, where we used $0 \leq E_0 \leq 1$ and $0 < \alpha < \frac16$.

\paragraph*{Bounds on \texorpdfstring{$y(t)$}{y(t)} and \texorpdfstring{$y_\xx(t)$}{y_x(t)}.} First, we bound the linear term in~\eqref{e:inty}. For this purpose, we assume that $1 \in [0,\smash{\widetilde{\tau}_{\max}})$ and $\eta(1) \leq \frac 12$ and decompose~\eqref{e:yexpr} at $t = 1$ as
\begin{align*}
y(1) = y_1 + y_2
\end{align*}
with
\begin{align*}
y_1 = \exp\left(\frac{\nu}{d}\left(\re^{-\partial_\xx^4 1} \gamma_0 - J_*(1)\right)\right)\left(\re^{\frac{\nu}{d} \left(J_*(1) - r(1)\right)} - 1\right), \qquad y_2 = \exp\left(\frac{\nu}{d}\left(\re^{-\partial_\xx^4 1} \gamma_0 - J_*(1)\right)\right).
\end{align*}
Applying the mean value theorem and the estimates~\eqref{e:Gamma_rates3},~\eqref{e:linr1}, and~\eqref{e:nlest88}, we obtain
\begin{align*}
\|y_1\|_\infty \lesssim \eta(1)^2, \qquad \|y_2\|_\infty \lesssim 1, \qquad \left\|y_2'\right\|_\infty \lesssim E_0.
\end{align*}
Thus, using the latter bounds and estimate~\eqref{e:yshort}, applying Proposition~\ref{prop:lin2}, and recalling $0 \leq E_0 \leq 1$ and $0 < \alpha < \frac16$, we obtain
\begin{align} \label{e:nlest91} 
\begin{split}
\left\|\re^{\left(d\partial_\xx^2 + a\partial_\xx\right) (t-1)} y(1)\right\|_\infty &\leq K_*,\\
\left\|\partial_\xx \re^{\left(d\partial_\xx^2 + a\partial_\xx\right) (t-1)} y(1)\right\|_\infty &\lesssim \frac{\|y_2\|_\infty^{1-2\alpha} \left\|y_2'\right\|_\infty^{2\alpha} + \|y_1\|_\infty + \left\|y_2'\right\|_\infty}{(1+t)^{\frac12 - \alpha}} \lesssim \frac{E_0^\alpha + \eta(t)^2}{(1+t)^{\frac12-\alpha}}
\end{split}
\end{align}
for all $t \in [0,\smash{\widetilde{\tau}_{\max}})$ with $t \geq 1$ and $\eta(t) \leq \frac{1}{2}$, where we used that $\eta$ is monotonically increasing.

We proceed with bounding the nonlinear term in~\eqref{e:inty}. Take $t \in [0,\smash{\widetilde{\tau}_{\max}})$ with $t \geq 1$ and $\eta(t) \leq \frac{1}{2}$. Define 
\begin{align*}
Y(t) := \sup_{1 \leq s \leq t} \|y(s)\|_\infty.
\end{align*}
Lemma~\ref{lemma_nonlinear_bound_on_G} and estimate~\eqref{e:vbound} give rise to the nonlinear bound
\begin{align*}
\|\Non_c(r(s),y(s),z(s),v(s),\gamma(s),\widetilde{\gamma}(s))\|_\infty \lesssim \eta(s)^2\left(1+ Y(t)\right)\frac{\log(2+s)}{\left(1+s\right)^{\frac{3}{2}-3\alpha}}
\end{align*}
for $s \in [1,t]$, where we use $\eta(t) \leq \frac{1}{2}$. Combining the latter with Proposition~\ref{prop:lin2}, we find a constant $R_* > 0$ such that
\begin{align} \label{e:nlest92}
\begin{split}
&\left\|\partial_\xx^m \int_1^t \re^{\left(d\partial_\xx^2 + a\partial_\xx\right) (t-s)} \Non_c(r(s),y(s),z(s),v(s),\gamma(s),\widetilde{\gamma}(s)) \de s\right\|_\infty \\
&\qquad \lesssim \int_1^t \frac{\eta(s)^2(1+Y(t))\log(2+s)}{(t-s)^{\frac{m}{2}} (1+s)^{\frac{3}{2}-3\alpha}} \de s \leq R_* \frac{\eta(t)^2(1+Y(t))}{(1+t)^{\frac{m}{2}}}.
\end{split}
\end{align}
for $m = 0,1$ and all $t \in [0,\smash{\widetilde{\tau}_{\max}})$ with $t \geq 1$ and $\eta(t) \leq \frac{1}{2}$. Taking suprema with respect to $t$ in~\eqref{e:inty} and using the estimates~\eqref{e:nlest91} and~\eqref{e:nlest92}, we find an $E_0$-independent constant $C_* > 0$ such that
\begin{align*}
Y(t) \leq C_*\left(1 + \eta(t)^2\left(1+Y(t)\right)\right)
\end{align*}
for all $t \in [0,\smash{\widetilde{\tau}_{\max}})$ with $t \geq 1$ and $\eta(t) \leq \frac{1}{2}$. This implies
\begin{align} \label{e:Ytbound}
Y(t) \leq 2C_*
\end{align}
for all $t \in [0,\smash{\widetilde{\tau}_{\max}})$ with $t \geq 1$ and $\eta(t) \leq \frac{1}{2}\min\{1,1/\sqrt{C_*}\}$. Moreover, the lower bound in~\eqref{e:yshort} and the fact that $\smash{\re^{\left(d\partial_\xx^2 + a\partial_\xx\right) t}}$ is a positive operator yield
\begin{align} \label{e:lowerybound}
0 \leq \re^{\left(d\partial_\xx^2 + a\partial_\xx\right) t} \left(K_*y(1) - 1\right) = K_*\re^{\left(d\partial_\xx^2 + a\partial_\xx\right) t} y(1) - 1
\end{align}
for $t \geq 1$. Finally, applying~\eqref{e:nlest91},~\eqref{e:nlest92},~\eqref{e:Ytbound}, and~\eqref{e:lowerybound} to~\eqref{e:inty} and using the short-time bounds~\eqref{e:yshort} and~\eqref{e:yshort2}, we obtain an $E_0$-independent constant $M_* > 0$ such that
\begin{align} \label{e:nlest9}
\begin{split}
y(\xx,t) \leq M_*, \qquad M_* y(\xx,t) \geq 1, \qquad \sqrt{t} \, \|y_\xx(t)\|_\infty \leq M_* \left(E_0^\alpha + \eta(t)^2\right)(1+t)^{\alpha}
\end{split}
\end{align}
for all $\xx \in \R$ and $t \in [0,\smash{\widetilde{\tau}_{\max}})$ with 
$$\eta(t) \leq \frac{1}{2}\min\left\{1,\frac{1}{\sqrt{C_*}},\frac{1}{\sqrt{K_* R_*(1+2C_*)}}\right\} =: \eta_0.$$

\paragraph*{Short-time bounds on \texorpdfstring{$\gamma(t)$}{gamma(t)} and its derivatives.} We use that $\gamma(t) = \smash{\re^{-\partial_\xx^4 t} \gamma_0}$ for $t \in [0,\smash{\widetilde{\tau}_{\max}})$ with $t \leq 1$ by Proposition~\ref{p:gamma}. Hence,~\eqref{e:Gamma_rates3} yields
\begin{align} \label{e:shorttimegamma}
\|\gamma(t)\|_\infty \lesssim 1, \qquad t^{\frac{j}4} \left\|\partial_\xx^j \gamma_{\xx}(t)\right\|_\infty \lesssim E_0, \qquad t^{\frac34} \left\|\widetilde{\gamma}(t)\right\|_\infty \lesssim E_0
\end{align}
for $j = 0,1,2$ and $t \in [0,\smash{\widetilde{\tau}_{\max}})$ with $t \leq 1$. In addition, it implies
\begin{align} \label{e:shorttimegamma2}
\left\|\gamma_{\xx\xx\xx\xx}(t)\right\|_{C_{\mathrm{ub}}^1} \lesssim E_0, \qquad  \left\|\widetilde{\gamma}_{\xx}(t)\right\|_{C_{\mathrm{ub}}^2} \lesssim E_0
\end{align}
for $t \in [0,\smash{\widetilde{\tau}_{\max}})$ with $\frac{t_*}{2} \leq t \leq 1$.

\paragraph*{Bounds on \texorpdfstring{$\gamma(t)$}{gamma(t)} and \texorpdfstring{$\gamma_\xx(t)$}{gamma_\xx(t)}.} We first consider the case $\nu \neq 0$. Recalling~\eqref{e:decomp_gamma} and~\eqref{e:defy}, we take the spatial derivative of
\begin{align*} \gamma(t) = r(t) + \frac{d}{\nu} \log(y(t)),\end{align*}
yielding
\begin{align*} \gamma_\xx(t) = r_\xx(t) + \frac{d y_\xx(t)}{\nu y(t)}.\end{align*}
We note that by~\eqref{e:nlest8} and~\eqref{e:nlest9} the above expressions are well-defined and we deduce 
\begin{align*}
\|\gamma(t)\|_\infty \leq \|r(t)\|_\infty + \frac{d}{\nu} \left|\log(M_*)\right| \lesssim 1
\end{align*}
and
\begin{align} \label{e:gammabounding}
\|\gamma_\xx(t)\|_\infty \leq \|r_\xx(t)\|_\infty + \frac{d M_*}{\nu} \|y_\zeta(t)\|_\infty \lesssim \frac{E_0^\alpha + \eta(t)^2}{(1+t)^{\frac12 - \alpha}}
\end{align}
for $t \in [0,\smash{\widetilde{\tau}_{\max}})$ with $t \geq 1$ and $\eta(t) \leq \eta_0$. Combining the latter with the short-time bounds~\eqref{e:shorttimegamma} yields
\begin{align} \label{e:nlest11} \|\gamma(t)\|_\infty \lesssim 1, \qquad 
\|\gamma_\xx(t)\|_\infty \lesssim \frac{E_0^\alpha + \eta(t)^2}{(1+t)^{\frac12 - \alpha}}, \qquad \|\gamma_\xx(0)\|_\infty \lesssim E_0^\alpha\end{align}
for $t \in [0,\smash{\widetilde{\tau}_{\max}})$ with $\eta(t) \leq \eta_0$, where we use $0 \leq E_0 \leq 1$ and $0 < \alpha < \frac16$. 

We proceed with the case $\nu = 0$. Recalling~\eqref{e:decomp_gamma}, we establish bounds on $\widetilde{y}(t)$. First, we invoke Proposition~\ref{prop:lin2} and estimates~\eqref{e:v0bound} and~\eqref{e:interp9} to bound the linear term 
\begin{align*}
I_*(t) := S_h^0(t) \left(v_0 + \gamma_0' v_0\right) + 
 \re^{\left(d\partial_\xx^2 + a\partial_\xx\right) t}\left(\gamma_0 - A_h(\phi_0') \gamma_0'\right)
\end{align*}
in~\eqref{e:def_tilde_y} by
\begin{align} \label{e:shorttildey}
\|I_*(t)\|_\infty \lesssim 1, \qquad \left\|\partial_\xx I_*(t)\right\|_\infty \lesssim \frac{E_0^\alpha}{(1+t)^{\frac12 - \alpha}}
\end{align}
for $t \in [0,\smash{\widetilde{\tau}_{\max}})$ with $t \geq 1$, where we used $0 \leq E_0 \leq 1$ and $0 < \alpha < \frac16$. Next, we bound the nonlinear terms in~\eqref{e:def_tilde_y}. With the aid of Proposition~\ref{prop:lin2} and estimate~\eqref{e:nlest71} we establish
\begin{align} \label{e:nlest70}
\begin{split}
\left\|\partial_\xx^m \int_0^t S_h^0(t-s) \mathcal Q_p(z(s),v(s),\gamma(s))\de s\right\|_\infty & \lesssim \int_0^t \frac{\eta(s)^2\log(2+s)}{(t-s)^{\frac{m}{2}} s^{\frac34} (1+s)^{\frac{3}{4} - 3\alpha}} \de s \lesssim \frac{\eta(t)^2}{(1+t)^{\frac{m}{2}}},
\end{split}
\end{align}
and, analogously,
\begin{align} \label{e:nlest72}
\begin{split}
&\left\|\partial_\xx^m\! \int_0^t S_h^1(t-s) \mathcal R_p(z(s),v(s),\gamma(s),\widetilde{\gamma}(s))\de s\right\|_\infty, 
\left\|\partial_\xx^m\! \int_0^t S_h^2(t-s) \mathcal S_p(z(s),v(s),\gamma(s))\de s\right\|_\infty, \\
&\left\|\partial_\xx^m \int_0^t \re^{\left(d\partial_\xx^2 + a\partial_\xx\right) (t-s)} \left(A_h(f_p) \partial_\xx \left(\gamma_\xx(s)^2\right)\right) \de s\right\|_\infty \lesssim \frac{\eta(t)^2}{(1+t)^{\frac{m}{2}}},
\end{split}
\end{align}
for $m = 0,1$ and $t \in [0,\smash{\widetilde{\tau}_{\max}})$ with $t \geq 1$ and $\eta(t) \leq \frac12$. Applying~\eqref{e:shorttildey},~\eqref{e:nlest70}, and~\eqref{e:nlest72} to~\eqref{e:def_tilde_y}, we conclude
\begin{align*}
\left\|\widetilde{y}(t)\right\|_\infty \lesssim 1, \qquad \left\|\widetilde{y}_\xx(t)\right\|_\infty \lesssim \frac{E_0^\alpha + \eta(t)^2}{(1+t)^{\frac12 - \alpha}}
\end{align*}
for $t \in [0,\smash{\widetilde{\tau}_{\max}})$ with $t \geq 1$ and $\eta(t) \leq \frac12$.
Combining the latter with the short-time estimates~\eqref{e:shorttimegamma} and the estimates~\eqref{e:nlest8} on $r(t)$ for large times, we bound~\eqref{e:decomp_gamma}, arriving at~\eqref{e:nlest11} for $t \in [0,\smash{\widetilde{\tau}_{\max}})$ with $\eta(t) \leq \frac12$.

\paragraph*{Bounds on \texorpdfstring{$\widetilde{\gamma}(t)$}{tildegamma(t)}, \texorpdfstring{$\gamma_{\xx\xx}(t)$}{gamma_xx(t)}, and their derivatives.} First, we use Proposition~\ref{prop:lin1} and the estimates~\eqref{e:v0bound} and~\eqref{e:interp33} to bound the linear term in~\eqref{e:intgamma} as
\begin{align} \label{e:lingamma}
\left\|(\partial_t - a\partial_\xx)^j \partial_\xx^l S_p^0(t)\left(v_0 + \phi_0'\gamma_0 + \gamma_0' v_0\right)\right\|_\infty \lesssim (1+t)^{-\frac{2j+l}{2} + 2\alpha} E_0^\alpha
\end{align}
for $t \geq 1$, $j \in \{0,1\}$, and $l \in \NM_0$ with $1 \leq \min\{j,l\}$ and $l \leq 5$, where we used $0 \leq E_0 \leq 1$ and $0 < \alpha < \frac16$.

Next, we bound the nonlinear terms in~\eqref{e:intgamma}. To this end, let $t \in (0,\smash{\widetilde{\tau}_{\max}})$ with $\eta(t) \leq \frac{1}{2}$. We invoke Lemma~\ref{lemma_nonlinear_bound_on_N} and employ the estimate~\eqref{e:vbound} to obtain
\begin{align} \label{e:nlest100}
\begin{split}
\|\mathcal Q(v(s),\gamma(s))\|_{\infty}, \left\|\mathcal R(v(s),\gamma(s),\widetilde{\gamma}(s))\right\|_{\infty}, \left\|\partial_\xx^j \mathcal S(v(s),\gamma(s))\right\|_{\infty} &\lesssim \frac{\eta(t)^2}{s^{\frac34} (1+s)^{\frac14 - 2\alpha}}
\end{split}
\end{align}
for $s \in (0,t]$ and $j = 0,1$. So, using Proposition~\ref{prop:lin1} and the facts that $S_p(t)$ vanishes on $[0,1]$ and we have $0 < \alpha < \frac16$, we infer 
\begin{align}
\begin{split}
\left\|(\partial_t - a\partial_\xx)^j \partial_\xx^l \int_0^t S_p^0(t-s) \mathcal{N}\left(v(s),\gamma(s),\partial_t \gamma(s)\right) \de s\right\|_\infty & \lesssim \int_0^t \frac{\eta(s)^2}{(1+t-s) s^{\frac34} (1+s)^{\frac14 - 2\alpha}} \de s\\
&\lesssim \frac{\eta(t)^2\log(2+t)}{(1+t)^{1 - 2\alpha}},\end{split} \label{e:nlest22}
\end{align}
for all $t \in [0,\smash{\widetilde{\tau}_{\max}})$ with $\eta(t) \leq \frac{1}{2}$ and $j,l \in \mathbb{N}_0$ with $2 \leq l + 2j \leq 5$. Thus, applying the linear bound~\eqref{e:lingamma} and the nonlinear bound~\eqref{e:nlest22} to the right-hand side of~\eqref{e:intgamma}, we obtain
\begin{align*}
\begin{split}
\left\|(\gamma_{\xx\xx}(t),\widetilde{\gamma}(t))\right\|_{C_{\mathrm{ub}}^3 \times C_{\mathrm{ub}}^3} \lesssim \left(E_0^\alpha + \eta(t)^2\right)\frac{\log(2+t)}{(1+t)^{1-2\alpha}},
\end{split}
\end{align*}
for all $t \in [0,\smash{\widetilde{\tau}_{\max}})$ with $t \geq 1$ and $\eta(t) \leq \frac{1}{2}$. Combining the latter with the short-time bounds~\eqref{e:shorttimegamma} and~\eqref{e:shorttimegamma2} and using that $0 \leq E_0 \leq 1$ and $0 < \alpha < \frac16$, we arrive at
\begin{align} \label{e:nlest3}
\begin{split}
t^{\frac14} \left\|\gamma_{\xx\xx}(t)\right\|_\infty \lesssim \left(E_0^\alpha + \eta(t)^2\right)\frac{\log(2+t)}{(1+t)^{\frac34-2\alpha}}, \qquad \sqrt{t} \,\left\|\gamma_{\xx\xx\xx}(t)\right\|_\infty \lesssim \left(E_0^\alpha + \eta(t)^2\right)\frac{\log(2+t)}{(1+t)^{\frac12-2\alpha}},
\end{split}
\end{align}
and
\begin{align} \label{e:nlest333}
\begin{split}
t^{\frac34} \left\|\widetilde{\gamma}(t)\right\|_\infty &\lesssim \left(E_0^\alpha + \eta(t)^2\right)\frac{\log(2+t)}{(1+t)^{\frac14-2\alpha}}, \\
\varrho(t)\left\|(\gamma_{\xx\xx\xx\xx}(t),\widetilde{\gamma}_\xx(t))\right\|_{C_{\mathrm{ub}}^1 \times C_{\mathrm{ub}}^2} &\lesssim \left(E_0^\alpha + \eta(t)^2\right)\frac{\log(2+t)}{(1+t)^{1-2\alpha}}
\end{split}
\end{align}
for all $t \in [0,\smash{\widetilde{\tau}_{\max}})$ with $\eta(t) \leq \frac{1}{2}$.

\paragraph*{Short-time bounds on \texorpdfstring{$\vt(t)$}{v(t)} and its derivatives.} Let $t \in (0,\smash{\widetilde{\tau}_{\max}})$ with $t \leq 1$ and $\eta(t) \leq \frac{1}{2}$. Using Lemma~\ref{lem:equivalence} and estimates~\eqref{e:vbound} and~\eqref{e:shorttimegamma}, we deduce
\begin{align} \label{e:vtbound}
\left\|\vt(s)\right\|_\infty \lesssim \left\|v(s)\right\|_\infty + \left\|\gamma_\xx(s)\right\|_\infty \lesssim \frac{\eta(t)}{(1+s)^{\frac12 - \alpha}}
\end{align}
for $s \in [0,t]$. Thus, Lemma~\ref{lemma_nonlinear_bound_on_mathringN} and estimates~\eqref{e:shorttimegamma} and~\eqref{e:vtbound} yield the nonlinear bound
\begin{align*}
\left\|\mathring{\mathcal N}(\vt(s),\gamma(s),\partial_t \gamma(s))\right\|_\infty &\lesssim \frac{E_0^\alpha + \eta(t)^2}{s^{\frac34}}
\end{align*}
for $s \in (0,t]$, where we used $0 \leq E_0 \leq 1$ and $0 < \alpha < \frac16$. Hence, applying the latter estimate and the Green's functions bounds in Proposition~\ref{prop:pointwiseGreen} to the Duhamel formula~\eqref{e:intmathringv}, we establish
\begin{align} \label{e:nlvbound}
\begin{split}
t^{\frac{j}{2}}\left|\partial_\xx^j \vt(\xx,t)\right| &\lesssim  \int_\R \frac{\re^{-\frac{(\xx - \xt)^2}{Mt}}}{\sqrt{t}} E_0\, \de \xt + t^{\frac{j}{2}} \int_0^t \int_\R \frac{\re^{-\frac{(\xx - \xt)^2}{M(t-s)}}}{(t-s)^{\frac{1+j}{2}} s^{\frac34}} \left(E_0^\alpha + \eta(t)^2\right) \de \xt \de s\\
&\lesssim E_0 + \int_0^t \frac{t^{\frac{j}{2}}\left(E_0^\alpha + \eta(t)^2\right) }{(t-s)^{\frac{j}{2}}  s^{\frac34}} \de s \lesssim E_0^\alpha + \eta(t)^2
\end{split}
\end{align}
for $j = 0,1$, $\xx \in \R$, and $t \in [0,\smash{\widetilde{\tau}_{\max}})$ with $t \leq t_*$ and $\eta(t) \leq \frac{1}{2}$, where we used $0 \leq E_0 \leq 1$ and $0 < \alpha < \frac16$. 

To prepare for the application of the nonlinear damping estimate from Proposition~\ref{prop:nonlinear_damping}, we also derive a bound on the second derivative of $\vt(t)$ at time $t = t_*$. For this purpose, we assume $t_* \in [0,\smash{\widetilde{\tau}_{\max}})$ with $\eta(t_*) \leq \frac{1}{2}$. Applying Lemma~\ref{lemma_nonlinear_bound_on_mathringN} once again, while using the estimates~\eqref{e:shorttimegamma},~\eqref{e:shorttimegamma2}, and~\eqref{e:nlvbound}, we arrive at
\begin{align*}
\left\|\mathring{\mathcal N}_1(\vt(s),\gamma(s),\partial_t \gamma(s))\right\|_\infty &\lesssim E_0^\alpha + \eta(t)^2
\end{align*}
for $s \in [\frac{t_*}{2},t_*]$, where we used $0 \leq E_0 \leq 1$ and $0 < \alpha < \frac16$. Combining the latter with Proposition~\ref{prop:pointwiseGreen}, we bound the spatial derivative of~\eqref{e:intmathringw} as
\begin{align} \label{e:nlwbound}
\begin{split}
\left|\vt_{\xx\xx}(\xx,t_*)\right| &\lesssim  \int_\R \re^{-\frac{2(\xx - \xt)^2}{Mt_*}}\left(E_0^\alpha + \eta(t_*)^2\right) \de \xt + \int_{\frac{t_*}{2}}^{t_*} \int_\R \frac{\re^{-\frac{(\xx - \xt)^2}{M(t_*-s)}}}{(t_*-s)} \left(E_0^\alpha + \eta(t_*)^2\right) \de \xt \de s\\ 
&\lesssim E_0^\alpha + \eta(t_*)^2
\end{split}
\end{align}
for $\xx \in \R$. 

\paragraph*{Short-time bound on \texorpdfstring{$z(t)$}{z(t)}.} Estimates~\eqref{e:shorttimegamma} and~\eqref{e:nlvbound} and Lemma~\ref{lem:equivalence} yield
\begin{align} \label{e:shortz}
\left\|z(t)\right\|_\infty \lesssim \left\|v(t)\right\|_\infty + \left\|\gamma_\zeta(t)\right\|_\infty \lesssim \left\|\vt(t)\right\|_\infty + \left\|\gamma_\zeta(t)\right\|_\infty \lesssim E_0^\alpha + \eta(t)^2 
\end{align}
for $t \in [0,\smash{\widetilde{\tau}_{\max}})$ with $t \leq t_*$ and $\eta(t) \leq \frac{1}{2}$, where we used $0 \leq E_0 \leq 1$ and $0 < \alpha < \frac16$. 

\paragraph*{Bounds on \texorpdfstring{$z(t)$}{z(t)}.} Using Proposition~\ref{prop:lin1} and the estimates~\eqref{e:Gamma_rates3},~\eqref{e:v0bound}, and~\eqref{e:interp6}, we bound the linear term
\begin{align*}
J_0(t) := \widetilde{S}(t)\left(v_0 + \phi_0'\gamma_0 + \gamma_0' v_0\right) + \left(1 - \chi(t)\right)\left(\phi_0' + k_0 \partial_k \phi(\cdot;k_0) \partial_\xx\right) \re^{-\partial_\xx^4 t} \gamma_0
\end{align*}
in~\eqref{e:intz} as
\begin{align} \label{e:linz}
\left\|J_0(t)\right\|_\infty \lesssim \frac{E_0^\alpha}{(1+t)^{1-2\alpha}}
\end{align}
for $t \geq 0$, where we used $0 \leq E_0 \leq 1$ and $0 < \alpha < \frac16$. On the other hand, using Proposition~\ref{prop:lin1} and estimates~\eqref{e:vbound} and~\eqref{e:nlest100}, we bound the nonlinear term in~\eqref{e:intz} as
\begin{align} \label{e:nonlz}
\begin{split}
&\left\|\int_0^t\widetilde{S}(t-s)\mathcal{N}(v(s),\gamma(s),\partial_t \gamma(s)) \de s + \gamma_\xx(t)v(t)\right\|_\infty \\
&\qquad\lesssim \frac{\eta(t)^2}{(1+t)^{1-2\alpha}} + \int_0^t \left(1 + \frac{1}{\sqrt{t-s}}\right) \frac{\eta(t)^2}{(1+t-s) s^{\frac34} (1+s)^{\frac14 - 2 \alpha}} \de s \lesssim \frac{\eta(t)^2 \log(2+t)}{(1+t)^{1 - 2\alpha}} 
\end{split}
\end{align}
for $t \in [0,\smash{\widetilde{\tau}_{\max}})$ with $t \geq t_*$ and $\eta(t) \leq \frac{1}{2}$. Combining the short-time bound~\eqref{e:shortz} with the long-time estimates~\eqref{e:linz} and~\eqref{e:nonlz}, we arrive at
\begin{align} \label{e:finalz}
\|z(t)\|_\infty \lesssim \left(E_0^\alpha + \eta(t)^2\right)\frac{\log(2+t)}{(1+t)^{1-2\alpha}}, \qquad \|z(0)\|_\infty \lesssim E_0^\alpha
\end{align}
for $t \in [0,\smash{\widetilde{\tau}_{\max}})$ with $\eta(t) \leq \frac{1}{2}$.

\paragraph*{Short-time bounds on \texorpdfstring{$\mathring{z}(t)$}{ringz(t)}, \texorpdfstring{$\mathring{z}_\xx(t)$}{ringz_x(t)}, and \texorpdfstring{$z_\xx(t)$}{z_x(t)}.} Recalling the identities~\eqref{e:defringv} and~\eqref{e:defringz} and using that we have $\|\gamma_\xx(t)\|_{\infty} < r_0$ by Proposition~\ref{p:gamma}, we find that the mean-value theorem yields
\begin{align} \label{e:equivvz}
\begin{split}
&t^{\frac{j}{2}} \left|\left\|\zt(t)\right\|_{C_{\mathrm{ub}}^j} - \left\|\vt(t)\right\|_{C_{\mathrm{ub}}^j}\right| \lesssim 
t^{\frac{j}{2}} \left\|\zt(t) - \vt(t)\right\|_{C_{\mathrm{ub}}^j}\\
&\qquad \lesssim t^{\frac{j}{2}} \left(\left\|\gamma_\xx(t)\right\|_{C_{\mathrm{ub}}^j} + \left\|\gamma_\xx(t)\right\|_\infty\left(\left\|\gamma(t)\right\|_\infty \left\|\phi_0'\right\|_{C_{\mathrm{ub}}^j} +  \sup_{k \in [k_0-r_0,k_0+r_0]} \left\|\partial_k \phi(\cdot;k)\right\|_{C_{\mathrm{ub}}^j}\right)\right)
\end{split}
\end{align}
for $j = 0,1,2$ and $t \in [0,\smash{\widetilde{\tau}_{\max}})$ with $\eta(t) \leq \frac12$. Applying the estimates~\eqref{e:shorttimegamma},~\eqref{e:nlvbound}, and~\eqref{e:nlwbound} to~\eqref{e:equivvz} and using $0 \leq E_0 \leq 1$ and $0 < \alpha < \frac16$, we establish
\begin{align} \label{e:ztest}
t^{\frac{j}{2}} \left\|\zt(t)\right\|_{C_{\mathrm{ub}}^j} \lesssim E_0^\alpha + \eta(t)^2
\end{align}
for $j = 0,1$ and $t \in [0,\smash{\widetilde{\tau}_{\max}})$ with $t \leq t_*$ and $\eta(t) \leq \frac12$, and
\begin{align} \label{e:ztest3}
 \left\|\zt(t_*)\right\|_{C_{\mathrm{ub}}^2} \lesssim E_0^\alpha + \eta(t_*)^2 
\end{align}
provided $t_* \in [0,\smash{\widetilde{\tau}_{\max}}]$ with $\eta(t_*) \leq \frac12$.

\paragraph*{Bounds on \texorpdfstring{$\mathring{z}(t)$}{ringz(t)}, \texorpdfstring{$\mathring{z}_\xx(t)$}{ringz_x(t)} and \texorpdfstring{$z_\xx(t)$}{z_x(t)}.} Lemma~\ref{lem:equivalence} and~\eqref{e:nlest11},~\eqref{e:nlest3}, and~\eqref{e:finalz} yield
\begin{align} \label{e:ztest2}
\left\|\zt(t)\right\|_\infty \lesssim \left(E_0^\alpha + \eta(t)^2\right)\frac{\log(2+t)}{(1+t)^{1-2\alpha}}
\end{align}
for $t \in [0,\smash{\widetilde{\tau}_{\max}})$ with $\eta(t) \leq \eta_0$. 
Combining~\eqref{e:nlest11},~\eqref{e:nlest3},~\eqref{e:nlest333},~\eqref{e:ztest3}, and~\eqref{e:ztest2} with the nonlinear damping estimate in Proposition~\ref{prop:nonlinear_damping}, we deduce
\begin{align} \label{e:nondamp}
\begin{split}
\big\|\zt(t)\big\|_{C_{\mathrm{ub}}^1} &\lesssim \big\|\zt(t)\big\|_\infty + \left(\re^{t_*-t} \left(E_0^\alpha + \eta(t_*)^2\right)^2 + \int_{t_*}^t \frac{\re^{s-t}\left(\left(E_0^\alpha + \eta(t)^2\right) \log(2+s)\right)^2}{(1+s)^{2 - 4\alpha}}\de s\right)^{\frac12}\\
&\lesssim \left(E_0^\alpha + \eta(t)^2\right)\frac{\log(2+t)}{(1+t)^{1-2\alpha}}
\end{split}
\end{align}
for $t \in [0,\smash{\widetilde{\tau}_{\max}})$ with $t \geq t_*$ and $\eta(t) \leq \eta_0$. Linking the short- and long-time estimates~\eqref{e:ztest} and~\eqref{e:nondamp}, we thus obtain
\begin{align} \label{e:finalztder}
\sqrt{t} \, \left\|\zt_\xx(t)\right\|_\infty \lesssim \left(E_0^\alpha + \eta(t)^2\right)\frac{\log(2+t)}{(1+t)^{\frac12-2\alpha}}
\end{align}
for $t \in [0,\smash{\widetilde{\tau}_{\max}})$ with $\eta(t) \leq \eta_0$. Finally, we apply Lemma~\ref{lem:equivalence} and use~\eqref{e:nlest11},~\eqref{e:nlest3}, and~\eqref{e:finalztder} to establish
\begin{align} \label{e:finalzder}
\sqrt{t} \, \left\|z_\xx(t)\right\|_\infty \lesssim \left(E_0^\alpha + \eta(t)^2\right)\frac{\log(2+t)}{(1+t)^{\frac12-2\alpha}}
\end{align}
 for $t \in [0,\smash{\widetilde{\tau}_{\max}})$ with $\eta(t) \leq \eta_0$.

\paragraph*{Bounds on \texorpdfstring{$\mathring{v}(t)$}{ringz(t)}, \texorpdfstring{$\mathring{v}_\xx(t)$}{ringz_x(t)}.} Combining~\eqref{e:nlest11},~\eqref{e:nlest3},~\eqref{e:ztest2},~\eqref{e:equivvz}, and~\eqref{e:finalztder}, we arrive at
\begin{align} \label{e:finalvt}
\left\|\vt(t)\right\|_\infty \lesssim \frac{E_0^\alpha + \eta(t)}{(1+t)^{\frac12-\alpha}}, \qquad 
\sqrt{t} \, \left\|\vt_\xx(t)\right\|_\infty \lesssim \left(E_0^\alpha + \eta(t)\right)(1+t)^{\alpha}
\end{align}
for $t \in [0,\smash{\widetilde{\tau}_{\max}})$ with $\eta(t) \leq \eta_0$.

\paragraph*{Proof of key estimates.} Combining~\eqref{e:nlest8},~\eqref{e:nlest9},~\eqref{e:nlest11},~\eqref{e:nlest3},~\eqref{e:nlest333},~\eqref{e:finalz}, and~\eqref{e:finalzder}, we conclude that there exists a constant $C > 0$, independent of $E_0$, such that the key estimates~\eqref{e:etaest} hold for all $t \in [0, \smash{\widetilde{\tau}_{\max}})$ with $\eta(t) \leq \eta_0$. Now take $E_0 \in (0, \varepsilon)$. As previously argued, this implies
\begin{align} \label{e:etaest2}
\eta(t) \leq 2C E_0^\alpha < \eta_0, \qquad \|\gamma(s)\|_\infty \leq C
\end{align}
for all $t \in [0, \smash{\widetilde{\tau}_{\max}})$, thereby ruling out the blow-up scenarios~\eqref{e:blowupgamma1} and~\eqref{e:blowupgamma2}, and hence ensuring $\smash{\widetilde{\tau}_{\max}} = \tau_{\max} = T_{\max}$. So,~\eqref{e:finalvt} shows that $u(t)$ remains uniformly bounded on $[0, T_{\max})$, excluding the blow-up alternative~\eqref{e:blowupu} and establishing that $\tau_{\max} = T_{\max} = \infty$. Finally, the bounds~\eqref{e:ztest2},~\eqref{e:finalztder},~\eqref{e:finalvt}, and~\eqref{e:etaest2} yield the estimates~\eqref{e:mtest10} and~\eqref{e:mtest2}.

\paragraph*{Optimal temporal decay rates.} It remains to obtain the estimates~\eqref{e:mtest10} and~\eqref{e:mtest2} for the case $\alpha = 0$. Here, we make use of the fact that we have established~
\begin{align} \label{e:etaest3}
\eta(t) \leq 2C E_0^{\tilde\alpha}, \qquad t \geq 0,
\end{align}
for a fixed $\tilde\alpha \in (0,\frac16)$. First, we apply  Proposition~\ref{prop:lin2} and employ~\eqref{e:yshort},~\eqref{e:nlest92}, and~\eqref{e:etaest3} to bound the right-hand side of~\eqref{e:inty} as
\begin{align*}
\|y_\xx(t)\|_\infty \lesssim \frac{1}{\sqrt{1+t}}
\end{align*}
for $t \geq 1$. The latter, together with~\eqref{e:gammabounding} and~\eqref{e:etaest3}, yields
\begin{align} \label{e:opt1}
\|\gamma_\xx(t)\|_\infty \lesssim \frac{1}{\sqrt{1+t}}
\end{align}
for $t \geq 1$. Therefore, Proposition~\ref{lemma_nonlinear_bound_on_N} and estimates~\eqref{e:vbound} and~\eqref{e:etaest3} result in the nonlinear bounds
\begin{align*}
\begin{split}
\|\mathcal Q(v(s),\gamma(s))\|_{\infty}, \left\|\mathcal R(v(s),\gamma(s),\widetilde{\gamma}(s))\right\|_{\infty}, \left\|\partial_\xx^j \mathcal S(v(s),\gamma(s))\right\|_{\infty} &\lesssim \frac{1}{s^{\frac34} (1+s)^{\frac14}},\\
\|\gamma_\xx(s)v(s)\|_\infty \lesssim \frac{1}{1+s}
\end{split}
\end{align*}
for all $s > 0$ and $j = 0,1$. This shows that the estimates~\eqref{e:nlest22} and~\eqref{e:nonlz} hold for $\alpha = 0$ and $t \geq 1$ (with $\eta(t)$ replaced by $1$). Applying the latter bounds, estimates~\eqref{e:Gamma_rates3} and~\eqref{e:v0bound}, and Proposition~\ref{prop:lin1} to~\eqref{e:intgamma} and~\eqref{e:intz}, we arrive at
\begin{align} \label{e:opt2}
\|z(t)\|_\infty, \left\|\left(\gamma_{\xx\xx}(t),\widetilde{\gamma}(t)\right)\right\|_{C_{\mathrm{ub}}^3 \times C_{\mathrm{ub}}^3} \lesssim \frac{\log(2+t)}{1+t}
\end{align}
for $t \geq t_*$. Next, we observe that Lemma~\ref{lem:equivalence} and the estimates~\eqref{e:nlest11},~\eqref{e:opt1}, and~\eqref{e:opt2} yield
\begin{align} \label{e:opt3}
\left\|\zt(t)\right\|_\infty \lesssim \frac{\log(2+t)}{1+t}
\end{align}
for $t \geq t_*$. Combining~\eqref{e:nlest11},~\eqref{e:ztest3},~\eqref{e:opt1},~\eqref{e:opt2}, and~\eqref{e:opt3} with Proposition~\ref{prop:nonlinear_damping}, we find that estimate~\eqref{e:nondamp} also holds for $\alpha = 0$ and $t \geq t_*$ (with $\eta(t)$ replaced by $1$). This leads to the bound
\begin{align} \label{e:opt4}
\left\|\zt(t)\right\|_{C_{\mathrm{ub}}^1}\lesssim \frac{\log(2+t)}{1+t}
\end{align}
for $t \geq t_*$. Finally,~\eqref{e:nlest11},~\eqref{e:equivvz},~\eqref{e:opt1},~\eqref{e:opt2}, and~\eqref{e:opt4} yield
\begin{align} \label{e:opt5}
\left\|\vt(t)\right\|_{C_{\mathrm{ub}}^1}\lesssim \frac{1}{\sqrt{1+t}}
\end{align}
for $t \geq t_*$. It follows from~\eqref{e:opt1},~\eqref{e:opt2},~\eqref{e:opt3},~\eqref{e:opt4}, and~\eqref{e:opt5} that the estimates~\eqref{e:mtest10} and~\eqref{e:mtest2} also hold for the case $\alpha = 0$.

\paragraph*{Approximation by the viscous Hamilton-Jacobi equation.} 
We begin with establishing auxiliary bounds for short time. We use Propositions~\ref{prop:lin_mod_1} and~\ref{prop:lin_mod_2} to establish the modulational bounds
\begin{align}\label{e:interp22}
\begin{split}
\left\|S_p^0(t)\left(\phi_0' \gamma_0\right) - \chi(t) \re^{\left(d\partial_\xx^2 + a\partial_\xx\right) t}\gamma_0\right\|_\infty &\lesssim E_0,\\
(1-\chi(t)) \left\|\re^{-\partial_\xx^4 t} \gamma_0 - \re^{\left(d\partial_\xx^2 + a\partial_\xx\right) t} \gamma_0\right\|_\infty &\lesssim E_0
\end{split}
\end{align}
for $t \geq 0$. On the other hand, given $\alpha \in (0,\frac16)$, Propositions~\ref{prop:lin1} and~\ref{prop:lin2} in combination with the estimates~\eqref{e:nlest100} and~\eqref{e:etaest2} yield the nonlinear bounds
\begin{align} \label{e:nlest222}
\begin{split}
\left\|\int_0^t S_p^0(t-s) \mathcal{N}\left(v(s),\gamma(s),\partial_t \gamma(s)\right) \de s\right\|_\infty & \lesssim \int_0^t \frac{\eta(s)^2}{s^{\frac34} (1+s)^{\frac14 - 2\alpha}} \de s \lesssim E_0^{2\alpha}(1+t)^{2\alpha}\\
\left\|\partial_\xx^m \int_0^t \re^{\left(d\partial_\xx^2 + a \partial_\xx\right)(t-s)} \gamma_\xx(s)^2 \de s\right\|_\infty & \lesssim \int_0^t \frac{\eta(s)^2}{(t-s)^{\frac{m}{2}} (1+s)^{1 - 2\alpha}} \de s \lesssim \frac{E_0^{2\alpha}}{(1+t)^{\frac{m}{2} - 2 \alpha}}
\end{split} 
\end{align}
for $t \geq 0$ and $m = 0,1$.

Having established the auxiliary bounds~\eqref{e:interp22} and~\eqref{e:nlest222}, we follow the approach of~\cite{BjoernMod} and distinguish between the cases $\nu = 0$ and $\nu \neq 0$. We begin with the case $\nu \neq 0$. Since $\smash{\re^{(d\partial_\xx^2 + a\partial_\xx)t}}$ is a positive operator, we have the pointwise estimate
\begin{align} \label{e:pointwise}
\re^{-\frac{\nu}{d} M} \leq \breve{y}(t) \leq \re^{\frac{\nu}{d} M}, \qquad \breve{y}(t) := \re^{(d\partial_\xx^2 + a\partial_\xx)t}\left(\re^{\frac{\nu}{d}\gamma_0}\right) 
\end{align}
for all $t \geq 0$. One readily verifies that the function $\breve{\gamma} \in \mathcal{Y}$ defined by
\begin{align*} 
\breve \gamma(t) = \frac{d}{\nu}\log\left(\breve{y}(t)\right)
\end{align*}
is a classical global solution to the viscous Hamilton-Jacobi equation~\eqref{e:HamJac} with initial condition $\breve \gamma(0) = \gamma_0 \in C_{\mathrm{ub}}^1(\R)$. Hence, it satisfies the Duhamel formula
\begin{align} \label{e:intgamma3}
\breve{\gamma}(t) = \re^{(d\partial_\xx^2 + a\partial_\xx)t} \gamma_0 + \nu \int_0^t \re^{(d\partial_\xx^2 + a\partial_\xx)(t-s)} \left(\breve{\gamma}_\xx(s)^2\right) \de s 
\end{align}
for $t \geq 0$. Moreover, the pointwise bound~\eqref{e:pointwise} and Proposition~\ref{prop:lin2} yield
\begin{align} \label{e:brevegammabound}
\left\|\breve{\gamma}(t)\right\|_\infty \leq M, \qquad \left\|\breve{\gamma}_\xx(t)\right\|_\infty \lesssim E_0, \qquad \left\|\breve{\gamma}_\xx(t)\right\|_\infty \lesssim \frac{1}{\sqrt{1+t}}
\end{align}
for all $t \geq 0$. We observe that the bounds on $\gamma(t)$ and $\breve{\gamma}(t)$ in~\eqref{e:mtest10} and~\eqref{e:brevegammabound} readily imply the estimates~\eqref{e:mtest3} and~\eqref{e:mtest33} in case $\alpha = 0$. Therefore, we focus on the case $\alpha \in (0,\frac16)$. By Proposition~\ref{prop:lin2} and the bound~\eqref{e:brevegammabound}, we obtain
\begin{align} \label{e:nlest777}
\left\|\partial_\xx^m \int_0^t \re^{\left(d\partial_\xx^2 + a\partial_\xx\right) (t-s)} \left(\breve{\gamma}_\xx(s)^2\right)\de s\right\|_\infty \lesssim \int_0^t \frac{E_0^{2\alpha}}{(t-s)^{\frac{m}{2}} (1+s)^{1-\alpha}} \de s \lesssim \frac{E_0^{2\alpha}}{(1+t)^{\frac{m}{2}-\alpha}}
\end{align}
for $m = 0,1$ and $t \geq 0$. Hence, applying Proposition~\ref{prop:lin1}, employing estimates~\eqref{e:v0bound},~\eqref{e:interp22},~\eqref{e:nlest222}, and~\eqref{e:nlest777}, using $0 \leq E_0 \leq 1$ and $0 < \alpha < \frac16$, and inserting the Duhamel formulas~\eqref{e:intgamma} and~\eqref{e:intgamma3}, we arrive at the short-time bound
\begin{align} \label{e:shortbrevegamma}
\left\|\gamma(t) - \breve{\gamma}(t)\right\|_\infty \lesssim E_0^{2\alpha}(1+t)^{2\alpha}
\end{align}
for all $t \geq 0$. On the other hand, using Proposition~\ref{prop:lin2}, the identities~\eqref{e:def_tilde_y} and~\eqref{e:intgamma3}, and the estimates~\eqref{e:v0bound},~\eqref{e:nlest70},~\eqref{e:nlest72},~\eqref{e:etaest2},~\eqref{e:nlest222}, and~\eqref{e:nlest777}, we deduce
\begin{align*}
\left\|\partial_\xx^m \left(\widetilde{y}(1) - \breve{\gamma}(1)\right)\right\|_\infty \lesssim E_0^\alpha
\end{align*}
for $m = 0,1$, where we use $0 \leq E_0 \leq 1$ and $0 < \alpha < \frac16$. Hence, using the mean-value theorem and the bound~\eqref{e:brevegammabound}, we obtain
\begin{align*}
\left\|\partial_\xx^m \left(y(1) - \breve{y}(1)\right)\right\|_\infty = \left\|\partial_\xx^m \left(\re^{\frac{\nu}{d} \widetilde{y}(1)} - \re^{\frac{\nu}{d} \breve{\gamma}(1)}\right)\right\|_\infty \lesssim E_0^\alpha
\end{align*}
for $m = 0,1$. Combining this with the identities~\eqref{e:inty} and $\breve{y}(t) = \smash{\re^{(d\partial_\xx^2 + a\partial_\xx)(t-1)}} \breve{y}(1)$, Proposition~\ref{prop:lin2}, and the estimates~\eqref{e:nlest92} and~\eqref{e:etaest2}, we infer
\begin{align} \label{e:intfinaly} 
\begin{split}
\left\|\partial_\xx^m \left(y(t) - \breve{y}(t)\right)\right\|_\infty &\lesssim \frac{E_0^\alpha}{(1+t)^{\frac{m}{2}}}
\end{split}
\end{align}
for $m = 0,1$ and $t \geq 0$. Finally, applying the mean-value theorem, using the identities~\eqref{e:defy} and~\eqref{e:decomp_gamma}, and the bounds~\eqref{e:nlest8},~\eqref{e:nlest9},~\eqref{e:etaest2},~\eqref{e:pointwise}, and~\eqref{e:intfinaly}, we obtain
\begin{align} \label{e:brevegammabounds3}
\begin{split}
\left\|\gamma(t) - \breve{\gamma}(t)\right\|_\infty &\lesssim \|r(t)\|_\infty + \left\|y(t) - \breve{y}(t)\right\|_\infty \lesssim E_0^\alpha + \frac{1}{(1+t)^{\frac12 - 2\alpha}},\\
\sqrt{t} \, \left\|\gamma_\xx(t) - \breve{\gamma}_\xx(t)\right\|_\infty &\lesssim \sqrt{t} \, \left(\|r_\xx(t)\|_\infty + \left\|y_\xx(t) - \breve{y}_\xx(t)\right\|_\infty\right) \lesssim E_0^\alpha
\end{split}
\end{align}
for all $t \geq 0$, where we use $0 \leq E_0 \leq 1$ and $0 < \alpha < \frac16$. Thus, using~\eqref{e:shortbrevegamma} for $\smash{E_0^{\frac43}} (1+t)^2 \leq 1$ and the first bound in~\eqref{e:brevegammabounds3} for $(1+t)^{-2} \leq \smash{E_0^{\frac43}}$, we establish the first inequality in~\eqref{e:mtest3}, where we use $0 < \alpha < \frac16$. On the other hand, using~\eqref{e:shorttimegamma} and~\eqref{e:brevegammabound} for short times $t \in [0,1]$ and the lower bound in~\eqref{e:brevegammabounds3} for large times $t \geq 1$, we obtain the second inequality in~\eqref{e:mtest3}. Finally, estimate~\eqref{e:mtest33} follows from~\eqref{e:mtest3} with the aid of the mean-value theorem.

We now consider the case $\nu = 0$. Then, the function $\breve{\gamma} \in \mathcal{Y}$ defined by
\begin{align*} 
\breve \gamma(t) = \re^{(d\partial_\xx^2 + a\partial_\xx)t} \gamma_0
\end{align*}
is a classical global solution to~\eqref{e:HamJac} with initial condition $\breve{\gamma}(0) = \gamma_0 \in C_{\mathrm{ub}}^1(\R)$. Again it follows from Proposition~\ref{prop:lin2} and the bounds~\eqref{e:mtest10} that the estimates~\eqref{e:mtest3} and~\eqref{e:mtest33} are trivially satisfied in the case $\alpha = 0$, allowing us to focus on the case $\alpha \in (0,\frac16)$. On the one hand, Proposition~\ref{prop:lin1}, the estimates~\eqref{e:v0bound},~\eqref{e:interp22}, and~\eqref{e:nlest222}, and the representation~\eqref{e:intgamma} yield the short-time bound
\begin{align} \label{e:shortbrevegamma4}
\left\|\gamma(t) - \breve{\gamma}(t)\right\|_\infty \lesssim E_0^{2\alpha}(1+t)^{2\alpha}
\end{align}
for all $t \geq 0$, where we used $0 \leq E_0 \leq 1$ and $0 < \alpha < \frac16$. On the other hand, invoking the identities~\eqref{e:decomp_gamma} and~\eqref{e:def_tilde_y}, applying Proposition~\ref{prop:lin2}, and using the bounds~\eqref{e:v0bound},~\eqref{e:nlest8},~\eqref{e:nlest70},~\eqref{e:nlest72}, and~\eqref{e:etaest2}, we obtain
\begin{align*} 
\begin{split}
\left\|\gamma(t) - \breve{\gamma}(t)\right\|_\infty &\lesssim \|r_\xx(t)\|_\infty + \left\|\widetilde{y}(t) - \breve{\gamma}(t)\right\|_\infty \lesssim E_0^\alpha + \frac{1}{(1+t)^{\frac12 - 2\alpha}},\\
\sqrt{t} \, \left\|\gamma_\xx(t) - \breve{\gamma}_\xx(t)\right\|_\infty &\lesssim \|r_\xx(t)\|_\infty + \left\|\widetilde{y}_\xx(t) - \breve{\gamma}_\xx(t)\right\|_\infty \lesssim E_0^\alpha
\end{split}
\end{align*}
for $t \in [0,1]$, where we again use $0 \leq E_0 \leq 1$ and $0 < \alpha < \frac16$. Similarly as in the case $\nu \neq 0$, we combine the latter with the short-time bounds~\eqref{e:shorttimegamma},~\eqref{e:shortbrevegamma4}, and $\|\breve{\gamma}_\xx(t)\|_\infty \lesssim E_0$ (cf.~Proposition~\ref{prop:lin2}). This results in the estimate~\eqref{e:mtest3}, which once more implies~\eqref{e:mtest33} via the mean-value theorem.
\end{proof}

\section{Discussion and outlook} \label{sec:discussion}

We discuss the wider applicability of our method and outline potential directions for future research. 

\subsection{Adaptations to other dissipative systems}

We anticipate that the modulational stability framework developed in this paper extends beyond the class of reaction-diffusion systems and applies to diffusively spectrally stable wave trains in other dissipative problems, provided that the linearization about the wave train generates a $C_0$-semigroup on $C_{\mathrm{ub}}(\R)$ whose high-frequency component is damped. A key observation is that the structure of the critical low-frequency component $S_p^0(t)$ of the semigroup is determined by the diffusive spectral stability assumptions~\ref{assD1}-\ref{assD3}, rather than the specific structure of the underlying equation. As a result, the linear estimates on modulational data derived in~\S\ref{sec:mod_data} are expected to carry over to a broader class of systems. This suggests that our method may even extend to certain dissipative \emph{quasilinear} problems, provided that one obtains sufficient regularity control within the nonlinear iteration scheme. Such control may be established via nonlinear damping estimates, as in~\cite{STVenant2}. Alternatively, if the quasilinear equation is parabolic, one may use pointwise Green's function estimates~\cite{howard,ZUH}; see~\cite[Section~6.6]{BjoernMod} for further discussion.

An interesting and more delicate challenge arises when additional conservation laws are present, as in the St.~Venant equations for shallow water waves~\cite{STVenant1}. In such cases, the spectrum of the linearization about the wave train possesses an additional critical mode at the origin, thereby violating the spectral assumption~\ref{assD3}. This changes the nature of the leading-order modulational dynamics: instead of being governed by the scalar viscous Hamilton-Jacobi equation~\eqref{e:HamJac}, it is described by a Whitham modulation system that captures interactions among multiple critical modes; see~\cite{JNRZ14}. This precludes a straightforward application of the Cole-Hopf transform and poses a significant obstacle for extending the current analysis.

\subsection{Extension to multiple spatial dimensions}

We conjecture that our modulational stability results extend to roll waves in reaction-diffusion systems in higher spatial dimensions. These roll waves arise by trivially extending a one-dimensional wave train in the transverse directions. We believe that the core techniques employed in our analysis, such as the Cole-Hopf transform and the detailed decomposition of the linear semigroup, remain valid in this higher-dimensional setting. While we do not expect fundamental differences in the modulational behavior of roll waves compared to the one-dimensional case, we do not anticipate that our framework can be lightly adapted to study fully nonlocalized modulations of the planar periodic patterns studied in~\cite{melinand2024}, or their higher-dimensional counterparts. The reason is that the leading-order dynamics of the  modulations of such periodic patterns, which are nontrivial in any spatial direction, are governed by systems that differ qualitatively from the viscous Hamilton-Jacobi equation. These may include hyperbolic-parabolic systems, models with cross-diffusion, or anisotropic systems with dispersive effects; see~\cite{melinand2024} for more details. Developing an $L^\infty$-theory for such periodic patterns therefore remains an open problem for future research.

\subsection{Unbounded initial phase modulations and wavenumber offsets} \label{sec:unbounded}

We expect that it may be possible to allow initial data $\gamma_0 \in \mathrm{BMO}(\mathbb{R})$ with $\gamma_0' \in C_{\mathrm{ub}}(\mathbb{R})$, where $\mathrm{BMO}(\mathbb{R})$ denotes the space of functions of bounded mean oscillation. This expectation is motivated by the estimate $\smash{\| \partial_\xx \re^{t \partial_\xx^2} g \|_\infty \lesssim t^{-\frac12} \|g\|_{\mathrm{BMO}}}$, which enables interpolation on the critical linear terms in the Duhamel formulation. In particular, this suggests that the initial data $\gamma_0$ may be spatially unbounded. We note that, for the application of the Cole-Hopf transform in case $\nu \neq 0$, it may be necessary to additionally assume that $\smash{\re^{\frac{\nu}{d} \gamma_0}} \in \mathrm{BMO}(\mathbb{R})$. This condition still permits $\gamma_0$ to be spatially unbounded. Furthermore, restricting to  $\nu = 0$, we expect to allow for algebraic growth of $\gamma_0$ at rate $|x|^{\beta}$ as $x \to \pm \infty$ by assuming $L^{p}$-localization of the derivative $\gamma_0'$. Here, $\beta>0$ is sufficiently small and $\frac{1}{1-\beta} < p < \infty$.

A further open question is whether wavenumber offsets can be permitted in settings where the phase modulation grows linearly at spatial infinity. For plane wave solutions in the real Ginzburg-Landau equation, this question has been answered affirmatively in~\cite{BK1,GALMI}. We refer to~\cite[Section~6.3]{BjoernMod} for further discussion.

\appendix

\section{Technical lemmas for linear estimates on modulational data} \label{app:aux}

This appendix is devoted to technical estimates underlying the $L^\infty$-bounds on modulational data in~\S\ref{sec:mod_data}. Our first result is a variant of the low-frequency estimate established in~\cite[Lemma~A.1]{BjoernMod}, addressing the case of a spatially periodic integral kernel. Following the treatment of modulational data in~\cite{JONZNL,JUNNL}, we make use of a Fourier series expansion of the periodic kernel.

\begin{lemma}\label{lem:semigroupEstimate1}
Let $m_1 \in \mathbb N_0$, $m_2 \in \{0,1\}$, and $\xi_0 \in (0,\pi]$. Let $\lambda \in C\big((-\xi_0,\xi_0),\C\big)$ and $F \in C\big(\R^3 \times [1,\infty),\C\big)$. Suppose that there exist constants $C, \mu > 0$ such that
\begin{itemize}
\item[i)] $\lambda'(0) \in \ri \R$;
\item[ii)] $\Re \, \lambda(\xi) \leq -\mu \xi^2$ for all $\xi \in (-\xi_0,\xi_0)$;
\item[iii)] $\overline{\mathrm{supp}(F(\cdot,\xx,\xt,t))} \subset (-\xi_0,\xi_0)$ for all $\xx,\xt \in \R$ and $t \geq 1$;
\item[iv)] $F$ is twice continuously differentiable in its first argument and $1$-periodic in its third argument such that $\smash{\big\|\partial_\xi^\ell F(\xi,\xx,\cdot,t)\big\|_{L^2(0,1)} \leq C t^{\frac{\ell}{2}}}$ for $\xi \in (-\xi_0,\xi_0)$, $\xx \in \R$, $t \geq 1$, and $\ell = 0,1,2$;
\item[v)] if $m_2 = 0$, then $\int_0^1 F(\xi,\xx,\xt,t) \de \xt = 0$ for all $\xi \in (-\xi_0,\xi_0)$, $\xx \in \R$, and $t \geq 1$. 
\end{itemize}
Let $a \in \R$ be such that $\lambda'(0) = a\ri$. Then, the estimate
\begin{align*}
\left\|\int_\R \int_\R \re^{t \lambda(\xi)} \xi^{m_1+m_2} F(\xi,\cdot,\xt,t) \re^{\ri \xi(\cdot-\xt)} v(\xt) \de \xi \de \xt \right\|_\infty &\lesssim t^{-\frac{m_1}{2}} \|v'\|_\infty
\end{align*}
holds for each $t \geq 1$ and $v \in C_{\mathrm{ub}}^1(\R,\R)$. 
\end{lemma}
\begin{proof} 
Expanding the $1$-periodic function $F(\xi,\xx,\cdot,t)$ as a Fourier series and subsequently integrating by parts, we obtain
\begin{align} \label{e:decomp_tech}
\begin{split}
&\int_\R \int_\R \re^{t \lambda(\xi)} \xi^{m_1+m_2} F(\xi,\xx,\xt,t) \re^{\ri \xi(\xx-\xt)} v(\xt) \de \xi \de \xt\\
&\qquad = \sum_{j \in \Z} \int_\R \int_\R \re^{t \lambda(\xi)} \xi^{m_1+m_2} \big\langle F(\xi,\xx,\cdot,t),\re^{2\pi \ri j \cdot}\big\rangle_{L^2(0,1)}  \re^{\ri \xi \xx + (2\pi j - \xi) \ri \xt} v(\xt) \de \xi \de \xt\\
&\qquad = \ri \int_\R \int_\R \re^{t \lambda(\xi)} \left(\xi^{m_1+m_2} h_0(\xi,\xx,\xt,t) - \xi^{m_1} h_1(\xi,\xx,\xt,t)\right) \re^{\ri \xi (\xx - \xt)} \de \xi \,  v'(\xt) \de \xt
\end{split}
\end{align}
for $\xx \in \R$, $t \geq 1$, and $v \in C_{\mathrm{ub}}^1(\R,\R)$, where we denote
\begin{align*}
h_0(\xi,\xx,\xt,t) := \sum_{j \in \Z \setminus \{0\}} \frac{\big\langle F(\xi,\xx,\cdot,t),\re^{2\pi \ri j \cdot}\big\rangle_{L^2(0,1)}}{2\pi j - \xi} \re^{2\pi \ri j \xt}
\end{align*}
and where
\begin{align*}
h_1(\xi,\xx,\xt,t) := \big\langle F(\xi,\xx,\cdot,t),1\big\rangle_{L^2(0,1)} = \int_0^1 F(\xi,\xx,y,t) \de y 
\end{align*}
vanishes identically in case $m_2 = 0$, by hypothesis~v). By Hypothesis~iv), we have
\begin{align} \label{e:h1est}
\left\|\partial_\xi^\ell h_1(\xi,\xx,\cdot,t)\right\|_\infty \leq  \left\|\partial_\xi^\ell F(\xi,\xx,\cdot,t)\right\|_{L^2(0,1)} \lesssim t^{\frac{\ell}{2}}
\end{align}
for $\xi \in (-\xi_0,\xi_0)$, $\xx \in \R$, $t \geq 1$, and $\ell = 0,1,2$. On the other hand, hypothesis~iv) in combination with an application of H\"older's and Bessel's inequality yields
\begin{align}  \label{e:hest1}
\begin{split}
\left\|h_0(\xi,\xx,\cdot,t)\right\|_{\infty} &\leq \sum_{j \in \Z \setminus \{0\}} \left|\frac{\big\langle F(\xi,\xx,\cdot,t),\re^{2\pi \ri j \cdot}\big\rangle_{L^2(0,1)}}{2\pi j - \xi}\right|\\
&\leq \left\|F(\xi,\xx,\cdot,t)\right\|_{L^2(0,1)} \left(\sum_{j \in \Z \setminus \{0\}} \frac{1}{(2\pi j - \xi)^2}\right)^{\frac12} \lesssim 1
\end{split}
\end{align}
for $\xi \in (-\xi_0,\xi_0)$, $\xx \in \R$, and $t \geq 1$. Similarly, computing
\begin{align*}
\partial_\xi^{\ell} \left(\frac{1}{2\pi j - \xi}\right) = \frac{\ell!}{(2\pi j - \xi)^{\ell+1}}
\end{align*}
for $j \in \Z \setminus \{0\}$ and using hypothesis~iv), we bound
\begin{align} \label{e:hest2}
\left\|\partial_\xi^\ell h_0(\xi,\xx,\cdot,t)\right\|_\infty \lesssim \sum_{m = 0}^\ell \left\|\partial_\xi^{\ell - m} F(\xi,\xx,\cdot,t)\right\|_{L^2(0,1)} \left(\sum_{j \in \Z \setminus \{0\}} \frac{1}{(2\pi j - \xi)^{2(m+1)}}\right)^{\frac12} \lesssim C t^{\frac{\ell}{2}}
\end{align}
for $\xi \in (-\xi_0,\xi_0)$, $\xx \in \R$, $t \geq 1$, and $\ell = 1,2$. 

Let $m \in \NM_0$ and $\ell \in \{0,1\}$. Proceeding as in~\cite[Lemma~A.1]{BjoernMod}, we use integration by parts to rewrite the integrand
\begin{align*}
I(\xx,\xt,t) := \int_\R \re^{t \lambda(\xi)} \xi^m h_\ell(\xi,\xx,\xt,t) \re^{\ri \xi (\xx - \xt)} \de \xi
\end{align*}
in~\eqref{e:decomp_tech} as 
\begin{align} \label{eq:split3}
I(\xx,\xt,t) = \left(1+\frac{(\xx-\xt+ a t)^2}{t}\right)^{-1} \left(I(\xx,\xt,t) + I_2(\xx,\xt,t)\right)
\end{align}
for $t \geq 1$ and $\xx,\xt \in \R$, where we denote
\begin{align*}
I_2(\xx,\xt,t) = \frac{1}{t} \int_\R \partial_\xi^2 \left(\re^{t \left(\lambda(\xi) - \lambda'(0) \xi\right)}  \xi^m h_\ell(\xi,\xx,\xt,t)\right) \re^{\ri \xi(\xx - \xt + a t)} \de \xi.
\end{align*}
To bound the first integral on the right-hand side of~\eqref{eq:split3}, we use hypotheses~ii) and~iii) and estimates~\eqref{e:h1est} and~\eqref{e:hest1}, yielding
\begin{align} \label{e:Iest}
\left|I(\xx,\xt,t)\right| \lesssim \int_{-\xi_0}^{\xi_0} \left|\xi^m \re^{t \lambda(\xi)}\right| \de \xi \lesssim \int_\R |\xi|^m \re^{-\mu \xi^2 t} \de \xi \lesssim t^{-\frac{m+1}{2}}
\end{align}
for $t \geq 1$ and $\xx,\xt \in \R$. To bound the second integral on the right-hand side of~\eqref{eq:split3}, we first observe that, by hypotheses~i) and~ii), estimates~\eqref{e:h1est},~\eqref{e:hest1}, and~\eqref{e:hest2}, and the fact that $\lambda \in C^2\big((-\xi_0,\xi_0),\C\big)$, it holds 
\begin{align*}
\begin{split}
&\left|\partial_\xi^2 \left(\re^{t (\lambda(\xi)-\lambda'(0)\xi)} \xi^m h_\ell(\xi,\xx,\xt,t)\right)\right|\\ 
&\qquad\lesssim \left(|\xi|^m t \left(1 + |\xi| \sqrt{t} + \xi^2 t\right) + m |\xi|^{m-1} \sqrt{t} \left(1 + |\xi|\sqrt{t}\right) + m(m-1) |\xi|^{m-2}\right)\re^{-\mu \xi^2 t}
\end{split}
\end{align*}
for $\xi \in (-\xi_0,\xi_0)$, $t \geq 1$ and $\xx,\xt \in \R$. Therefore, using hypotheses~ii) and~iii), we establish
\begin{align} \label{e:I2est}
\begin{split}
\left|I_2(\xx,\xt,t)\right| &\lesssim \int_\R \re^{-\mu \xi^2 t} \bigg(|\xi|^m \left(1 + |\xi| \sqrt{t} + \xi^2 t\right)\\ &\qquad \qquad  \qquad + \, \frac{m}{\sqrt{t}} \, |\xi|^{m-1} \left(1 + |\xi|\sqrt{t}\right) + \frac{m(m-1)}{t} |\xi|^{m-2}\bigg)\de \xi \lesssim t^{-\frac{m+1}{2}}
\end{split}
\end{align}
for $\xx,\xt \in \R$ and $t \geq 1$. Applying the estimates~\eqref{e:Iest} and~\eqref{e:I2est} to~\eqref{eq:split3}, we arrive at
\begin{align*}
\left|I(\xx,\xt,t)\right| \leq t^{-\frac{m+1}{2}}\left(1+\frac{(\xx-\xt+ a t)^2}{t}\right)^{-1}
\end{align*}
for $\xx,\xt \in \R$ and $t \geq 1$. Finally, we use this pointwise bound to estimate the right-hand side of~\eqref{e:decomp_tech}, yielding
\begin{align*}
\left|\int_\R \int_\R \re^{t \lambda(\xi)} \xi^{m_1+m_2} F(\xi,\xx,\xt,t) \re^{\ri \xi(\xx-\xt)} v(\xt) \de \xi \de \xt \right| &\lesssim \int_\R t^{-\frac{m_1 + 1}{2}} \left(1+\frac{(\xx-\xt+ a t)^2}{t}\right)^{-1} \|v'\|_\infty \de \xt\\ 
&\lesssim t^{-\frac{m_1}{2}} \|v'\|_\infty
\end{align*}
for $\xx \in \R$, $t \geq 1$, and $v \in C_{\mathrm{ub}}^1(\R,\R)$, which concludes the proof.
\end{proof} 

The second technical estimate is a slight modification of the high-frequency bound established in~\cite[Lemma~A.2]{HDRS22}.

\begin{lemma}
\label{lemma_higher_order_low}
Let $\xi_0,d > 0$ and $a \in \R$. Let $F \in C_{\mathrm{ub}}^2(\R,\R)$ be supported on $\R \setminus [-\xi_0,\xi_0]$. Then, there exists a constant $\mu_0 > 0$ such that we have
\begin{align*}
\left\|\int_\R\int_\R \re^{\ri \xi(\cdot-\xt) + \left(a \ri \xi - d\xi^2\right)t} F(\xi) \de \xi\, v(\xt) \de \xt\right\|_{\infty} \lesssim  \left(1 + \frac{1}{\sqrt{t}}\right)\re^{-\mu_0 t} \|v\|_{\infty}
\end{align*}
for $v \in C_{\mathrm{ub}}(\R,\R)$ and $t > 0$.
\end{lemma}
\begin{proof}
We follow the strategy of the proof of~\cite[Lemma~A.2]{HDRS22} and use integration by parts to rewrite the integrand
\begin{align*}
I(\xx,\xt,t) = \int_\R \re^{\ri \xi(\cdot-\xt) + \left(a \ri \xi - d\xi^2\right)t} F(\xi) \de \xi
\end{align*}
as 
\begin{align} \label{e:rewrite_technical}
I(\xx,\xt,t) = \frac{1}{1 + \left(\xx - \xt + at\right)^2} \left(I(\xx,\xt,t) +  \int_\R \partial_\xi^2\left(e^{-d\xi^2 t}F(\xi)\right) \re^{\ri\xi(\xx-\xt+at)}\de \xi\right)
\end{align}
for $\xx,\xt \in \R$ and $t > 0$. We compute
\begin{align*}
\partial_\xi^2\left(\re^{-d\xi^2 t}F(\xi)\right) = \re^{-d t \xi^2} \left(F''(\xi)-4 d t \xi F'(\xi)+2 d t F(\xi) \left(2 d t \xi^2-1\right)\right)
\end{align*}
for $\xi \in\R$ and $t > 0$. We estimate
\begin{align*}
\left|I(\xx,\xt,t)\right| &\lesssim \int_{\R \setminus [-\xi_0,\xi_0]} \re^{-\frac{d}{2}\xi^2 t - \frac{d}{2}\xi_0^2 t} \|F\|_{\infty} \lesssim \frac{1}{\sqrt{t}} \re^{- \frac{d}{2}\xi_0^2 t}\|F\|_{\infty} ,\\
\left|\int_\R \partial_\xi^2\left(e^{-d\xi^2 t}F(\xi)\right) \re^{\ri\xi(\xx-\xt+at)}\de \xi\right| 
&\lesssim \int_{\R \setminus [-\xi_0,\xi_0]} \re^{-\frac{d}{2}\xi^2 t - \frac{d}{2}\xi_0^2 t} \left(1 + t + |\xi| t + \xi^2 t^2\right) \|F\|_{C_{\mathrm{ub}}^2} \de \xi\\ &\lesssim \left(\sqrt{t} + \frac{1}{\sqrt{t}}\right)\re^{-\frac{d}{2}\xi_0^2 t} \|F\|_{C_{\mathrm{ub}}^2}
\end{align*}
for $\xx,\xt \in \R$ and $t > 0$. Applying these estimates to the identity~\eqref{e:rewrite_technical}, we arrive at
\begin{align*}
\left|\int_\R \int_\R I(\xx,\xt,t) \de \xi\, v(\xt)\de \xt\right| &\lesssim
\int_\R \left(\sqrt{t} + \frac{1}{\sqrt{t}}\right)\frac{\re^{-\frac{d}{2}\xi_0^2 t} \|v\|_\infty}{1 + \left(\xx-\xt + at\right)^2} \de \xt \lesssim \left(1 + \frac{1}{\sqrt{t}}\right)\re^{-\frac{d}{4}\xi_0^2 t} \|v\|_\infty
\end{align*}
for $\xx \in \R$, $v \in C_{\mathrm{ub}}(\R,\R)$, and $t > 0$, which yields the desired bound with $\mu_0 = \frac{d}{4}\xi_0^2$. 
\end{proof}

\section{Local existence of the phase modulation} \label{app:B}

We establish existence of a maximally defined solution $\gamma(t)$ to the integral equation~\eqref{e:intgamma} by employing a contraction-mapping argument. 

\begin{proof}[Proof of Proposition~\ref{p:gamma}]
It follows from standard analytic semigroup theory~\cite{LUN} that the orbit map $\Gamma \in C\big([0,\infty),C_{\mathrm{ub}}^1(\R)\big)$ given by $\Gamma(t) = \smash{\re^{-\partial_\xx^4 t} \gamma_0}$ enjoys the regularity property $\Gamma \in C^j\big((0,\infty),C_{\mathrm{ub}}^l(\R)\big)$ for all $j,l \in \NM_0$ and it obeys the estimates~\eqref{e:Gamma_rates3} for all $t > 0$. In particular, it holds $\|\partial_\xx \Gamma(t)\|_\infty = \smash{\|\re^{-\partial_\xx^4 t} \gamma_0'\|_\infty} \leq \|\gamma_0'\|_\infty < r_0$ for $t \geq 0$. Since the propagator $S_p^0(t)$ vanishes on $[0,1]$, the function $\Gamma(t)$ solves~\eqref{e:intgamma} for $t \in [0,1]$. Hence, we have $\tau_{\max} \geq 1$.

Next, assume $\tau_{\max} \leq T_{\max}$. As $\tau_{\max} \geq 1$, we must have $T_{\max} \geq 1$. Take $t_0 > \frac12$ and $\delta, \eta > 0$ such that $t_0 + \delta < T_{\max}$ and $\eta < r_0$. Let $\check{\gamma} \in C\big([0,t_0],C_{\mathrm{ub}}^1(\R)\big) \cap C\big((0,t_0],C_{\mathrm{ub}}^2(\R)\big) \cap C^1\big((0,t_0],C_{\mathrm{ub}}(\R)\big)$ be a solution to~\eqref{e:intgamma} with $\|\check{\gamma}_\xx(t)\|_{\infty} \leq r_0 - \eta$ for all $t \in [0,t_0]$. Let $R \geq 1$ be such that $\smash{\|(\check{\gamma}(t),\partial_t \check{\gamma}(t))\|_{C_{\mathrm{ub}}^2 \times C_{\mathrm{ub}}} \leq R}$ for all $t \in [\tfrac12,t_0]$. We argue that $\check{\gamma}$ can be extended to a solution $\smash{\gamma_{\mathrm{ext}} \in C\big([0,t_0+\delta],C_{\mathrm{ub}}^1(\R)\big) \cap C\big((0,t_0+\delta],C_{\mathrm{ub}}^2(\R)\big) \cap C^1\big((0,t_0+\delta],C_{\mathrm{ub}}(\R)\big)}$ to~\eqref{e:intgamma} which satisfies $\|\partial_\xx\gamma_{\mathrm{ext}}(t)\|_\infty \leq r_0 - \tfrac12 \eta$ for all $t \in [0,t_0+\delta]$ and $\smash{\|(\gamma_{\mathrm{ext}}(t),\partial_t \gamma_{\mathrm{ext}}(t))\|_{C_{\mathrm{ub}}^2 \times C_{\mathrm{ub}}} \leq 2R}$ for all $t \in [\tfrac12,t_0+\delta]$. To this end, we close a contraction mapping argument in the metric space
\begin{multline*}
\mathcal{M} = \Big\{(\gamma,\gamma_t) \in C\big([t_0,t_0+\delta],C_{\mathrm{ub}}^2(\R) \times C_{\mathrm{ub}}(\R)\big) : \|\gamma_\xx(t)\|_\infty \leq r_0 - \tfrac12 \eta \text{ and }\\ \|(\gamma(t),\gamma_t(t))\|_{C_{\mathrm{ub}}^2 \times C_{\mathrm{ub}}} \leq 2R \text{ for } t \in [t_0,t_0+\delta]\Big\},
\end{multline*}
endowed with the metric from the Banach space $C\big([t_0,t_0+\delta],C_{\mathrm{ub}}^2(\R) \times C_{\mathrm{ub}}(\R)\big)$. 

Applying the mean-value theorem and Proposition~\ref{well_posed_full_sol}, it follows that $V \colon C_{\mathrm{ub}}(\R) \times [t_0,t_0+\delta] \to C_{\mathrm{ub}}(\R)$ given by
\begin{align*}V(\gamma,t)[\xx] &= u(\xx-\gamma(\xx),t) - \phi_0(\xx)\end{align*}
is continuous in $t$ and fulfills
\begin{align*}
\left\|V(\gamma,t) - V(\tilde{\gamma},t)\right\|_\infty \leq \|u_\xx(t)\|_\infty \|\gamma - \tilde{\gamma}\|_\infty
\end{align*}
for $\gamma,\tilde\gamma \in C_{\mathrm{ub}}(\R)$ and $t \in [t_0,t_0+\delta]$. Therefore, setting $\mathcal{K} = \{(\gamma,\gamma_t) \in C_{\mathrm{ub}}^2(\R) \times C_{\mathrm{ub}}(\R) : \|\gamma'\|_\infty \leq r_0-\frac12\eta, \|(\gamma,\gamma_t)\|_{C_{\mathrm{ub}}^2 \times C_{\mathrm{ub}}} \leq 2R\}$ and recalling~\eqref{e:defnonl0}, the nonlinear maps $\mathcal{K} \times [t_0,t_0+\delta] \to C_{\mathrm{ub}}(\R), \, (\gamma,\gamma_t,t) \mapsto \mathcal{Q}(V(\gamma,t),\gamma), \mathcal{R}(V(\gamma,t),\gamma,\gamma_t), \mathcal{S}(V(\gamma,t),\gamma)$ are bounded, continuous in $t$, and Lipschitz continuous in $(\gamma,\gamma_t)$. 

Since the propagator $S_p^0(t)$ vanishes on $[0,1]$, it must hold $\check{\gamma}(t) = \Gamma(t)$ for all $t \in [0,1]$. We show that the action of the nonlinearities on $\Gamma(t)$ is well-defined. Thus, applying the estimates~\eqref{e:Gamma_rates3}, using that $t \mapsto \|u(t)\|_\infty$ is bounded on $[0,1]$, and noting that $\mathcal{R}(v,\gamma,\gamma_t)$ is linear in $\gamma_t$ and $\gamma_{\xx\xx}$, we establish
\begin{align*}
\left\|\mathcal{Q}(V(\Gamma(t),t),\Gamma(t))\right\|_\infty, \left\|\mathcal{S}(V(\Gamma(t),t),\Gamma(t))\right\|_\infty \lesssim 1, \qquad \left\|\mathcal{R}(V(\Gamma(t),t),\Gamma(t),\partial_t \Gamma(t))\right\|_\infty \leq t^{-\frac34}
\end{align*}
for all $t \in (0,1]$. Combining the latter with~\eqref{e:defnonl} and Proposition~\ref{prop:lin1}, we arrive at the estimate
\begin{align*}
\int_0^{\frac12} S_p^0(t-s) \mathcal{N}\left(V(\Gamma(s),s),\Gamma(s),\partial_s \Gamma(s)\right)\de s \lesssim \int_0^{\frac12} s^{-\frac34} \de s \lesssim 1
\end{align*}
for all $t \geq 0$. 

Finally, we observe that, by Proposition~\ref{prop:lin1}, the propagators $\smash{\partial_t^\ell S_p^0(t) \partial_\xx^i} \colon C_{\mathrm{ub}}(\R) \to C_{\mathrm{ub}}^l(\R)$ are $t$-uniformly bounded and strongly continuous on $[0,\infty)$ for any $i,l,\ell \in \mathbb N_0$. A standard contraction mapping argument, cf.~\cite[Theorem~6.1.2]{Pazy}, in the complete metric space $\mathcal{M}$ yields a unique solution $(\gamma,\gamma_t) \in \mathcal{M}$ to the integral system
\begin{align*}
\gamma(t) &= S_p^0(t)\left(v_0 + \phi_0'\gamma_0 + \gamma_0' v_0\right) + \left(1 - \chi(t)\right) \Gamma(t) + \int_0^{\frac12} S_p^0(t-s) \mathcal{N}(V(\Gamma(s),s),\Gamma(s),\partial_s \Gamma(s))\de s\\
&\qquad + \, \int_{\frac12}^{t_0} S_p^0(t-s) \mathcal{N}(V(\check{\gamma}(s),s),\check{\gamma}(s),\partial_s \check{\gamma}(s))\de s\\ 
&\qquad + \, \int_{t_0}^{t_0 + \delta} S_p^0(t-s) \mathcal{N}(V(\gamma(s),s),\gamma(s),\gamma_t(s))\de s,\\
\gamma_t(t) &= \partial_t S_p^0(t)\left(v_0 + \phi_0'\gamma_0 + \gamma_0' v_0\right) - \chi'(t) \Gamma(t) + \left(1 - \chi(t)\right) \partial_t \Gamma(t)\\ 
&\qquad + \, \int_0^{\frac12} \partial_t S_p^0(t-s) \mathcal{N}(V(\Gamma(s),s),\Gamma(s),\partial_s \Gamma(s))\de s\\ 
&\qquad + \, \int_{\frac12}^{t_0} \partial_t S_p^0(t-s) \mathcal{N}(V(\check{\gamma}(s),s),\check{\gamma}(s),\partial_s \check{\gamma}(s))\de s\\ 
&\qquad + \, \int_{t_0}^{t_0 + \delta} \partial_t S_p^0(t-s) \mathcal{N}(V(\gamma(s),s),\gamma(s),\gamma_t(s))\de s,
\end{align*}
provided $\delta > 0$ is sufficiently small. 

By construction, we have $\gamma \in C^1\big([t_0,t_0+\delta],C_{\mathrm{ub}}(\R)\big)$ with $\partial_t \gamma(t) = \gamma_t(t)$ for all $t \in [t_0,t_0+\delta]$. Consequently,
\begin{align*}
\gamma_{\mathrm{ext}}(t) = \begin{cases} \check{\gamma}(t), & t \in [0,t_0), \\ \gamma(t), & t \in [t_0,t_0+\delta],\end{cases}
\end{align*}
defines a solution to~\eqref{e:intgamma}, which extends $\check{\gamma}$. In particular, it satisfies the estimate $\|\partial_\xx\gamma_{\mathrm{ext}}(t)\|_\infty < r_0$ for all $t \in [0,t_0+\delta]$. 

As shown in~\cite[Theorem~4.3.4]{CA98} and~\cite[Theorem~6.1.4]{Pazy}, this extension procedure yields the existence of the desired maximal solution $\gamma(t)$. Its regularity properties then follow by the fact that the propagators $\smash{\partial_t^\ell S_p^0(t) \partial_\xx^i} \colon C_{\mathrm{ub}}(\R) \to C_{\mathrm{ub}}^l(\R)$ are $t$-uniformly bounded and it holds $\Gamma \in C^i\big((0,\infty),C_{\mathrm{ub}}^l(\R)\big)$ for all $i,l,\ell \in \mathbb N_0$.
\end{proof}

\bibliographystyle{abbrv}
\bibliography{mybib}

\def\cprime{$'$}
\begin{thebibliography}{10}

\bibitem{AdR1}
J.~Alexopoulos and B.~de~Rijk.
\newblock Nonlinear stability of periodic wave trains in the
  {F}itz{H}ugh-{N}agumo system against fully nonlocalized perturbations.
\newblock {\em preprint arXiv:2409.17859}, 2024.

\bibitem{BK1}
J.~Bricmont and A.~Kupiainen.
\newblock Renormalization group and the {G}inzburg-{L}andau equation.
\newblock {\em Comm. Math. Phys.}, 150(1):193--208, 1992.

\bibitem{CA98}
T.~Cazenave and A.~Haraux.
\newblock {\em An introduction to semilinear evolution equations}, volume~13 of
  {\em Oxford Lecture Series in Mathematics and its Applications}.
\newblock The Clarendon Press, Oxford University Press, New York, 1998.
\newblock Translated from the 1990 French original by Yvan Martel and revised
  by the authors.

\bibitem{BjoernMod}
B.~de~Rijk.
\newblock Nonlinear stability and asymptotic behavior of periodic wave trains
  in reaction-diffusion systems against {$C_{\rm ub}$}-perturbations.
\newblock {\em Arch. Ration. Mech. Anal.}, 248(3):Paper No. 36, 53, 2024.

\bibitem{DOE19}
A.~Doelman.
\newblock Pattern formation in reaction-diffusion systems---an explicit
  approach.
\newblock In {\em Complexity science}, pages 129--182. World Sci. Publ.,
  Hackensack, NJ, 2019.

\bibitem{DSSS}
A.~Doelman, B.~Sandstede, A.~Scheel, and G.~Schneider.
\newblock The dynamics of modulated wave trains.
\newblock {\em Mem. Amer. Math. Soc.}, 199(934):viii+105, 2009.

\bibitem{FREI}
H.~Freist\"{u}hler and D.~Serre.
\newblock {$L^1$} stability of shock waves in scalar viscous conservation laws.
\newblock {\em Comm. Pure Appl. Math.}, 51(3):291--301, 1998.

\bibitem{GALMI}
T.~Gallay and A.~Mielke.
\newblock Diffusive mixing of stable states in the {G}inzburg-{L}andau
  equation.
\newblock {\em Comm. Math. Phys.}, 199(1):71--97, 1998.

\bibitem{HDRS22}
B.~Hilder, B.~de~Rijk, and G.~Schneider.
\newblock Nonlinear stability of periodic roll solutions in the real
  {G}inzburg-{L}andau equation against {$C_{\rm ub}^m$}-perturbations.
\newblock {\em Comm. Math. Phys.}, 400(1):277--314, 2023.

\bibitem{howard}
P.~Howard.
\newblock Short-time existence theory toward stability for nonlinear parabolic
  systems.
\newblock {\em J. Evol. Equ.}, 15(2):403--456, 2015.

\bibitem{IYSA}
S.~Iyer and B.~Sandstede.
\newblock Mixing in reaction-diffusion systems: large phase offsets.
\newblock {\em Arch. Ration. Mech. Anal.}, 233(1):323--384, 2019.

\bibitem{JONZNL}
M.~A. Johnson, P.~Noble, L.~M. Rodrigues, and K.~Zumbrun.
\newblock Nonlocalized modulation of periodic reaction diffusion waves:
  nonlinear stability.
\newblock {\em Arch. Ration. Mech. Anal.}, 207(2):693--715, 2013.

\bibitem{JONZW}
M.~A. Johnson, P.~Noble, L.~M. Rodrigues, and K.~Zumbrun.
\newblock Nonlocalized modulation of periodic reaction diffusion waves: the
  {W}hitham equation.
\newblock {\em Arch. Ration. Mech. Anal.}, 207(2):669--692, 2013.

\bibitem{JNRZ14}
M.~A. Johnson, P.~Noble, L.~M. Rodrigues, and K.~Zumbrun.
\newblock Behavior of periodic solutions of viscous conservation laws under
  localized and nonlocalized perturbations.
\newblock {\em Invent. Math.}, 197(1):115--213, 2014.

\bibitem{JONZ}
M.~A. Johnson and K.~Zumbrun.
\newblock Nonlinear stability of spatially-periodic traveling-wave solutions of
  systems of reaction-diffusion equations.
\newblock {\em Ann. Inst. H. Poincar\'e Anal. Non Lin\'eaire}, 28(4):471--483,
  2011.

\bibitem{STVenant1}
M.~A. Johnson, K.~Zumbrun, and P.~Noble.
\newblock Nonlinear stability of viscous roll waves.
\newblock {\em SIAM J. Math. Anal.}, 43(2):577--611, 2011.

\bibitem{JUN}
S.~Jung.
\newblock Pointwise asymptotic behavior of modulated periodic
  reaction-diffusion waves.
\newblock {\em J. Differential Equations}, 253(6):1807--1861, 2012.

\bibitem{JUNNL}
S.~Jung and K.~Zumbrun.
\newblock Pointwise nonlinear stability of nonlocalized modulated periodic
  reaction-diffusion waves.
\newblock {\em J. Differential Equations}, 261(7):3941--3963, 2016.

\bibitem{LUN}
A.~Lunardi.
\newblock {\em Analytic semigroups and optimal regularity in parabolic
  problems}.
\newblock Progress in Nonlinear Differential Equations and their Applications,
  16. Birkh\"auser Verlag, Basel, 1995.

\bibitem{melinand2024}
B.~Melinand and L.~M. Rodrigues.
\newblock Phase sinks and sources around two-dimensional periodic-wave
  solutions of reaction-diffusion-advection systems.
\newblock {\em preprint arXiv:2408.14869}, 2024.

\bibitem{Pazy}
A.~Pazy.
\newblock {\em Semigroups of linear operators and applications to partial
  differential equations}, volume~44 of {\em Applied Mathematical Sciences}.
\newblock Springer-Verlag, New York, 1983.

\bibitem{STVenant2}
L.~M. Rodrigues and K.~Zumbrun.
\newblock Periodic-coefficient damping estimates, and stability of
  large-amplitude roll waves in inclined thin film flow.
\newblock {\em SIAM J. Math. Anal.}, 48(1):268--280, 2016.

\bibitem{SAN3}
B.~Sandstede, A.~Scheel, G.~Schneider, and H.~Uecker.
\newblock Diffusive mixing of periodic wave trains in reaction-diffusion
  systems.
\newblock {\em J. Differential Equations}, 252(5):3541--3574, 2012.

\bibitem{ScheelWu}
A.~Scheel and Q.~Wu.
\newblock Diffusive stability of {T}uring patterns via normal forms.
\newblock {\em J. Dynam. Differential Equations}, 27(3-4):1027--1076, 2015.

\bibitem{SCH}
G.~Schneider.
\newblock Nonlinear diffusive stability of spatially periodic
  solutions---abstract theorem and higher space dimensions.
\newblock In {\em Proceedings of the {I}nternational {C}onference on
  {A}symptotics in {N}onlinear {D}iffusive {S}ystems ({S}endai, 1997)},
  volume~8 of {\em Tohoku Math. Publ.}, pages 159--167. Tohoku Univ., Sendai,
  1998.

\bibitem{SU17book}
G.~{Schneider} and H.~{Uecker}.
\newblock {\em {Nonlinear PDEs. A dynamical systems approach}}, volume 182.
\newblock Providence, RI: American Mathematical Society (AMS), 2017.

\bibitem{SUKH}
A.~Sukhtayev, K.~Zumbrun, S.~Jung, and R.~Venkatraman.
\newblock Diffusive stability of spatially periodic solutions of the
  {B}russelator model.
\newblock {\em Comm. Math. Phys.}, 358(1):1--43, 2018.

\bibitem{PLO}
H.~van~der Ploeg and A.~Doelman.
\newblock Stability of spatially periodic pulse patterns in a class of
  singularly perturbed reaction-diffusion equations.
\newblock {\em Indiana Univ. Math. J.}, 54(5):1219--1301, 2005.

\bibitem{vanH}
A.~van Harten.
\newblock Modulated modulation equations.
\newblock In {\em Proceedings of the {IUTAM}/{ISIMM} {S}ymposium on {S}tructure
  and {D}ynamics of {N}onlinear {W}aves in {F}luids ({H}annover, 1994)},
  volume~7 of {\em Adv. Ser. Nonlinear Dynam.}, pages 117--130. World Sci.
  Publ., River Edge, NJ, 1995.

\bibitem{ZUM23}
K.~Zumbrun.
\newblock Forward-modulated damping estimates and nonlocalized stability of
  periodic {L}ugiato-{L}efever waves.
\newblock {\em Ann. Inst. H. Poincar\'{e} C Anal. Non Lin\'{e}aire}, 2023.

\bibitem{ZUH}
K.~Zumbrun and P.~Howard.
\newblock Pointwise semigroup methods and stability of viscous shock waves.
\newblock {\em Indiana Univ. Math. J.}, 47(3):741--871, 1998.

\end{thebibliography}
\end{document}